\documentclass[11pt]{amsart}
\usepackage[utf8]{inputenc}
\usepackage[english]{babel}
\usepackage[letterpaper,top=2cm,bottom=2cm,left=3cm,right=3cm,marginparwidth=2.5cm]{geometry}
\usepackage{amsmath,amsfonts,amssymb,amsthm,xcolor,tikz,todonotes,xspace,euscript,nicefrac,mathrsfs,calligra}
\usepackage{csquotes}
\usepackage{xparse}
\usepackage{enumitem}
\usepackage[maxbibnames=99,backend=biber,url=false,style=alphabetic,doi=true,eprint=false]{biblatex}
\renewbibmacro*{url+urldate}{}
\renewbibmacro{in:}{}
\AtEveryBibitem{
        \clearfield{isbn}
        \clearlist{location}
        \clearfield{month}
        \clearfield{day}
        \clearfield{series}
        \clearfield{pages}
        \clearlist{language}
        \clearlist{translator}
        \clearfield{urldate}
        \clearfield{pubstate}
        \clearfield{url}
        \clearfield{urlyear}
        \clearfield{urlmonth}
}
\usepackage{tikz-cd}
\usepackage[all]{xy}
\usepackage{graphicx}
\usepackage[normalem]{ulem}
\usepackage{bm,stmaryrd}
\usepackage{mathtools}
\usepackage{quiver}

\usepackage{comment}

\usepackage{collectbox}


\makeatletter
\def\@settitle{\begin{center}\baselineskip14\p@\relax
    \bfseries \@title
  \end{center}}
\def\@setauthors{\begingroup
  \def\thanks{\protect\thanks@warning}\trivlist
  \centering\footnotesize \@topsep30\p@\relax
  \advance\@topsep by -\baselineskip
  \item\relax
  \author@andify\authors
  \def\\{\protect\linebreak}{\authors}\ifx\@empty\contribs
  \else
    ,\penalty-3 \space \@setcontribs
    \@closetoccontribs
  \fi
  \endtrivlist
  \endgroup
}
\makeatother
\usepackage[colorlinks=true,allcolors=blue]{hyperref}
\usepackage{zref-clever}
\usepackage{aliascnt}
\zcsetup{cap}
\providecommand{\cref}[1]{\zcref{#1}}
\providecommand{\Cref}[1]{\zcref{#1}}

\zcRefTypeSetup{equation}{
    name-sg=,
    Name-sg=,
    name-pl=,
    Name-pl=
}

\NewDocumentCommand{\NewTheoremWithZref}{ m o m m }{\IfNoValueTF{#2}{\newtheorem{#1}{#3}\zcRefTypeSetup{#1}{name-sg=#3, Name-sg=#3, name-pl=#4, Name-pl=#4}}{\newaliascnt{#1}{#2}\newtheorem{#1}[#1]{#3}\aliascntresetthe{#1}\zcRefTypeSetup{#1}{name-sg=#3, Name-sg=#3, name-pl=#4, Name-pl=#4}}}
\newif\ifBF 
\BFfalse
\newif\ifanindex
\anindextrue
\ifanindex
\usepackage{anindex}
\anindexsetup{
    heading,
    lines,
        title = {Index of Notation},
}

\fi
\newcommand{\Bv}{\mathbf{v}}
\newcommand{\Bw}{\mathbf{w}}
\newtheorem{itheorem}{Theorem}
\zcRefTypeSetup{itheorem}{
    name-sg=theorem,
    Name-sg=Theorem,
    name-pl=theorems,
    Name-pl=Theorems
}

\zcRefTypeSetup{icor}{
    name-sg=corollary,
    Name-sg=Corollary,
    name-pl=corollaries,
    Name-pl=Corollaries
}

\newtheorem{Theorem}{Theorem}[section]
\zcRefTypeSetup{Theorem}{
    name-sg=Theorem,
    Name-sg=Theorem,
    name-pl=Theorems,
    Name-pl=Theorems
}
\NewTheoremWithZref{Proposition}[Theorem]{Proposition}{Propositions}
\NewTheoremWithZref{Lemma}[Theorem]{Lemma}{Lemmas}

\NewTheoremWithZref{Corollary}[Theorem]{Corollary}{Corollaries}
\NewTheoremWithZref{Conjecture}[Theorem]{Conjecture}{Conjectures}

\NewTheoremWithZref{Definition}[Theorem]{Definition}{Definitions}
\NewTheoremWithZref{Example}[Theorem]{Example}{Examples}
\NewTheoremWithZref{Remark}[Theorem]{Remark}{Remarks}

\numberwithin{equation}{section}

\tikzset{wei/.style={draw=red,double=red!40!white,double distance=1.5pt,thin}}
\newcounter{subeqn}

\makeatletter
\@addtoreset{subeqn}{equation}

\newcommand{\nc}{\newcommand}

\newcommand{\mapFromHToG}{\varsigma}
\nc{\bla}{{\boldsymbol{\la}}}
\nc{\mmod}{\operatorname{-mod}}
\nc{\h}{\mathfrak h}
\nc{\K}{\Bbbk}
\newcommand{\LtT}{\mathfrak{\tilde t}}

\nc{\LtTl}{\LtT^{(\lambda)}}
\nc{\KK}{\mathbb{K}}
\nc{\LL}{\mathbb{L}}
\nc{\Nl}{\bfN^{\lambda}}
\nc{\Gl}{G_{\lambda}}

\nc{\tGl}{\tG_{\lambda}}
\nc{\Ne}{\bfN^{\eta}}
\nc{\Ge}{G_{\eta}}
\nc{\g}{\mathfrak g}
\nc{\LT}{\mathfrak t}

\nc{\LTl}{\mathfrak t_{\lambda}}
\nc{\LtG}{\tilde{\mathfrak g}}
\nc{\LG}{\mathfrak g}

\nc{\LGL}{\mathfrak {gl}}
\nc{\fM}{\mathfrak M}

\newcommand{\tM}{\widetilde{\mathcal{M}}}

\newcommand{\elli}{E}
\newcommand{\soa}{\mathsf{A}}
\newcommand{\Gmsc}{\mathbb{G}_m^{\mathrm{sc}}}

\newcommand{\tMB}{\widetilde{\mathcal{M}}_{\mathsf{BFN}}}
\newcommand{\BFN}{\mathsf{BFN}}

\newcommand{\M}{\mathcal{M}}
\nc{\bM}{\mathbf M}
\nc{\bR}{\mathbf R}
\nc{\Bm}{\mathbf m}
\nc{\Bn}{\mathbf n}
\nc{\bS}{\mathbf S}
\nc{\bT}{\mathbf T}
\nc{\bU}{\mathbf U}
\nc{\bV}{\mathbf V}
\nc{\bi}{\mathbf i}
\nc{\bp}{\mathbf p}
\nc{\barQ}{\bar{Q}}
\nc{\barP}{\bar{P}}
\nc{\barX}{\bar{X}}
\nc{\hsigma}{\hat{\sigma}}
\nc{\TL}{\tilde{\mathscr{T}}_{\mathcal{L}}}
\nc{\bs}{\mathbf s}
\nc{\wtlev}{n}
\nc{\C}{\mathbb C}
\nc{\Sym}{\operatorname{Sym}}
\nc{\acham}{\eta}
\nc{\tU}{\mathcal{U}}

\nc{\PolKLR}{\mathsf{Pol}}
\nc{\BY}{\mathbf{Y}}
\nc{\longi}{{\boldsymbol{\ell}}}
\nc{\red}{\operatorname{red}}
\nc{\ind}{\operatorname{ind}}
\nc{\yz}{z}
\nc{\YZ}{Z}
\nc{\Z}{\mathbb Z}
\nc{\R}{\mathbb R}
\nc{\N}{\mathbb N}
\nc{\B}{\mathcal B}
\nc{\cE}{\mathcal E}
\nc{\cF}{\mathcal F}
\nc{\fB}{\mathfrak B}
\nc{\con}{\sim}
\nc{\indices}{\Sigma}
\nc{\pgl}{\mathfrak{pgl}}
\nc{\ev}{\mathsf{ev}}
\nc{\Hom}{\operatorname{Hom}}
\nc{\End}{\operatorname{End}}
\nc{\res}{\operatorname{res}}
\nc{\al}{\alpha}
\nc{\Stein}{\mathbb{X}}
\nc{\pStein}{\mathbb{Y}}
\nc{\vp}{\varphi}
\nc{\Cth}{S_h}
\nc{\cO}{\mathcal{O}}
\nc{\fg}{\mathfrak{g}}
\nc{\one}{\mathbf{1}}
\nc{\bb}{\mathbf{b}}
\nc{\ext}{\operatorname{Ext}}
\nc{\out}{\operatorname{out}}
\nc{\FY}{FY}
\nc{\ep}{\epsilon}
\nc{\bz}{{\mathbf z}}
\nc{\inn}{\operatorname{in}}
\nc{\BK}{{\reflectbox{\rm R}}}
\nc{\Bi}{\mathbf{i}}
\nc{\Ba}{\mathbf{a}}
\nc{\Bj}{\mathbf{j}}
\nc{\Wei}{\EuScript{W}}
\nc{\Bb}{\mathbf{b}}
\nc{\Bnu}{{\boldsymbol{\nu}}}
\nc{\quiver}{I}
\nc{\tGamma}{\tilde{\Gamma}}
\nc{\tGammabR}{\tGamma_{\bR}}
\nc{\GammabR}{\Gamma_{\bR}}
\nc{\diam}{\diamond}
\nc{\la}{\lambda}
\nc{\Yml}{Y_\mu^\lambda}
\nc{\FYml}{FY_\mu^\lambda}
\nc{\bgam}{{\boldsymbol{\gamma}}}
\nc{\blam}{{\boldsymbol{\lambda}}}
\nc{\gr}{\operatorname{gr}}
\nc{\Spec}{\operatorname{Spec}}
\nc{\Stendhal}{Stendhal\xspace}

\nc{\Tsetlin}{\foreignlanguage{russian}{Цетлин}\xspace}
\nc{\GT}{Gelfand-Tsetlin\xspace}
\nc{\GTc}{\mbox{\rm\foreignlanguage{russian}{ГЦ}}}
\nc{\MaxSpec}{\operatorname{MaxSpec}}
\nc{\Cartan}{\C[H_\bullet^{(\bullet)}]}
\nc{\tmetric}{\mathscr{\tilde T}}
\nc{\metric}{\mathscr{T}}
\nc{\pmmetric}{{}_{\pm}\mathscr{T}}
\nc{\pmetric}{{}_{+}\mathscr{T}}
\nc{\mmetric}{{}_{-}\mathscr{T}}
\nc{\Pol}{\mathsf{Pol}}
\nc{\hh}{h}
\nc{\wtmodY}{{Y^\la_\mu\operatorname{-wtmod}}}
\nc{\wtmodFY}{{FY^\la_\mu\operatorname{-wtmod}}}
\nc{\wtmodBK}{{\BK\operatorname{-wtmod}}}
\nc{\fdFY}{{FY^\la_\mu\operatorname{-mod}_{\operatorname{fd}}}}
\nc{\OFY}{{FY^\la_\mu{\text{-}\cO}}}
\nc{\fdY}{{Y^\la_\mu\operatorname{-mod}_{\operatorname{fd}}}}
\nc{\OY}{{Y^\la_\mu{\text{-}\cO}}}

\nc{\yMon}{\mathsf{a}}
\nc{\zMon}{\mathsf{b}}
\nc{\GL}{\operatorname{GL}}
\nc{\PGL}{\operatorname{PGL}}
\nc{\supp}{\operatorname{supp}}
\nc{\calL}{\mathcal{L}}
\nc{\calK}{\mathcal{K}}
\nc{\calF}{\mathcal{F}}
\nc{\mla}{\mathfrak{m}}
\nc{\nla}{\mathfrak{n}}

    \newcommand{\Gm}{\mathbb{G}_m}
    \newcommand{\Ga}{\mathbb{G}_a}

\newcommand{\arxiv}[1]{\href{http://arxiv.org/abs/#1}{\tt arXiv:\nolinkurl{#1}}}

\nc{\Gr}{\mathsf{Gr}}
\nc{\Grlmbar}{\Gr^{\overline{\lambda}}_\mu}
\nc{\excise}[1]{}
\newcommand{\WHilb}[1]{W\hspace{-.2em}\operatorname{-Hilb}(#1)}
\newcommand{\tWHilb}[2]{W\hspace{-.2em}\operatorname{-Hilb}_{#1}(#2)}
\newcommand{\HilbW}[1]{\operatorname{Hilb}^W(#1)}
\newcommand{\tHilb}[2]{\operatorname{Hilb}^W_{#1}(#2)}

\newcommand{\NH}[2]{\operatorname{NH}(#1,#2)}
\newcommand{\simples}{\Sigma}

\newcommand{\ab}{\operatorname{ab}}

\newcommand{\bfN}{\mathbf{N}}

    \newcommand{\tG}{\tilde{G}}
     \newcommand{\tT}{\tilde{T}}
     \newcommand{\Tvee}{T^{\vee}}

\newcommand{\tH}{\tilde{H}}

\newcommand{\SymtW}{R^W}
\newcommand{\bGl}{{\bar{G}_{\lambda}}}
\newcommand{\ints}{\varrho}
\newcommand{\Higgs}{\mathsf{Higgs}}

\newcommand{\mapl}{j_{\la}}
\newcommand{\base}[1]{\mathcal{B}_{#1}}
\newcommand{\Eu}{\operatorname{Eu}}

\newcommand{\toric}[1]{Y(#1)}
\newcommand{\toricbl}[1]{\widetilde{Y}(#1)}
\newcommand{\toricblc}[1]{\overline{Y}(#1)}
\newcommand{\torb}[2]{\mathbb{T}^{#1}(#2)}
\newcommand{\tK}{\tilde{K}}
\newcommand{\Bl}{\operatorname{Bl}}
\newcommand{\Leaf}[1]{\operatorname{Leaf}(#1)}
\newcommand{\CH}{\operatorname{CH}}
\newcommand{\Coulomb}{\mathsf{Coulomb}}

\newcommand{\nonabbl}[1]{\widetilde{\mathcal{Y}}(#1)}
\renewcommand{\sc}{\operatorname{sc}}
\newcommand{\ad}{\operatorname{ad}}
\newcommand{\Ad}{\operatorname{Ad}}
\newcommand{\Aone}{\mathbb{A}^1}

 \usetikzlibrary{decorations.pathmorphing,backgrounds,decorations.markings,babel,patterns,patterns.meta}

\title{On the geometry of Coulomb branches}

\author{Ben Webster}
\address{B.~Webster: Department of Pure Mathematics, University of Waterloo \&
Perimeter Institute for Theoretical Physics, Canada}
\email{ben.webster@uwaterloo.ca}

\begin{document}
\ifanindex
\else 
\newcommand{\notation}[2]{}
\fi
\begin{abstract}
Coulomb branches of 3-dimensional $\mathcal{N}=4$ gauge theories, as defined by Braver\-man--Finkelberg--Nakajima, form a large and interesting class of symplectic singularities.  In this paper, we give a geometric description of these varieties in terms of blowups of toric compactifications, uniformly across the rational, $K$-theoretic, and elliptic settings, which in physics are often associated with 3-, 4-, and 5-dimensional versions of these theories.  We use the resulting local description to classify the symplectic leaves of Coulomb branches and describe their transverse slices in terms of the classification of zero-dimensional leaves of certain smaller Coulomb branches.  

We also discuss related topics, including removing an unnecessary hypothesis from a theorem of Gannon and the author on functoriality for Coulomb branches, and clarifying the relationship with transverse Hilbert schemes by refining the approach of Bielawski and Foscolo.
\end{abstract}

\maketitle

\section{Introduction}
\label{sec:introduction}

Let $G$ be a reductive complex algebraic group, and $\bfN$ a representation of $G$.  
In \cite{BFN}, Braverman, Finkelberg, and Nakajima define a space $\mathcal{R}_{G,\bfN}$ whose equivariant Borel--Moore homology $H^{G}_\bullet(\mathcal{R}_{G,\bfN})$ carries the structure of a commutative algebra under convolution, and they define the Coulomb branch $\M(G,\bfN)$ to be the spectrum of this algebra.  This construction has a natural extension to $K$-theory and elliptic cohomology, and these three cases are often connected in physics with 3-, 4-, and 5-dimensional versions of the corresponding quantum field theories, respectively.  We will refer to them as the rational, $K$-theoretic, and elliptic Coulomb branches.

The study of Coulomb branches of 3-dimensional $\mathcal{N}=4$ supersymmetric gauge theories has attracted considerable interest in recent years, and their $K$-theoretic and elliptic versions have also begun to receive attention.  Our aim in this paper is to shed light on two natural questions about Coulomb branches that have proved surprisingly difficult to answer:
\begin{enumerate}
	\item How can we geometrically describe Coulomb branches in a way that unifies the rational, $K$-theoretic, and elliptic cases?
    \item It is known that Coulomb branches are symplectic singularities, and thus have finitely many symplectic leaves.  What is the indexing set for these leaves?
\end{enumerate}
The study of these leaves in the physics literature, where they are interpreted as phases of the corresponding quantum field theories, has been extensive; for one relevant example, see \cite{bourgetDecayFission2024}.  By contrast, outside the special case of quiver gauge theories, this question has received less attention in the mathematical literature, where important results were obtained in the finite ADE and affine type A cases in \cite{nakajimaCherkisBow2017, muthiahSymplecticLeaves2023}.
In this paper, we reduce question (2) to understanding zero-dimensional leaves of Coulomb branches; this can be done in many important cases, though not in full generality.  Question (1) is too broad to resolve completely in a single paper, but in the course of addressing question (2), we obtain more general results that also shed light on it.
In particular, we provide:
\begin{enumerate}[label=(\alph*)]
    \item A uniform description of rational, $K$-theoretic, and elliptic versions of Coulomb branches via compactifications and blowups, including the Coulomb branch with symmetrizers of Nakajima and Weekes.  Although we do not regard this as a ``geometric description'' in the sense of a moduli-space realization, it does provide an explicit description of the Weil divisors of the Coulomb branch, which determine the variety because it is normal.  More precisely, we give a geometric refinement of the observation from \cite[Th. 5.26]{BFN} that the Coulomb branch can be reconstructed from its behavior over the generic points of the discriminant locus in $\base{G}$, which is effectively a repackaging of GKM theory for the affine Grassmannian.  
    
    For $B=\Ga,\Gm$, we prove that this construction agrees with the usual rational and $K$-theoretic BFN Coulomb branches.  Since there is no consensus definition of the elliptic Coulomb branch in the literature at the moment, the elliptic version should be read as a definition, or proposal, whose compatibility with future constructions remains to be checked.
	\item An extension of the functoriality results of \cite{gannonFunctorialityCoulomb2025} for Coulomb branches to arbitrary homomorphisms, removing the gluability hypothesis required there.
	\item A discussion of the relationship with the results of Bielawski and Foscolo \cite{bielawskiHypertoricVarieties2023}.  They use similar blowups to describe Coulomb branches in terms of transverse Hilbert schemes, but unfortunately their approach does not work for all gauge groups.  We recast these results using nilHecke algebras in a way that makes it easier to see when the map to the transverse Hilbert scheme is an isomorphism and when it is not.
\end{enumerate}
Finally, we turn to our theorems on symplectic leaves and their slices.  Our perspective is that the natural datum indexing a symplectic leaf $L$ of the ordinary Coulomb branch $\M(G,\bfN)$ is the image $U_L$ of its closure under the integrable-system map $\M(G,\bfN)\to \LT/W$.  Roughly speaking, the classification has two stages: first, one identifies the possible images $U_L$, and then one analyzes the smaller Coulomb branch that governs the fiber over a generic point of such an image.
We prove that every such image comes from a {\bf flat} of the weight hyperplane arrangement in $\LT$, that is, an intersection of hyperplanes defined by weights of $\bfN$; see \cref{sec:leaves-and-flats} for details.  If $\lambda\in \LT$ is generic in a flat $U$, then the weights that vanish on $U$ are precisely the weights of the fixed subspace $\Nl$.  Our description therefore depends fundamentally on $\Nl$ viewed as a representation of the centralizer $\Gl=C_G(\lambda)$.  Let $\Gl^0$ be the identity component of the kernel of the action of $\Gl$ on $\Nl$, and let $\Gl^1=\Gl/\Gl^0$.  Then $\Gl^1$ acts almost faithfully on $\Nl$, by which we mean that the kernel is finite.  
We let $\LTl^0$ be the Lie algebra of $\Gl^0\cap T$, the maximal torus of $\Gl^0$.
\notation{$U_L$}{Image in $\LT/W$ of the closure of a symplectic leaf $L$}
\notation{$\Nl$}{The fixed subspace of $\bfN$ under the action of $\lambda$}
\notation{$\Gl, \Gl^0,\Gl^1$}{The centralizer of $\lambda$ in $G$, the identity component of its kernel on $\Nl$, and the quotient acting almost faithfully on $\Nl$}

The leaves lying over a flat are then controlled by the Coulomb branch $\M(\Gl^1, \Nl)$.  The most important case is when the pair $(\Gl^1,\Nl)$ is good, that is, when the natural $\Gm$-action on $\M(\Gl^1, \Nl)$ contracts to a single point, which is a symplectic leaf.
\begin{itheorem}[\cref{leaf-bijection,cor:slice-to-leaf}]
	\label{ith:main}
    If the pair $(\Gl^1,\Nl)$ is good, then there is a unique symplectic leaf whose image is dense in the flat $\LTl^0$.  The transverse slice to this leaf is isomorphic to $\M(\Gl^1, \Nl)$.  On the other hand, if $\M(\Gl^1,\Nl)$ has no zero-dimensional leaves, then there are no symplectic leaves whose image is dense in $\LTl^0$.
\end{itheorem}
The remaining case is when $\M(\Gl^1,\Nl)$ has zero-dimensional leaves but the pair $(\Gl^1,\Nl)$ is not good.  In that case, \cref{leaf-bijection} parametrizes the leaves with image $\LTl^0$ by orbits of a finite group on the zero-dimensional leaves of $\M(\Gl^1,\Nl)$.

Although this picture still needs further development, we give preliminary results relating slices in Coulomb branches to leaves in Higgs branches, making precise some of the proposals in \cite{nakajimaQuestionsProvisional2015}.

The paper is organized as follows.  In \cref{sec:new-construction}, we first give a geometric construction of Coulomb branches, beginning with the abelian case and then passing to the general case.  We then prove its equivalence with the BFN approach for rational and $K$-theoretic Coulomb branches; the elliptic case will have to wait until the literature on elliptic Coulomb branches is more developed.  
We then relate the local structures of the rational, $K$-theoretic, and elliptic cases, and explain how the construction extends to the Coulomb branch with symmetrizers of Nakajima and Weekes, attached to non-symmetric Cartan data \cite{nakajimaCoulombBranches2021}.  

In \cref{sec:functoriality}, we use this construction to prove the extension of the functoriality results of \cite{gannonFunctorialityCoulomb2025}.  In \cref{sec:transverse-hilbert-schemes}, we compare our construction with transverse Hilbert schemes and explain where that comparison breaks down.  Finally, \cref{sec:slices-and-leaves} studies symplectic leaves and their slices and proves \cref{ith:main}.

\section*{Acknowledgements}

Special thanks to Ki Fung Chan, Amihay Hanany, and Alex Weekes for conversations that led to the genesis of this paper.  I also had useful conversations with Tom Gannon, Constantin Teleman, Hiraku Nakajima, Dinakar Muthiah, Gwyn Bellamy, Travis Schedler, Aiden Suter, Lorenzo Foscolo, and Roger Bielawski, who provided feedback on preliminary versions of this paper.  I also received important elliptic insights from Sam DeHority, Davide Gaiotto, Mykola Semenyakin, and Mayuko Yamashita.

This work is supported by Discovery Grant RGPIN-2024-03760 from the Natural Sciences and Engineering Research Council of Canada. 
This research was also supported by Perimeter Institute for Theoretical Physics. Research at Perimeter Institute is supported by the Government of Canada through the Department of Innovation, Science and Economic Development and by the Province of Ontario through the Ministry of Research and Innovation.

\section{A new construction of Coulomb branches}
\label{sec:new-construction}
\subsection{Setup and notation}
\label{subsec:setup-and-notation}
\subsubsection{Setting}

Throughout the paper, we write $\tG$ for a complex reductive group and $G'$ for a normal subgroup such that $\tG/G'$ is a torus.  We also fix a representation of $\tG$ on $\bfN$ for which the action of $G'$ on $\bfN$ factors through a quotient group $G$; the pair $(G,\bfN)$ is the main input to our Coulomb-branch construction.  That is, we have a commutative diagram of groups:
\[\begin{tikzcd}
	{G'} & G \\
	\tG & {GL(\bfN)}
	\arrow[two heads, from=1-1, to=1-2]
	\arrow[hook, from=1-1, to=2-1]
	\arrow[from=1-2, to=2-2]
	\arrow[from=2-1, to=2-2]
\end{tikzcd}\]
    We write $W$ for the common Weyl group of $G$ and $\tG$.
    \begin{Remark}\label{rem:G-double-prime}
        We can make this even more general by allowing $G'$ to not factor through $G$, but only through a quotient $G''$ of $G$, as long as the action of $G$ on $\bfN\oplus \mathfrak{g}$ factors through $G''$.   
    \end{Remark}
    
  The example $\tG=G'=SL(2)$, $G=PGL(2)$, with $\bfN$ an arbitrary representation, will play an important role in our proof of the comparison with the BFN construction; see \cref{ex:PGL2-running,ex:SL2-running,ex:SL2-PGL2-cover}.

   Unlike in \cite{gannonFunctorialityCoulomb2025}, we allow $G'\to G$ to be a non-trivial finite cover, since this distinction is significant in the $K$-theoretic and elliptic cases.  For example, in the work of Cautis--Williams \cite[Th. 1.8]{cautisClusterTheory2019}, the case where $G$ is adjoint and $\tG$ is its simply connected cover makes a notable appearance.

\subsubsection{Cohomological context} In this section, we give a general geometric construction of Coulomb branches for the pair $(G,\bfN)$ with flavor deformation by $\tG$ in the rational (BFN), $K$-theoretic, and elliptic cases.  This description is implicit in \cite{BFN}, but not in the generality or the explicit form that we need.  Let $\K$ be an algebraically closed field of characteristic coprime to the order of the Weyl group $W$; we will work over $\K$ throughout.  

Fix a one-dimensional connected commutative algebraic group $B/\K$, namely $B=\mathbb{G}_a=\Spec\K[x]$, $\mathbb{G}_m=\Spec\K[z,z^{-1}]$, or an elliptic curve $\elli$.  For a torus $T$ with cocharacter lattice $X_*(T)$, we set
\[\base{T}=B\otimes_{\Z} X_*(T),\]
so $\base{T}$ is the Lie algebra $\mathfrak t$ in the rational case ($B=\mathbb{G}_a$), the algebraic torus $T$ itself in the $K$-theoretic case ($B=\mathbb{G}_m$), and a product of elliptic curves in the elliptic case ($B=\elli$).  For general $G$, we set $\base{G}=\base{T}/W$, where $W$ acts naturally on $\base{T}$.\notation{$\base{G}$}{Base attached to $G$ in the chosen cohomology theory}

The formal group laws attached to $B$ have natural associated cohomology theories:
\begin{enumerate}
    \item In the rational case ($B=\mathbb{G}_a$), we have ordinary cohomology $H^*_{T}(\mathrm{pt})\cong \K[\LT]$.  For $G=SL(2)$ or $ PGL(2)$, we have $\base{G}=\Aone$, but with the coordinate given by the square of the usual coordinate on $\LT$.  
    \item In the $K$-theoretic case ($B=\mathbb{G}_m$), we have $K$-theory $K^*_T(\mathrm{pt})\cong \K[T]$.  Thus, now $\base{SL(2)}$ is $\Aone$ with coordinate given by $y=e^{\alpha/2}+e^{-\alpha/2}$ of the maximal torus, while $\base{PGL(2)}$ is $\Aone$ with coordinate given by $y^2-2=e^{\alpha}+e^{-\alpha}$.
    \item In the elliptic case ($B=\elli$), we have elliptic cohomology $Ell_T(\mathrm{pt})\cong \mathcal{O}_{\base{T}}$.  Note that while usual cohomology and $K$-theory give usual algebras, elliptic cohomology gives a sheaf of algebras on the abelian variety $\base{T}$.  For $G=SL(2), PGL(2)$, we have $\base{G}=\elli/\{\pm 1\}$, the quotient of the elliptic curve by the involution given by negation, which is a rational curve.
\end{enumerate}
In each of these cases, $\base{G}$ is the spectrum of the respective cohomology theory applied to the classifying space $BG$, so the corresponding equivariant theory applied to any space is naturally a sheaf over this base.

In \cite{BFN}, Braverman, Finkelberg, and Nakajima define a space $\mathcal{R}_{G,\bfN}$, which is a cofinite-dimensional subspace of the associated bundle of $\bfN\llbracket t \rrbracket$ over the affine Grassmannian $\Gr_G=G(\!(t)\!)/G\llbracket t \rrbracket$.  This variety carries an action of the group $\tG$: for any reductive group, the action on its affine Grassmannian factors through the adjoint group $G_{\operatorname{ad}}$, and $\tG_{\operatorname{ad}}=G_{\operatorname{ad}}$.  Thus we can let $\tG$ act on $\mathcal{R}_{G,\bfN}\subset \Gr_G\times \bfN \llbracket t \rrbracket$ by the diagonal action.  
The flavor-deformed Coulomb branch $\tMB(\tG, G,\bfN)$ is then the spectrum of the equivariant Borel--Moore homology $H_*^{\tG}(\mathcal{R}_{G,\bfN})$, with its natural convolution product.  This construction has obvious analogues in the $K$-theoretic and elliptic cases, where we replace Borel--Moore homology with the appropriate cohomology theory; we write $\tMB(\tG,G,\bfN)$ for the resulting varieties in these cases as well, leaving $B$ implicit.\notation{$\tMB(\tG,G,\bfN)$}{BFN Coulomb branch, or its $K$-theoretic or elliptic analogue, with flavor deformation by $\tG$} It is manifest from the definition that the BFN Coulomb branch $\tMB(\tG,G,\bfN)$ is a variety over $\base{\tG}$, and this is the framework that we will use to understand its geometry.

\subsubsection{Euler classes}
For any weight $\lambda\in \LT_{\Z}$, we write $e^{\lambda}$ for the corresponding character of $T$, and similarly for coweights as characters on $\Tvee$.  
For each weight $\varphi$, we define a section $\Eu(\varphi)$ of a line bundle on $\base{\tT}$; we call $\Eu(\varphi)$ the {\bf Euler class} of $\varphi$.\notation{$\Eu(\varphi)$}{Euler-class section attached to the weight $\varphi$}\begin{enumerate}
    \item In the rational case ($B=\mathbb{G}_a$), we take $\Eu(\varphi)$ to be the linear function on $\base{\tT}$ given by the composition of the map $\base{\tT}\to \mathfrak t$ with the weight $\varphi\colon \mathfrak t\to \mathbb{G}_a$.  
    \item In the $K$-theoretic case ($B=\mathbb{G}_m$), we take $\Eu(\varphi)$ to be the regular function on $\base{\tT}$ given by $1-e^{-\varphi}$.  
    \item In the elliptic case ($B=\elli$), we take $\Eu(\varphi)$ to be the pullback of the theta section $\vartheta$ along the map $\base{\tT}\to \elli$ induced by $\varphi$; this is a section of the line bundle $\mathcal{L}_{\varphi}$ which is the pullback of the line bundle on $\elli$ given by the divisor of the origin.
\end{enumerate}
We write $D_{\varphi}$ for the divisor of zeros of $\Eu(\varphi)$.\notation{$D_{\varphi}$}{Divisor of zeros of the Euler class $\Eu(\varphi)$}

The non-zero weights of $\bfN$ for the torus $\tT$ lie on finitely many distinct lines $\ell_1,\dots, \ell_k$ in the rational character lattice $X^*(\tT)\otimes\mathbb{Q}$.  For each line $\ell_i$, we choose a primitive weight $\eta_i$ generating the intersection of the weight lattice with this line.  Thus, if $\varphi_{i1},\dots, \varphi_{i r_i}$ are the distinct non-zero weights of $\bfN$ lying on this line, we can write
\[\varphi_{ij}=m_{ij}\eta_i\]
with $m_{ij}\in \mathbb{Z}$.
\notation{$\ell_i,\eta_i$}{Line spanned by a set of weights of $\bfN$ and its primitive generator}

We can then define a divisor $D_{\ell_i}=\sum_{j=1}^{r_i }|m_{ij}|D_{\varphi_{ij}}$.  The choice of sign for $\eta_i$ determines a decomposition by separating the weights on $\ell_i$ according to the sign of $m_{ij}$:
\begin{align}
    \label{eq:sigma-choice}
D^+_{\ell_i}&=\sum_{m_{ij}>0} m_{ij}D_{\varphi_{ij}}& D^-_{\ell_i}&=\sum_{m_{ij}<0} -m_{ij}D_{\varphi_{ij}}.
\end{align}
This yields a factorization $D_{\ell_i}=D^+_{\ell_i}+D^-_{\ell_i}$.  
For an irreducible divisor $D$ in $\base{\tT}$, we write $k_{i,D}$ and $k^{\pm}_{i,D}$ for the coefficients of $D$ in the divisors $D_{\ell_i}$ and $D^{\pm}_{\ell_i}$, respectively, so $k_{i,D}=k^+_{i,D}+k^-_{i,D}$.\notation{$D_{\ell_i},D^{\pm}_{\ell_i}$}{Divisor contributed by the weights on the line $\ell_i$}\notation{$k_{i,D},k^{\pm}_{i,D}$}{Multiplicities of $D$ in $D_{\ell_i},D^{\pm}_{\ell_i}$}

The primitive generator $\eta_i$ of the line $\ell_i$ is only determined up to sign, so we will sometimes want to use the symmetry of replacing $\eta_i$ by $-\eta_i$.  This simply exchanges the divisors $D^+_{\ell_i}$ and $D^-_{\ell_i}$, yielding otherwise equivalent data.

We will use the ambient ring $\K(\base{\tT})[\Tvee]$ throughout.  Equivalently, these are the rational functions on $\base{\tT}\times \Tvee$ whose only possible poles lie over divisors in $\base{\tT}$.

\subsection{Abelian case}
\label{subsec:abelian-case}

\subsubsection{Construction}
Consider the case when $G=T$ is a torus.  We define a sheaf $\soa_{T, \bfN}$ as a subalgebra of the constant sheaf with value $\K(\base{\tT})[\Tvee]$ by imposing regularity conditions along the preimages of irreducible divisors in $\base{\tT}$.  We leave $\tT$ implicit in the notation, since it is reflected by the underlying base space $\base{\tT}$.  

We can state this definition in a geometric way as follows.  The characters $\eta_i$ can be interpreted as cocharacters of $\Tvee$, and so they give an action of $\Gm$ on $\base{\tT}\times \Tvee$ by translation on the fiber.  We use them to define a toric partial compactification $\toric{\tT,\bfN}$ of $\base{\tT}\times \Tvee$ by adding divisors $D^{i,0}$ and $D^{i,\infty}$ corresponding to the limits of the orbits $\eta_i(t)\cdot x$ as $t\to 0,\infty$.  For a single line, since $\eta$ is primitive, we can split $T\cong \Gm\times T'$ with $\eta$ identified with the projection to the first factor.  Then $\toric{\tT,\bfN}$ is the product $\base{T}\times \mathbb{P}^1\times (T')^{\vee}$, and $D^{i,0}$ and $D^{i,\infty}$ are the divisors lying over $0$ and $\infty$ in the $\mathbb{P}^1$-factor.
\notation{$\toric{\tT,\bfN}$}{Toric partial compactification of $\base{\tT}\times \Tvee$}

To describe this more precisely, consider the multigraded $\mathcal O_{\base{\tT}}$-algebra
\[
\mathcal R=\bigoplus_{\mathbf n\in \mathbb N^k}\mathcal R_{\mathbf n},\qquad
\mathcal R_{\mathbf n}=\bigoplus_{\substack{\lambda\in  X_*(T)  \\ |\langle \lambda,\eta_i\rangle|\le n_i\ \forall i}}
\mathcal O_{\base{\tT}}\,e^{\lambda}\,\mathbf u^{\mathbf n}.
\]
Then $\toric{\tT,\bfN}$ is the relative multiProj of 
$\mathcal R$ over the base $\base{\tT}$.

Thus, we can consider the interaction of the divisors $D^{i,0}$ and $D^{i,\infty}$ with the divisors $D_{\ell_i}^{\pm}$, thought of as subschemes.  
Let $X_i$ be the not necessarily reduced subscheme given by the union of $D^{i,0}\cap D_{\ell_i}^{-}$ and $D^{i,\infty}\cap D_{\ell_i}^{+}$.

\begin{Definition}
Let $\toricbl{\tT,\bfN}$ be the blowup of $\toric{\tT,\bfN}$ along the union of the $X_i$, with the strict transform of the divisors $D^{i,0}$ and $D^{i,\infty}$ removed. 
\notation{$\toricbl{\tT,\bfN}$}{Blowup model used to construct the abelian Coulomb branch}The Coulomb branch $\tM=\tM(\tT,T,\bfN)$ is the affinization of $\toricbl{\tT,\bfN}$ over $\base{\tT}$.  
\end{Definition}

\begin{Example}
    Consider the case where $\tT=T=\Gm$ and $\bfN$ has weights $m_1,\dots, m_r$.  In this case, $\base{\tT}=\Aone$ with coordinate $\tau$ and $\toric{\tT,\bfN}=\Aone\times \mathbb{P}^1$.  If $z$ is the dual coordinate on $\Tvee$, then the divisors $D^+_{\ell}$ and $D^-_{\ell}$ are given by $\tau^{m^+}=0$ and $\tau^{m^-}=0$, where $m^+$ and $m^-$ are the sums of the positive and negative weights, respectively.  The divisors $D^{0}$ and $D^{\infty}$ are given by $z=0$ and $z=\infty$, respectively.  Thus, the subschemes $D^{i,0}\cap D_{\ell_i}^{-}$ and $D^{i,\infty}\cap D_{\ell_i}^{+}$ are defined by the ideals $(z,\tau^{m^-})$ and $(z^{-1},\tau^{m^+})$ on natural charts of $\Aone\times \mathbb{P}^1$.  We obtain regular functions on the blowup minus the strict transforms by taking the ratios $x=\tau^{m^-}z^{-1}$ and $y=\tau^{m^+}z$, which are regular on $\toricbl{\tT,\bfN}$.  Together with $\tau$, these generate the regular functions on $\toricbl{\tT,\bfN}$ and the resulting affine variety $\tM$ is defined by the equation $xy=\tau^{m^++m^-}$.
\end{Example}
\begin{Example}\label{ex:generic-fiber}
    Note that this example also describes the behavior over the generic point of any divisor $D$ in $\base{\tT}$ in higher-rank cases.  Unless $D$ is a component of some $D_{\ell_i}$, the fiber over that point is $\Tvee$ after base change to the function field $\K(D)$.  If $D$ is a component of some $D_{\ell_i}$, then $\toric{\tT,\bfN}$ compactifies one factor of $\Tvee$ to $\mathbb{P}^1$, and the blowup $\toricbl{\tT,\bfN}$ then blows up $0$ and $\infty$ in this $\mathbb{P}^1$, producing exceptional divisors.  Thus, after base change to the local ring at the generic point of $D$, we obtain the variety defined by the equation $xy=\tau^{k_{i,D}}$, where $\tau$ is a local coordinate on $\base{\tT}$ at that point, times a torus.
\end{Example}

\begin{Remark}
    It can sometimes be helpful to think about blowing up the non-transverse intersections of the divisors $D_{\ell_i}^{\pm}$ with the exceptional locus of the blowup.  By iterating this process, we can obtain over $D$ a chain of divisors where the fiber over each point is a chain of $\mathbb{P}^1$'s times a torus;  the number of resulting $\mathbb{P}^1$'s is $k_{i,D}-1$. 
    
    We could also construct this chain by blowing up the reduced intersection of the preimage of the divisor $D$ with $D^{i,0}$ (resp. $D^{i,\infty}$), then the intersection of the exceptional locus with $D^{i,0}$, and so on.  The number of blowups needed is $k^{\pm}_{i,D}-1$.  Let $\toricblc{\tT,\bfN}$ be this iterated blowup; it also has affinization $\tM$.\notation{$\toricblc{\tT,\bfN}$}{Iterated blowup resolving the chains over the divisors $D$}\end{Remark}
More explicitly, for each irreducible divisor $D$ in $\base{\tT}$, we can define two discrete valuations $\nu_{D,\pm}$ on the function field $\K(\base{\tT}\times \Tvee)$.\notation{$\nu_{D,\pm}$}{Valuations imposing the abelian regularity conditions} Implicitly, these valuations also depend on the choice of $i$, but for each divisor $D$, there is at most one $i$ such that $k^{\pm}_{i,D}\neq 0$, so we can leave the choice of $i$ implicit.  If $k^{\pm}_{i,D}=0$, then we take $\nu_{D,\pm}$ to be the pullback of the usual valuation $\nu_D$ along the blowup map.  
Note that we have $\nu_{D,+}=\nu_{D,-}=\nu_D$ for all but finitely many $D$.  Otherwise:
\begin{enumerate}
    \item If $k^+_{i,D}\neq 0$, let $\nu_{D,+}$ be the order of vanishing along the component of the exceptional locus of the blowup $\toricbl{\tT,\bfN}\to \toric{\tT,\bfN}$ which lies over $D\cap D^{i,\infty}$ .
    \item If $k^-_{i,D}\neq 0$, let $\nu_{D,-}$ be the order of vanishing along the component of the exceptional locus which lies over $D\cap D^{i,0}$.
\end{enumerate}
In the notation of \cref{ex:generic-fiber}, the components corresponding to $\nu_{D,+}$ and $\nu_{D,-}$ are given by $x=0$ and $y=0$.
\begin{Definition}
    Let $\soa_{T,\bfN}$ be the sheaf whose sections $\soa_{T,\bfN}(U)$ on an open set $U\subset \base{\tT}$ form the subring of $\K(\base{\tT})[\Tvee]$ consisting of functions $f$ such that $\nu_{D,+}(f)\geq 0$ and $\nu_{D,-}(f)\geq 0$ for all divisors $D$.  We can equivalently describe $\tM=\tM(\tT,T,\bfN)$ as $\Spec \soa_{T,\bfN}$.
    \notation{$\soa_{T,\bfN}$}{The sheaf giving the functions on the abelian Coulomb branch}\end{Definition}

Note that replacing $\eta_i$ by $-\eta_i$ leaves $\toric{\tT,\bfN}$ unchanged but interchanges the boundary divisors $D^{i,0}$ and $D^{i,\infty}$.  The divisors $D^+_{\ell_i}$ and $D^-_{\ell_i}$ are exchanged as well, so the valuations $\nu_{D,+}$ and $\nu_{D,-}$ are exchanged.  Thus, the resulting variety and the resulting subring are the same.

\subsubsection{Gluing}
We have subrings $\soa_{T,\bfN}^{\pm}$ given by imposing only the regularity conditions corresponding to one of the signs: that is, $\nu_{D,\pm}(f)\geq 0$ for all $D$. Let $\torb{\pm}{\tT,\bfN}$ be the relative spectrum of $\soa_{T,\bfN}^{\pm}$;  the embeddings in $\K(\base{\tT})[\Tvee]$ give birational maps $\torb{\pm}{\tT,\bfN}\dashrightarrow \base{\tT}\times \Tvee$.
\begin{Lemma}\label{lem:gluing}
    As a $\Tvee$-space, each variety $\torb{\pm}{\tT,\bfN}$ is a principal $\Tvee$-bundle over $\base{\tT}$.  The space $\tM$ is the relative affinization of the variety obtained by gluing $\torb{+}{\tT,\bfN}$ and $\torb{-}{\tT,\bfN}$ together along the open subset where the birational isomorphism to $\base{\tT}\times \Tvee$ is regular.
\end{Lemma}
To trivialize the $\Tvee$-bundle $\torb{\pm}{\tT,\bfN}$, the divisors $D_{\ell_i}^{\pm}$ must be principal, so this is not always possible.  If $\operatorname{div}(f_i^{\pm})=D_{\ell_i}^{\pm}$, then the birational automorphism $\prod_i \eta_i^{\pm 1}(f_i^{\pm})$ of $\base{\tT}\times \Tvee$ induces an isomorphism between $\torb{\pm}{\tT,\bfN}$ and $\base{\tT}\times \Tvee$.  Thus, if $B=\Ga$ or $\Gm$, such a trivialization exists, but if $B=\elli$, it can exist only when $D_{\ell_i}^{\pm}$ are trivial, since $\base{\tT}$ has no non-trivial effective principal divisors. 
\begin{proof}
Since $\base{\tT}$ is smooth, all divisors on it are locally principal; thus every point in $\base{\tT}$ has an affine neighborhood $U$ such that the divisors $D_{\ell_i}^{\pm}=\operatorname{div}(f_i^{\pm})$ are principal on $U$ for some $f_i^{\pm}\in \mathcal{O}_{\base{\tT}}(U)$.  The birational automorphism $\prod_i \eta_i^{\pm 1}(f_i^{\pm})$ as above gives an isomorphism between $\torb{\pm}{\tT,\bfN}|_U$ and $U\times \Tvee$, so $\torb{\pm}{\tT,\bfN}$ is a principal $\Tvee$-bundle over $\base{\tT}$.  

The statement about relative affinization is just a restatement of the definition of $\soa_{T,\bfN}$: the regularity conditions defining $\soa_{T,\bfN}$ show that $\soa_{T,\bfN}=\soa_{T,\bfN}^+\cap \soa_{T,\bfN}^-$, and the right-hand side is the ring of functions on the variety obtained by gluing $\torb{+}{\tT,\bfN}$ and $\torb{-}{\tT,\bfN}$ together along the open subset where they are both regular.
\end{proof}

\begin{Remark}
    This construction is very similar to the one given by Teleman in \cite[Th. 1 \& Th. 2]{telemanRoleCoulomb2021}, but slightly more general, since it works without using massive deformations.  That is, Teleman always chooses $\tT = T \times \Gm$ with $\Gm$ acting by scaling, so that each weight $\varphi_i$ of $T$ on $\bfN$ corresponds to its own line $\ell_i$ of weights with $\eta_i=(\varphi_i,1)$.  Thus, $D^-_i$ is trivial for all $i$.  In this case, the variety $\torb{-}{\tT,\bfN}$ is just $\base{\tT}\times \Tvee$ with the birational isomorphism defined in the proof of \cref{lem:gluing} trivial; since \cite{telemanRoleCoulomb2021} only considers the case $B=\Ga,\Gm$, the bundle $\torb{+}{\tT,\bfN}$ is also trivial, so we can think of composing with this trivialization to get a birational automorphism of $\base{\tT}\times \Tvee$, which is the gluing used in \cite[Th. 1 \& Th. 2]{telemanRoleCoulomb2021}.
\end{Remark}

\subsubsection{The case of $\Gm$}
\label{ex:Gm}
    Let $\tT=\Gm$ and let $\bfN$ be a representation with weights $m_1,\dots, m_k$, at least one of which is non-zero.  In this case, there is only one line $\ell$, and we can take $\eta=1$ as the primitive generator.  We obtain divisors $D^+_{\ell}$ and $D^-_{\ell}$ in $\base{\tT}$.
\begin{enumerate}
    \item If $B=\base{\tT}=\Ga$, then $D^+_{\ell}$ and $D^-_{\ell}$ are both supported at the origin, with multiplicities $m^+=\sum_{i=1}^k \max(m_i,0)$ and $m^-=\sum_{i=1}^k -\min(m_i,0)$, respectively. 
    \item If $B=\base{\tT}=\Gm$, then $D^+_{\ell}$ and $D^-_{\ell}$ are supported at roots of unity, with the weight $m_i$ contributing multiplicity $|m_i|$ to the divisor supported at the $m_i$-th roots of unity in $D^+_{\ell}$ or $D^-_{\ell}$, depending on the sign of $m_i$.
    \item If $B=\base{\tT}=\elli$, then $D^+_{\ell}$ and $D^-_{\ell}$ are similar, with $m_i$-torsion points replacing roots of unity.
\end{enumerate}
Note here that the trigonometric and elliptic cases are not fundamentally more complicated in terms of local behavior, but require more careful bookkeeping, since the divisors appearing may have multiple components.

As discussed above, in this case, we have $\toric{\tT,\bfN}\cong B\times \mathbb{P}^1$, and $D^{0}$ and $D^{\infty}$ are the divisors given by the fibers over $0,\infty\in \mathbb{P}^1$, respectively.  
The blowup $\toricbl{\tT,\bfN}$ produces a new exceptional $\mathbb{P}^1$ at each point where $D^+_{\ell}$ meets $D^{\infty}$ and where $D^-_{\ell}$ meets $D^0$.  To construct the Coulomb branch, we remove the strict transforms of $D^0$ and $D^{\infty}$, so we are left with fibers that look like:
\begin{enumerate}
    \item $\Gm$ if $D^+_{\ell}$ and $D^-_{\ell}$ do not meet $D^{\infty}$ and $D^0$, respectively.  This is already affine, so it is unchanged by affinization.
    \item $\mathbb{A}^1\cup\mathbb{A}^1=\Spec \K[x,y]/(xy)$ if only one of these intersections is non-empty. This is similarly not affected by affinization.
    \item A $\mathbb{P}^1$ with $\mathbb{A}^1$'s attached at $0$ and $\infty$ if both intersections are non-empty.  In this case, the affinization contracts the $\mathbb{P}^1$ to a point, so we are left with $\mathbb{A}^1\cup\mathbb{A}^1$.  
\end{enumerate}
In the cases where we obtain $\mathbb{A}^1\cup\mathbb{A}^1$, the valuations $\nu_{D,\pm}$ are given by the order of vanishing on the two components.  

If we instead pass to $\toricblc{\tT,\bfN}$, we get a chain of $\mathbb{P}^1$'s at each such point, and removing the strict transforms of $D^0$ and $D^{\infty}$ removes one point from each end of the chain.  Thus, affinization contracts the entire chain to $\mathbb{A}^1\cup\mathbb{A}^1$.

Thus, to sum up, in the Coulomb branch, the fiber over a point is $\mathbb{A}^1\cup\mathbb{A}^1$ if the point is in the support of $D_{\ell}$, and $\Gm$ otherwise.  In the first case, the two components of the fiber correspond to the two valuations $\nu_{D,\pm}$, and in the second case, there is only one valuation $\nu_D$.

At first glance, this calculation seems independent of $k_{i,D}$.  However, nearby fibers show that $k_{i,D}$ does matter: as suggested by the chain of $\mathbb{P}^1$'s in $\toricblc{\tT,\bfN}$ that do not intersect the strict transforms of $D^0$ and $D^{\infty}$, the Coulomb branch is \'etale locally modeled on the $A_{k_{i,D}-1}$ surface singularity, and $\toricblc{\tT,\bfN}$ minus $D^0$ and $D^{\infty}$ gives its minimal resolution.  
In the rational case, we obtain this $A_{k_{i,D}-1}$ singularity on the nose, while in the trigonometric and elliptic cases, we obtain a family of such singularities lying over different finite-order points in $\base{\tT}\cong B$, with the rank determined by the sum of the absolute values of the weights of $\bfN$ on the fixed points of the corresponding element when $B=\Gm$.  The elliptic analogue of taking fixed points is to consider the associated bundle obtained from the line bundle, equivalently the principal $\Gm$-bundle, on $\base{\tT}$ given by the divisor $(b)-(0)$, and then take its trivial summand, that is, the subsheaf generated by global sections.

\subsubsection{Weight spaces and valuations}

We can describe the functions on the Coulomb branch more explicitly using the fact that $\toricbl{\tT,\bfN}$ still carries a residual action of $\Tvee$, and hence so does $\tM$.  Since every function is a finite sum of $\Tvee$-weight vectors, it is enough to understand the regularity conditions on each weight space separately.  The $\Tvee$-weights are given by characters $\gamma\in 
X^*(\Tvee)=X_*(T)$, which we may regard as regular functions $e^{\gamma}\colon \Tvee\to \mathbb{G}_m$.  The $\gamma$-weight subspace of $\K(\base{\tT})[\Tvee]$ consists of functions of the form $fe^{\gamma}$, where $f$ is a rational function on $\base{\tT}$.
\begin{Lemma}\label{lem:regularity-coefficients}
    A function of the form $fe^{\gamma}$ lies in $\soa_{T,\bfN}(U)$ if and only if for every divisor $D$ in $U\subset \base{\tT}$, we have 
    \begin{align}\label{eq:valuation-formulas}
        \nu_{D,+}(fe^{\gamma}) &= \nu_D(f)-\langle \gamma,\eta_i\rangle k^+_{i,D} \geq 0,\\
        \nu_{D,-}(fe^{\gamma}) &= \nu_D(f)+\langle \gamma,\eta_i\rangle k^-_{i,D} \geq 0.
    \end{align}
\end{Lemma}
\begin{proof}
The poles and zeros of $fe^{\gamma}$ are:
\begin{enumerate}
    \item Along the preimage of a divisor $D$ in $\base{\tT}$, the order of vanishing is given by $\nu_D(f)$.
    \item Along the divisors $D^{i,0}$ and $D^{i,\infty}$, the order of vanishing is given by $\langle \gamma,\eta_i\rangle$ and $-\langle \gamma,\eta_i\rangle$ respectively.
\end{enumerate}
\'Etale locally, the blowup $\toricbl{\tT,\bfN}$ looks like the blowup of $\mathbb{A}^{2n}$ at the ideal $(x_1,x_2^m)$ for some $m$.  The vanishing order of $x_1^ax_2^b\cdots $ along the exceptional divisor is given by $ma+b$.  Translated back to $\toricbl{\tT,\bfN}$, this gives the formulas in \cref{eq:valuation-formulas}.
\end{proof}

It is helpful to picture these regularity conditions by looking at the vanishing order of $e^{\gamma}$ and $f$ along the components over $D$ in $\toricbl{\tT,\bfN}$.  The function $f$ is regular on $D^{i,0}$ and $D^{i,\infty}$, and has the same vanishing order $\nu_D(f)$ on each of the components over $D$.  On the other hand, the function $e^{\gamma}$ is regular on the preimage of $D$, but has order $\pm \langle \gamma,\eta_i\rangle$ on $D^{i,0}$ and $D^{i,\infty}$, as discussed above.  Thus, its order of vanishing on the components over $D$ forms an arithmetic progression \[-\langle \gamma,\eta_i\rangle k^+_{i,D}, -\langle \gamma,\eta_i\rangle (k^+_{i,D}-1),\dots, \langle \gamma,\eta_i\rangle k^-_{i,D}.\]  Of course, if $fe^{\gamma}$ is regular on the two outermost components, then it is regular on all the intermediate ones as well.  
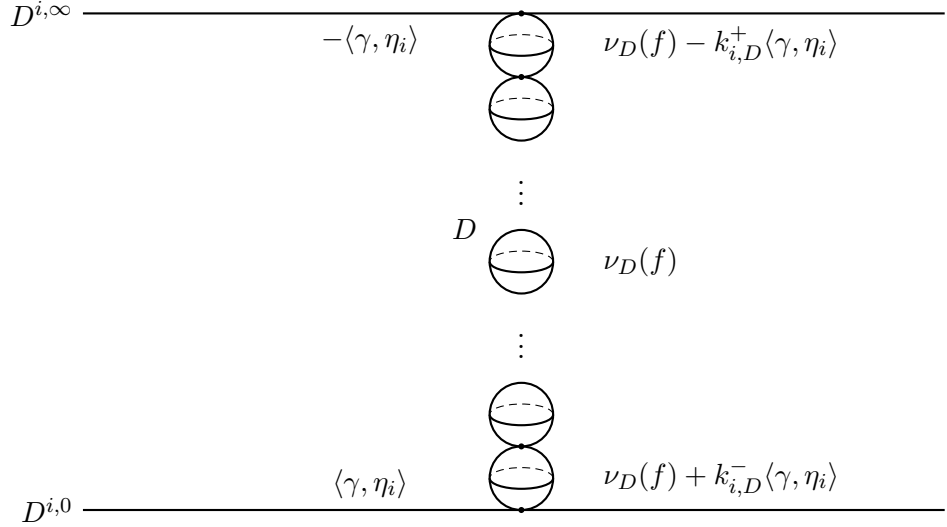
\begin{figure}[ht]
\centering
\begin{tikzpicture}[x=1cm,y=1cm,baseline=(current bounding box.center)]
\draw[thick] (0,6.8) -- (11.4,6.8);
    \draw[thick] (0,0.24) -- (11.4,0.24);
  \node[left] at (0,6.8) {$D^{i,\infty}$};
    \node[left] at (0,0.24) {$D^{i,0}$};
    \node at (3.8,6.45) {$-\langle \gamma,\eta_i\rangle$};
    \node at (3.8,0.59) {$\langle \gamma,\eta_i\rangle$};

\node[left] at (5.35,3.95) {$D$};

\newcommand{\picsphere}[2]{\draw[thick] (#1,#2) circle (0.42);
    \draw[thick] (#1-0.42,#2) arc[start angle=180,end angle=360,x radius=0.42,y radius=0.14];
    \draw[densely dashed] (#1+0.42,#2) arc[start angle=0,end angle=180,x radius=0.42,y radius=0.14];
  }

\picsphere{5.8}{6.38}
  \picsphere{5.8}{5.54}
        \node at (5.8,4.53) {$\vdots$};
        \picsphere{5.8}{3.52}
        \node at (5.8,2.51) {$\vdots$};
        \picsphere{5.8}{1.50}
        \picsphere{5.8}{0.66}

\fill (5.8,6.80) circle (0.04);
  \fill (5.8,5.96) circle (0.04);
        \fill (5.8,1.08) circle (0.04);
        \fill (5.8,0.24) circle (0.04);

\node[right] at (6.75,6.38) {$\nu_D(f)-k^+_{i,D}\langle \gamma,\eta_i\rangle$};
        \node[right] at (6.75,3.52) {$\nu_D(f)$};
                    \node[right] at (6.75,0.66) {$\nu_D(f)+k^-_{i,D}\langle \gamma,\eta_i\rangle$};

\end{tikzpicture}
\caption{Schematic of the blowup $\toricbl{\tT,\bfN}$ near a divisor $D\subset \base{\tT}$. The vertical chain records the exceptional components and the vanishing order of $fe^{\gamma}$ along them.}
\label{fig:toric-blowup-chain}
\end{figure}

We can also see these new components of $\tM$ as the limits of certain rational sections of the map $\base{\tT}\times \Tvee\to \base{\tT}$; more precisely, the components that sit above the divisor $D$ can be defined as the limits of sections of the form $\eta_i^{\vee}(f)\cdot s$ for $s$ a regular section of $\base{\tT}\times \Tvee\to \base{\tT}$ and $f$ a rational function on $\base{\tT}$ whose order of vanishing $\nu_D(f)$ satisfies  $-k^+_{i,D}\leq \nu_D(f)\leq k^-_{i,D}$.  Each possible value of $\nu_D(f)$ gives a different component of the fiber over $D$ in $\toricblc{\tT,\bfN}$.

\subsubsection{Monopole operators in the abelian case}
Fix a cocharacter $\gamma$ of $T$.  For the sake of simplicity, we choose $\eta_i$ so that $\langle \gamma,\eta_i\rangle\geq 0$ for all $i$; if this is not the case, we can replace $\eta_i$ by $-\eta_i$ as discussed above.  
\cref{lem:regularity-coefficients} gives us a divisor 
\[E(\gamma)= \sum_{\langle \gamma,\varphi_{ij}\rangle>0} \langle \gamma,\varphi_{ij}\rangle D_{\varphi_{ij}}=\sum_i \langle \gamma,\eta_i\rangle  D_{\ell_i}^+\]
\notation{$E(\gamma)$}{Divisor governing the $\gamma$-weight space in $\soa_{T,\bfN}$}in $\base{\tT}$ such that $fe^{\gamma}$ lies in $\soa_{T,\bfN}(U)$ if and only if $f$ lies in $\mathcal O_{\base{\tT}}(-E(\gamma))(U)$.  In particular, the $\gamma$-weight space of $\soa_{T,\bfN}$ is a copy of the line bundle $\mathcal O_{\base{\tT}}(-E(\gamma))$ on $\base{\tT}$.  Thus, we can write $\soa_{T,\bfN}$ as a direct sum of line bundles on $\base{\tT}$ indexed by the characters of $\Tvee$.    

Note that in the elliptic case, the line bundles may be non-trivial.  For example, if we consider $\bfN=\C$ with a weight 1 action of $\tT=\Gm$, then $D^+$ is the origin with multiplicity 1 and $D^-$ is trivial.  Thus, $E(\gamma)=\gamma D^+$ if $\gamma>0$ and $E(\gamma)=0$ if $\gamma\leq 0$.  Therefore, the functions satisfying \cref{lem:regularity-coefficients} for a given $\gamma$ form a copy of the trivial bundle if $\gamma\leq 0$ and a copy of $\mathcal{O}(-\gamma D^+)$ if $\gamma>0$.
\begin{Remark}
    This leads to a somewhat surprising conclusion: if we instead took weight $-1$, the resulting Coulomb branch would not be isomorphic to that for $+1$ via the usual map of multiplying by Euler classes.  If we considered two-dimensional representations with weights $(1,1)$ and $(1,-1)$, the discrepancy becomes even starker: in the former case, the Coulomb branch has affinization $\mathbb{A}^1$, whereas in the latter case, the affinization is a single point, so the resulting Coulomb branches cannot be isomorphic as schemes.  
    
    This asymmetry is expected from the 5d $\mathcal{N}=1$ viewpoint, where the elliptic curve is the compactification torus.  From that viewpoint, the sign choice is part of the choice of 5d Chern-Simons coupling.  More precisely, for a $\Gm$-gauge theory with an odd number of charge-one hypermultiplets, the parity anomaly forces the level to be half-integral; changing the polarization of the charged matter shifts this level by an integer \cite{seibergFivedimensional1996, intriligatorFivedimensional1997}.  Thus, the failure of the two sign choices to give canonically isomorphic elliptic Coulomb branches reflects extra 5d data rather than a defect of the construction.  The role of equivariant elliptic cohomology in this picture is parallel to its physical realization by placing a 3-dimensional theory on an elliptic curve, as in \cite{bullimoreEllipticCurve2022}.  
\end{Remark}

For $B=\Ga,\Gm$, the line bundles appearing are all trivial, so choosing a trivialization of each line bundle gives a distinguished basis vector in each $\gamma$-weight space, which we call the {\bf monopole operator} of $\gamma$.  We can construct such a trivialization by taking 
\begin{equation}\label{eq:monopole-operator-abelian}
f_{\gamma}=\prod_{\langle \gamma,\varphi_{ij}\rangle>0}\Eu({\varphi_{ij}})^{\langle \gamma,\varphi_{ij}\rangle},\qquad r_{\gamma}=f_{\gamma}e^{\gamma}.    
\end{equation}
\notation{$r_{\gamma}$}{Monopole operator of coweight $\gamma$}

In the elliptic case, this product of Euler classes is not a regular function; rather, it defines a map from $\mathcal O_{\base{\tT}}(-E(\gamma))$ to $\soa_{T,\bfN}$.  We regard this map as the elliptic analogue of the monopole operator of $\gamma$.  

\subsubsection{Comparison with the abelian BFN construction}
Finally, we show that this construction recovers the original definition of the Coulomb branch.  Since no careful analysis of the elliptic BFN construction currently exists in the literature, we leave that case to future work and assume in this section that $B=\Ga$ or $\Gm$.      
\begin{Theorem}\label{thm:abelian-case}
    The variety $\tM$ is isomorphic to the usual BFN Coulomb branch, which we denote $\tMB(\tT,T,\bfN)$, if $B=\Ga$, and to the usual $K$-theoretic Coulomb branch if $B=\Gm$.
\end{Theorem}
\begin{proof}
    This is philosophically the same as \cite[Th. 1 \& 2]{telemanRoleCoulomb2021}, but we cannot directly apply that result because we have not checked compatibility with the ``massive'' deformations (usually called flavor deformations in other sources).
    
    We wish to show that the birational map from $\tM\dashrightarrow \base{\tT}\times \Tvee$ given by the embedding $\soa_{T,\bfN}(U)\subset \K(\base{\tT})[\Tvee]$ for any affine set $U$ induces an isomorphism with the BFN Coulomb branch, which also has a birational map to $\base{\tT}\times \Tvee$ given in \cite[Cor. 5.11(1)]{BFN}.

    There is a monopole operator $r_\gamma^{\BFN}\in \K[\tMB(\tT,T,\bfN)]$ for each cocharacter $\gamma$ of $T$, and 
    by \cite[\S 4(vi)]{BFN}, the monopole operator $r_\gamma^{\BFN}$ maps to  $e^\gamma\cdot \Eu(\bfN\llbracket t \rrbracket/(\bfN\llbracket t \rrbracket\cap t^{\gamma}\bfN\llbracket t \rrbracket))$ on $\base{\tT}$.  The weight line for $\varphi_{ij}$ contributes in $\bfN\llbracket t \rrbracket/(\bfN\llbracket t \rrbracket\cap t^{\gamma}\bfN\llbracket t \rrbracket)$ with multiplicity $\langle \gamma,\varphi_{ij}\rangle$  if $\langle \gamma,\varphi_{ij}\rangle>0$ and does not contribute if $\langle \gamma,\varphi_{ij}\rangle\leq 0$.  Thus, the image of $r_\gamma^{\BFN}$ is $r_{\gamma}$ as defined in \cref{eq:monopole-operator-abelian}.
    
    We have already established in \cref{lem:regularity-coefficients} that $r_{\gamma}$ generates a copy of $\K[\base{\tT}]$ which is the $\gamma$-weight space of $\soa_{T,\bfN}$.  The same is true for $r_\gamma^{\BFN}$ in $\K[\tMB(\tT,T,\bfN)]$ by \cite[Th. 4.1]{BFN}, so the map $\K[\tMB(\tT,T,\bfN)]\to \soa_{T,\bfN}$ is an isomorphism on each $\gamma$-weight space, and thus an isomorphism.
\end{proof}

\subsection{General case}
\label{subsec:general-case}

\subsubsection{Construction}
We now extend the abelian construction to the non-abelian case.  The basic strategy is the same: we start from the birational model over $\base{\tT}\times \Tvee$ and impose explicit regularity conditions along the divisors that control codimension-$1$ behavior.  The new feature is that we must incorporate the Weyl-group action and the additional regularity conditions coming from the root divisors in order to obtain a well-behaved variety over $\base{G}$.  This is a natural expectation from GKM theory: the equivariant cohomology $H^*_G(X)$ of a $G$-space differs from the $T$-fixed-point contribution $H^*_T(X^T)^W$ only by regularity conditions along the divisors corresponding to the stabilizers of non-fixed points.  We will see below that the same regularity conditions, imposed at the root divisors $D_{\alpha}$, cut $\soa_{G,\bfN}$ out of $(\K(\base{\tT})[\Tvee])^W$.

The space $\base{\tT}\times \Tvee$ has an induced action of the Weyl group $W$ of $G$, where $W$ acts on $\base{\tT}$ by the usual action on the torus and on $\Tvee$ by the dual action. 
\begin{Lemma}\label{lem:goodness-degree}
    The sheaf $\soa_{T,\bfN}$ carries, and thus the variety $\tM(\tT,T,\bfN)$ inherits, a unique action of the Weyl group $W$ of $G$ such that the birational map $\base{\tT}\times \Tvee\to \tM(\tT,T,\bfN)$ is $W$-equivariant.
\end{Lemma}
\begin{proof}
 By the usual structure theorems of finite-dimensional $G$-representations, the $W$-action sends each weight to another weight.  Consequently, it sends each line $\ell_i$ of weights to another line $\ell_j$.  Since the $W$-action on $\base{\tT}$ is compatible with the $W$-action on the character lattice, it also sends $\eta_i$ to $\pm \eta_j$ for some sign and similarly $D_{\ell_i}^{\pm}$ to $D_{\ell_j}^{\pm}$ or $D_{\ell_j}^{\mp}$.  This shows that $W$ lifts to an action on $\toricbl{\tT,\bfN}$ and thus on $\tM(\tT,T,\bfN)$ as well.  The map $\tM(\tT,T,\bfN)\to \base{\tT}$ is $W$-equivariant, so taking pushforward, we get an induced action of $W$ on $\soa_{T,\bfN}$ as well.
\end{proof}

Given this Weyl group action, we can define $\soa_{G,\bfN}$; in fact, the primary difficulty in the consideration of the non-cotangent case in \cite{telemanCoulombBranches2022} is the definition of this $W$-action.  The $W$-invariants of $\soa_{T,\bfN}$ form a subalgebra, but they are not yet the correct algebra for the Coulomb branch.  Rather, we must impose weaker regularity conditions along the preimages of the divisors $D_{\alpha}$ in $\base{G}$.  

For each root $\alpha$, let $\ell_{\alpha}$ denote the line spanned by $\alpha$.  In addition to the weight divisors, we must also account for the divisor in $\Tvee$ defined by the equation $e^{\alpha^{\vee}}=1$; we denote it by $D_{\alpha,1}$.\notation{$D_{\alpha,1}$}{Root divisor in $\Tvee$ cut out by $e^{\alpha^{\vee}}=1$}

\begin{Definition}
Let $\epsilon_{\alpha}=1$ if $\alpha/2$ is not a character of $T$ and $\epsilon_{\alpha}=2$ if $\alpha/2$ is a character of $T$.  
    Let $J_{\alpha}$ be the ideal in $\K[\tM(\tT,T,\bfN)]$ consisting of functions that vanish to order $\epsilon_{\alpha}$ on the intersection of the preimage of $D_{\alpha}$ in $\toricbl{\tT,\bfN}$ with $D_{\alpha,1}$.  Let $\Bl_{J_{\alpha}}(\tM(\tT,T,\bfN))$ be the blowup of $\tM(\tT,T,\bfN)$ with this ideal as center.\notation{$\epsilon_{\alpha}$}{Multiplicity governing the root blowup}\notation{$J_{\alpha}$}{Ideal defining the root blowup center}\notation{$\nonabbl{\alpha,\tG,\bfN}$}{Local non-abelian blowup model for the root $\alpha$}

Let $\nonabbl{\alpha,\tG,\bfN}$ be the complement in $\Bl_{J_{\alpha}}(\tM(\tT,T,\bfN))$ of the principal transform of $\epsilon_{\alpha}D_{\alpha}$, that is, of the divisor cut out by the pullback of the functions vanishing on $\epsilon_{\alpha}D_{\alpha}$ as sections of $\mathcal{O}(1)$.
\end{Definition} 
This is an affine variety over $\base{\tG}$, since $\mathcal{O}(1)$ is relatively ample on the blowup.  The map $\nonabbl{\alpha,\tG,\bfN}\to\tM(\tT,T,\bfN)$ is an isomorphism away from the preimage of $D_{\alpha}$, and in the preimage of $D_{\alpha}$, there are a finite number of irreducible divisors $E_1,\dots, E_r$ that map dominantly onto a component of $D_{\alpha}$.
\begin{Definition}\label{def:Coulomb-branch-nonabelian}
    Consider the sheaf of rings on $\base{\tG}$ whose sections $\soa_{G,\bfN}(U)$ form the subring of $(\K(\base{\tT})[\Tvee])^W$ consisting of functions that are regular 
    \begin{itemize}
    \item on the preimage of $U\setminus(\cup_{\alpha} D_{\alpha})$ in $\tM(\tT,T,\bfN)$ and 
    \item at the generic points of all components $E_i$ whose images intersect $U$.
    \end{itemize}   
    The Coulomb branch $\tM(\tG,G,\bfN)$ is the relative spectrum of $\soa_{G,\bfN}$. Let $\ints$ be the natural map $\tM(\tG,G,\bfN)\to \base{\tG}$.
    \notation{$\tM(\tG,G,\bfN)$}{The Coulomb branch of the representation $\bfN$ of $G$ with flavor deformation for $\tG$}\notation{$\soa_{G,\bfN}$}{The sheaf on $\base{\tG}$ giving the functions on the Coulomb branch} 
    \notation{$\ints$}{Natural map from the Coulomb branch to $\base{\tG}$}\end{Definition}

In the semi-simple rank-$1$ examples below, there is a unique root $\alpha$, so we abbreviate $\nonabbl{\alpha,\tG,\bfN}$ to $\nonabbl{\tG,\bfN}$.  

\subsubsection{Rank-$1$ examples}
In the following two examples, let $m=k_{i,D}^{\pm}$ for $D$ the origin in $\base{\tT}$; these multiplicities are equal since the unique non-trivial element of the Weyl group acts by negation. Thus, this single multiplicity governs the rank-$1$ abelian model.

\begin{Example}    
    \label{ex:PGL2-running}
    We now discuss the example $G=\tG=PGL(2)$.  For simplicity, we only consider the case $B=\Ga$.  Consider the subring $A=\K[\tau, z\tau^m,z^{-1}\tau^m]\subset A'=\K[\tau,z^{\pm 1}]$, where $\tau$ is a coordinate on $\base{\tT}=\mathbb{A}^1$ and $z$ is a coordinate on $\Tvee=\Gm$.  By \cref{ex:Gm}, we have  
    \[\tM(\tT,T,\bfN)=\Spec(A).\]
    A point of the spectrum is determined by a value of $\tau$ and, if $\tau\neq 0$, a value of $z\neq 0$.  If $\tau=0$, then we get two copies of $\mathbb{A}^1$ meeting at a point, with coordinates given by $z\tau^m$ and $z^{-1}\tau^m$; there is only one point of intersection, since $(z\tau^m)(z^{-1}\tau^m)=\tau^{2m}=0$.  
    
    The divisor $D_{\alpha}$ in $\base{G}$ is the origin, and the relevant points of $D_{\alpha,1}$ in $\base{\tT}\times \Tvee$ are given by $\tau=0$ and $z=\pm 1$.  
    The ideal $J$ is generated by $z-z^{-1}$ and $\tau$ if $m=0$, and by $\tau^m z, \tau^mz^{-1},\tau$ if $m>0$.  
    
    In the case $m=0$, the blowup adds two exceptional $\mathbb{P}^1$'s over the two points $z=\pm 1$, and removing the principal transform, which is just the strict transform in this case, leaves us with two $\mathbb{A}^1$'s, as expected.  We have \[\nonabbl{G,\bfN}=\Spec \K[\tau, z^{\pm 1}, \frac{z-z^{-1}}{\tau}].\]  
    The Weyl group $W$ is $\Z/2\Z$, which acts by sending $\tau$ to $-\tau$ and $z$ to $z^{-1}$, so the invariants are generated by $z+z^{-1}, \tau^2$, and $\frac{z-z^{-1}}{\tau}$, in agreement with \cite{BFN}.  In this case, $\nonabbl{G,\bfN}$ is the universal centralizer of $SL(2)$ and thus is smooth.

    If $m>0$, then the blowup maps to $\tM(\tT,T,\bfN)\times \mathbb{P}^2$, with homogeneous coordinates $[a:b:c]=[\tau^m z: \tau^mz^{-1}:\tau]$ on the second factor.  The image is cut out by $ab=c^2\tau^{2m-2}$.  Away from the origin, this still gives a single point, but over the origin, we obtain the projective variety $ab=c^2$ if $m=1$, namely the $2$-uple embedding of $\mathbb{P}^1$, and $ab=0$ if $m>1$, namely two copies of $\mathbb{P}^1$ meeting at $[0:0:1]$.

    Removing the principal transform means restricting to $c\neq 0$.  We then obtain \[\nonabbl{G,\bfN}=\Spec \K[\tau^{m-1}z=a/c, \tau^{m-1}z^{-1}=b/c, \tau]\] and the invariants are generated by $\tau^{m-1}z+(-\tau)^{m-1}z^{-1}$, $\tau^{m}z+(-\tau)^{m}z^{-1}$, and $\tau^2$.  Note that in this case, $\nonabbl{G,\bfN}$ is an $A_{2m-2}$ surface singularity (or $\mathbb{A}^2$ if $m=1$) and thus is normal.
    
    Since $\nonabbl{G,\bfN}$ is normal in all cases, the variety $\tM(G,G,\bfN)$ is just the quotient $\nonabbl{G,\bfN}/W$.  
\end{Example}

\begin{Example}\label{ex:SL2-running}
The case of $G=SL(2)$ is slightly more delicate.  Just as before, we have $\tM(\tT,T,\bfN)=\Spec(A)$ depending on $m$ as in \cref{ex:PGL2-running}.  We use the same coordinates as before, but now $D_{\alpha,1}$ is given just by $z=1$.  Since $\epsilon_{\alpha}=2$, the ideal $J$ is $(z-1,\tau)^2$ if $m=0$, and $(\tau^m (z-1), \tau^m(z^{-1}-1),\tau^2)$ if $m>0$.  

In the case $m=0$, there is now a single exceptional $\mathbb{P}^1$.  Removing the principal transform, we find $\nonabbl{G,\bfN}=\Spec \K[\tau, z^{\pm 1}, \frac{z-1}{\tau}]$.  Thus, the invariants are generated by $\frac{z-2+z^{-1}}{\tau^2}$, $\tau^2$, and $\frac{z-z^{-1}}{\tau}$.   In this case, $\nonabbl{G,\bfN}$ is the universal centralizer of $PGL(2)$ and thus is smooth.

  If $m=1$, we have the same map to $\tM(\tT,T,\bfN)\times \mathbb{P}^2$ with \[[a:b:c]=[\tau (z-1): \tau (z^{-1}-1):\tau^2].\]  If we invert $c$, we obtain the spectrum of $\K[\tau,\frac{z-1}{\tau}, z^{\pm 1}]$, so the invariants are generated by $\frac{z-z^{-1}}{\tau}, z+z^{-1}$ and $\tau^2$. 
  
    If $m>1$, the ideal can be rewritten as $(\tau^m z, \tau^m z^{-1},\tau^2)$, and we get a similar result, with invariants generated by $\tau^{m-2}z+(-\tau)^{m-2}z^{-1}, \tau^{m-1}z+(-\tau)^{m-1}z^{-1},\tau^2$.  Since these last two cases are isomorphic to $PGL(2)$ Coulomb branches, we have already checked that they are normal, so the variety $\tM(G,G,\bfN)$ is just the quotient $\nonabbl{G,\bfN}/W$ in all cases.
\end{Example}

As in \cite{BFN}, understanding these examples is crucial to the proof of normality of $\tM(\tG,G,\bfN)$ in the general case.  For now, we note that these examples are sufficient to check this normality in the rank-$1$ case.
\begin{Lemma}\label{lem:rank-1-normal}
        If $G$ and $\tG$ have rank $1$, then $\tM(\tG,G,\bfN)$ is normal.
\end{Lemma}
\begin{proof}
    Normality is a local property, so it is enough to consider the stalk $\soa_b$ of $\soa_{G,\bfN}$ at a point $b$ in $\base{G}$ and to prove that it is integrally closed.       Let $b'$ be a preimage of $b$ in $\base{T}$ and let $\tau$ be a uniformizer at $b'$.   By \cref{lem:regularity-coefficients}, the stalk of $\soa_{T,\bfN}$ at $b$ is the subring of $Y=\mathcal{O}_{\base{T},b}[z,z^{-1}]$ generated over $\mathcal{O}_{\base{T},b}$ by $\tau^{m_+}z, \tau^{m_-}z^{-1}$ where $m_{\pm}$ are the multiplicities of this point in the divisors $D^{\pm}$; that is, it is the ring $\mathcal{O}_{\base{T},b}[u,v]/(uv-\tau^{m_++m_-})$ where $u$ and $v$ are the generators corresponding to $\tau^{m_+}z$ and $\tau^{m_-}z^{-1}$, respectively. 

Note that this ring is integrally closed:
\begin{enumerate}
\item[$(S_2)$] It satisfies Serre's condition $S_2$ since it is the coordinate ring of a hypersurface in $\mathbb{A}^2_{\mathcal{O}_{\base{T},b}}$.
\item[$(R_1)$] It is regular in codimension $1$ since the singular locus is given by $u=v=\tau=0$, which has codimension $2$.
\end{enumerate} 

We can now analyze the stalk $\soa_b$  by considering the following cases:
\begin{enumerate}
        \item If $b$ is not fixed by $W$, then $\soa_b$ is the same as the stalk of $\soa_{T,\bfN}$ at $b'$, so this stalk is integrally closed.
        \item If $b$ is fixed by $W$, but not in $D_{\alpha}$, then $\soa_b$ is the invariants of the stalk of $\soa_{T,\bfN}$ at $b'$ under the action of $W$.  This is integrally closed, since we are taking invariants under a finite group action on an integrally closed ring.
           \item If $b$ is fixed by $W$ and lies in $D_{\alpha}$, then the blowup at
        $J_{\alpha}$ changes the underlying variety.  By averaging, we can assume that
        $s\tau=-\tau$ for $s$ the non-trivial element of $W$.  Repeating the argument in
        the examples above, we find that the stalk of $\soa_{G,\bfN}$ at the image of
        $b$ in $\base{G}$ is the subring of $Y^W$ generated over $\mathcal{O}_{\base{G},b}$ by the functions
        \[
        \frac{z-2+z^{-1}}{\tau^2},\qquad \frac{z-z^{-1}}{\tau},\qquad \tau^2
        \]
        if $G=SL(2)$ and $m=0$, and by
        \[
        \tau^{m-\epsilon_{\alpha}}z+(-\tau)^{m-\epsilon_{\alpha}}z^{-1},\qquad
        \tau^{m+1-\epsilon_{\alpha}}z+(-\tau)^{m+1-\epsilon_{\alpha}}z^{-1},\qquad
        \tau^2
        \]
        otherwise.  This is smooth if $m-\epsilon_{\alpha}<0$, and otherwise it is the
        $W$-invariants in
        $\mathcal{O}_{\base{T},b}[u,v]/(uv-\tau^{2m-2\epsilon_{\alpha}})$.
        In either case, it is integrally closed as established before.
\end{enumerate}
This shows the desired normality.
\end{proof}

\subsubsection{Slices and codimension-$1$ structure}
\label{subsubsec:slices-codim-1}
In order to analyze the structure of $\tM(\tG,G,\bfN)$, we can use the locality of its definition.  
Let $\lambda\in \LtT$.  We have a corresponding Levi $\Gl=C_G(\lambda)$ and fixed subspace $\Nl$.  Note that the roots of $\Gl$ are precisely the roots of $G$ that vanish on $\lambda$, and the weights of $\Nl$ are precisely the weights of $\bfN$ that vanish on $\lambda$.  
\begin{Definition} 
Let $\base{\tT,\lambda}$ be the complement in $\base{\tT}$ of all divisors $D_{\varphi}$ for weights $\varphi$ such that $\varphi(\lambda)\neq 0$, and all divisors $D_{\alpha}$ for roots $\alpha$ such that $\alpha(\lambda)\neq 0$.  
\notation{$\base{\tT,\lambda}$}{Open subset where only the divisors of weights and roots vanishing on $\lambda$ remain}\end{Definition}
The restriction of $\soa_{G,\bfN}$ to this open set is determined by the valuations discussed earlier, simply ignoring those coming from the weights and roots that do not vanish on $\lambda$.  Thus, exactly the same conditions apply to the sections of $\soa_{\Gl,\Nl}$, after accounting for the difference in the Weyl groups, that is, the fact that $\tM(\tGl,\Gl,\Nl)$ is a scheme over $\base{\tGl}=\base{\tT}/W_{\lambda}$.  
Throughout, for a scheme $X$ over $\base{\tG}$ and a $\base{\tG}$-scheme $U$, we write $X_U$ for the base change $X\times_{\base{\tG}} U$.  The discussion above shows that:
\begin{Lemma}\label{lem:slice-ours}
$\tM(\tG, G,\bfN)_{\base{\tGl,\lambda}}\cong \tM(\tGl,\Gl,\Nl)_{\base{\tGl,\lambda}}$
\end{Lemma}
The corresponding result for BFN Coulomb branches is implicit in \cite[Th. 5.26]{BFN}:
\begin{Lemma}\label{lem:slice-BFN}
$\tMB(\tG,G,\bfN)_{\base{\tGl,\lambda}}\cong \tMB(\tGl,\Gl,\Nl)_{\base{\tGl,\lambda}}$
\end{Lemma}
\begin{proof}
We have a natural map $\tMB(\tG,G,\bfN)_{\base{\tGl,\lambda}} \to \tMB(\tGl,\Gl,\bfN)$ induced by the map $\Gr_{\tGl}\to \Gr_{\tG}$ which is the inclusion of the fixed points of $\la$.  By the usual long exact sequence, the kernel and cokernel of this map are supported on the subspace in $\LtT$ given by the union of the Lie algebras of the stabilizers of points in $\Gr_{\tG}\setminus \Gr_{\tGl}$, which are always proper. The stabilizer of any point in $\Gr_{\tG}$ is an intersection of root subspaces, so any such intersection which does not contain $\lambda$ must lie in the vanishing set of a root which is non-zero on $\lambda$. Thus, this support has trivial intersection with $\base{\tT,\lambda}$, and so we have an isomorphism
    \[\tMB(\tG,G,\bfN)_{\base{\tGl,\lambda}} \cong \tMB(\tGl,\Gl,\bfN).\]
    On the other hand, \cite[Rem. 5.14]{BFN} identifies $\tMB(\tGl,\Gl,\Nl)$ with $\tMB(\tGl,\Gl,\bfN)$ after base change to $\base{\tGl,\lambda}$, which gives the claim.
    \end{proof}

Particularly interesting cases of these results occur when:
\begin{enumerate}
	\item No weights or roots vanish on $\lambda$.  We call such $\lambda$ {\bf generic}.
In this case, we get the open inclusion of $(\LtT^{\circ}\times T^{\vee})/W$ into $\tMB(\tG,G,\bfN)$ discussed in \cite[Lem. 5.9]{BFN}.

\item The weights and roots vanishing on $\lambda$ lie in a single line, or equivalently, $\lambda$ does not lie in the intersection of two distinct hyperplanes of the generalized root arrangement.  We call $\lambda$ {\bf subgeneric} in this case (assuming it is not generic).  For subgeneric $\lambda$, if $\Nl$ is non-trivial, the centralizer $\Gl$ acts on $\Nl$ via a rank-$1$ quotient.
\end{enumerate}

We wish to think of the maps $\mapl$ for generic and subgeneric $\lambda$ as an \'etale open cover of $\tMB(\tG,G,\bfN)$.  Indeed, these maps are \'etale because the maps $\base{\tGl,\lambda}\to \base{\tG}$ are \'etale, but they are not jointly surjective.  However, they do cover an open subset $\tMB^{\bullet}(\tG,G,\bfN)$ given by the preimage of the generic and subgeneric points in $\base{\tG}$, and this is enough to control the structure of $\tMB(\tG,G,\bfN)$:
\begin{Lemma}\label{lem:bullet-affinization}
	If $B=\Ga$ or $\Gm$, the variety $\tMB(\tG,G,\bfN)$ is the affinization of the open subset $\tMB^{\bullet}(\tG,G,\bfN)$ over $\base{\tG}$. 
\end{Lemma} 
\begin{proof}
    This is proven in the case $B=\Ga$ in \cite[Th. 5.26]{BFN}, and the $B=\Gm$ case is effectively the same.
    Let $j\colon \tMB^{\bullet}(\tG,G,\bfN)\to \tMB(\tG,G,\bfN)$ be the inclusion.  We need to show that the natural map $j_*j^*\mathcal{O}_{\tMB^{\bullet}(\tG,G,\bfN)}\to \mathcal{O}_{\tMB(\tG,G,\bfN)}$ is an isomorphism.  It is easier to see this after pushing forward to $\base{\tG}$. 
    Since the map $\tMB(\tG,G,\bfN)\to \base{\tG}$ is flat, the pushforward of the structure sheaf is a vector bundle, and on a smooth variety, it satisfies the condition that $\ints_*j_*j^*  \mathcal{O}_{\tMB(\tG,G,\bfN)}= i_* i^* \ints_* \mathcal{O}_{\tMB(\tG,G,\bfN)}\to \ints_* \mathcal{O}_{\tMB(\tG,G,\bfN)}$ is an isomorphism, for $i$ the inclusion of the (sub)generic locus in $\base{\tG}$.  The relative affinization of $\tMB^{\bullet}(\tG,G,\bfN)$ is the relative spectrum of $\ints_*j_*j^*  \mathcal{O}_{\tMB(\tG,G,\bfN)}$, so the natural map from this affinization to $\tMB(\tG,G,\bfN)$ is an isomorphism.
\end{proof}

Since it will be important, we note that the same is true for $\tM(\tG,G,\bfN)$ as well: 
\begin{Lemma}\label{lem:bullet-affinization-ours}
    The variety $\tM(\tG,G,\bfN)$ is the relative affinization of the preimage $\tM^{\bullet}(\tG,G,\bfN)$ of the generic or subgeneric points in $\base{\tG}$.  
\end{Lemma}
\begin{proof}
    By definition, $\soa_{G,\bfN}$ is the intersection of the $W$-invariants of the structure sheaf on $\tM^{\circ}(\tT,T,\bfN)$, the preimage of the generic points, with functions regular at a finite number of generic points of divisors $E_i$; we ignore any divisors that don't map dominantly to one of these components.  The former condition lies over generic points of $\base{\tG}$, and the latter over subgeneric points, since the $E_i$ are precisely the divisors lying over components of $D_{\alpha}$.  Both conditions are unchanged by restricting to $\tM^{\bullet}(\tG,G,\bfN)$.   
\end{proof}

\subsubsection{Finite covers}
\label{subsubsec:finite-covers}

It is useful to note how the spaces $\tM(\tG,G,\bfN)$ and $\tMB(\tG,G,\bfN)$ change when we replace $\tG$ and $G$ by covers.  
Let $\tG'$ be a finite cover of $\tG$ with kernel $\tK$, and let $G'$ be a finite cover of $G$ with kernel $K$. 
Note that the map $\base{\tT'}\to \base{\tT}$ is a cover with Galois group 
\[\tK_B\cong\begin{cases} 1 & \text{if } B=\Ga \\ \tK & \text{if } B=\Gm \\ \tK^2 & \text{if } B=\elli \end{cases}. \] 
The map $\Tvee\to (T')^{\vee}$ is a cover with Galois group $K^{\vee}$, the Pontryagin dual of $K$.  
It is manifest from the definition that:
\begin{Lemma}\label{lem:cover-ours}
    We have an isomorphism \[\tM(\tG',G,\bfN)\cong \tM(\tG,G,\bfN)\times_{\base{\tG}} \base{\tG'}.\]  This induces a free $\tK_B$ action on the variety $\tM(\tG',G,\bfN)$ such that \[\tM(\tG',G,\bfN)/\tK_B\cong \tM(\tG,G,\bfN).\] 
    
    Furthermore, the variety $\tM(\tG,G,\bfN)$ carries an action of $K^{\vee}$ such that \[\tM(\tG,G,\bfN)/K^{\vee}\cong \tM(\tG,G',\bfN).\]
\end{Lemma}
The same statement holds for $\tMB$ as well. 
\begin{Lemma}\label{lem:cover-BFN}
     We have isomorphisms \begin{align*}
        \tMB(\tG',G,\bfN)/\tK_B&\cong \tMB(\tG,G,\bfN)\\
         \tMB(\tG,G,\bfN)/K^{\vee}&\cong \tMB(\tG,G',\bfN).
     \end{align*}
\end{Lemma}
\begin{proof}
    The first isomorphism is just base change in equivariant cohomology.  The second is the fact that the action of $K^{\vee}$ on $\K[\tMB(\tG,G,\bfN)]$ is induced by the grading of this ring by $K\cong \pi_1(G)/\pi_1(G')$, so the $K^{\vee}$-invariants are precisely the degree 0 part of this grading, which is $\K[\tMB(\tG,G',\bfN)]$.  This is effectively the same argument as \cite[Prop. 3.18]{BFN}, but applied to a finite group rather than a torus.
\end{proof}

\begin{Example}\label{ex:SL2-PGL2-cover}
    One example that will be relevant is when $\tG=SL(2)$ and $G=PGL(2)$, and $\bfN$ is an arbitrary representation of $G$.  The map $\tT\to T$ is the double cover of $\Gm$ by itself. We have a canonical map $\base{\tT}\to \base{T}$, which is always a cover.  The variety $\tM(\tG,G,\bfN)$ is the base change of $\tM(G,G,\bfN)$ along this map, so we have the following cases:
    \begin{enumerate}
        \item If $B=\Ga$, then $\base{\tT}=\base{T}$, so $\tM(\tG,G,\bfN)\cong \tM(G,G,\bfN)$.
        \item If $B=\Gm$, then $\base{\tT}\to \base{T}$ is a double cover, so $\tM(\tG,G,\bfN)$ is a double cover of $\tM(G,G,\bfN)$, with singular fibers over $\pm 1$ in $\base{\tT}$ (whereas $\tM(G,G,\bfN)$ only has singular fiber over $1$).
        \item If $B=\elli$, then $\base{\tT}\to \base{T}$ is the squaring map, so $\tM(\tG,G,\bfN)$ is a 4-fold cover of $\tM(G,G,\bfN)$, with singular fibers over the four $2$-torsion points in $\base{\tT}$ (whereas $\tM(G,G,\bfN)$ only has singular fiber over the identity).
    \end{enumerate}
\end{Example}

\subsubsection{Semi-simple rank $1$}
\label{subsubsec:semi-simple-rank-1}

A particularly important special case arises when we consider $\Gl$ and $\Nl$ for $\lambda$ subgeneric.  In this case, the centralizers $\tGl$ and $\Gl$ act on $\Nl$ via a rank-$1$ quotient.  Since we have already analyzed the abelian case, we assume that $G$ is non-abelian, so the semi-simple rank of all these groups is $1$, and the action of $\tGl$ and $\Gl$ on $\Nl$ factors through $H\cong PGL(2)$ or $SL(2)$, which acts faithfully.  The kernels of the maps from these groups to $H$ are tori $C,\tilde{C}$.  Note that the maps $C\to C'=\Gl/[\Gl,\Gl]$ and $\tilde{C}\to \tilde{C}'=\tGl/[\tGl,\tGl]$ are isogenies.  Thus, the map $\tGl\to H\times \tilde{C}'$ is a finite cover.  Similarly, there is a cover $H'$ of $H$ such that $H'\times \tilde{C}\to \tGl$ is a finite cover, so we can apply the results above to show that $\tM(\tG,G,\bfN)$ is a finite quotient of the variety
\[\tM(H'\times \tilde{C},H\times C',\Nl)=\tM(H',H,\Nl)\times \tM(\tilde{C},C',0)\]
using the product decomposition of \cite[\S 3(vii)(a)]{BFN}. 
The variety $\tM(\tilde{C},C',0)$ is just $\base{\tilde{C}} \times (C')^{\vee}$, and so is smooth.  The variety $\tM(H',H,\Nl)$ is normal by \cref{lem:rank-1-normal}. Since a quotient of a normal variety by a finite group is again normal, we find that:
\begin{Corollary}
    The variety $\tM(\tGl,\Gl,\Nl)$ is normal for $\lambda$ subgeneric. 
\end{Corollary}
By \cref{lem:bullet-affinization-ours}, we have that $\tM(\tG,G,\bfN)$ is the affinization of the open subset \'etale covered by open subsets of $\tM(\tGl,\Gl,\Nl)$ for $\lambda$ subgeneric, so this implies that:
\begin{Corollary}\label{cor:normality-ours}
    The variety $\tM(\tG,G,\bfN)$ is normal.
\end{Corollary}
Implicit in this discussion is a description of the codimension-$1$ primes in $\tM(\tG,G,\bfN)$.  Of course, most of these are simply the images of divisors in $(\base{\tT}\times \Tvee)/W$; the only subtlety is understanding the divisors lying over components of $D_{\ell_i}$ or $D_{\alpha}$ in $\base{\tT}$.  We can collect the
results above to give a precise description of these divisors:
\begin{Proposition}\label{prop:divisors}
    The irreducible Cartier divisors in $\tM(\tG,G,\bfN)$ are as follows:   
    \begin{enumerate}
        \item The divisors in $(\base{\tT}\times \Tvee)/W$ that do not lie over a component of $D_{\ell_i}$ or $D_{\alpha}$.
        \item For each component $D$ of $D_{\ell_i}$ which is not a component of $D_{\alpha}$, the divisors $D_{+}$ and $D_{-}$ in $\tM(\tG,G,\bfN)$ corresponding to the valuations $\nu_{D,+}$ and $\nu_{D,-}$.
        \item For each component $D$ of $D_{\alpha}$, either one or two divisors in $\tM(\tG,G,\bfN)$ lying over $D$.  We have two divisors, defined by the blowups of $e^{\eta_{\alpha}}=\pm 1$ in $\Tvee$, if and only if 
        \begin{enumerate}
            \item $\alpha^{\vee}/2$ is a cocharacter of $T$ and $k_{\alpha,D}=0$, or
            \item $\alpha/2$ is a character of $T$ and $k_{\alpha,D}=1$.
        \end{enumerate}
        Otherwise, the preimage of $D$ in $\tM(\tG,G,\bfN)$ is a single divisor.
    \end{enumerate}
\end{Proposition}

\subsection{Comparison and variants}
\label{subsec:comparison-and-variants}
Having constructed $\tM(\tG,G,\bfN)$, we now record the main comparisons and variants that justify the definition and clarify its scope.  We begin with the comparison with the BFN construction, then explain how the same local picture compares the rational, $K$-theoretic, and elliptic cases, and finally indicate the modification needed for Coulomb branches with symmetrizers.

\subsubsection{Comparison with the BFN construction}
\label{subsubsec:BFN-comparison}

First, we consider the case where $G$ and $\tG$ are semi-simple of rank $1$, so that there is only one root $\alpha$ and one line $\ell_{\alpha}$ of weights.  
As with the discussion around \cref{thm:abelian-case}, we leave the case $B=\elli$ to future work, since the necessary analysis of the elliptic BFN construction does not yet exist in the literature.  Thus, throughout this subsection, we assume that $B$ is either $\Ga$ or $\Gm$.

Before comparing with our definition, we record some basic facts about the BFN Coulomb branches.  In the case $B=\Gm$, these do not seem to have been recorded explicitly in the literature and, in the case $B=\Ga$, they are left slightly implicit.  

Since $G$ has rank 1, the Weyl group $W=\{1,s\}$ has order 2 and acts by sending $s\cdot \alpha=-\alpha$.  
The weight $\eta=\eta_{\al}$ can be chosen to be $\alpha$ if $G=PGL(2)$ and $\alpha/2$ if $G=SL(2)$.  In either case, every weight of $\bfN$ is a multiple of $\eta$, so we can write these weights as $\varphi_j=m_j\eta$ for some integers $m_j$.  As with the root, we have $s\cdot \varphi_j=-\varphi_j$ for all $j$.  

We let $P^{\pm}=\prod_{m_{j}>0} \Eu(\pm\varphi_{j})^{m_{j}}$, which is a regular function on $\base{\tG}$.  In the case $G=SL(2)$, we use a further factorization of this polynomial 
\[ P_1^\pm = \prod_{m_{j}>0} \Eu(\pm\varphi_{j})^{\lfloor m_{j}/2 \rfloor}, \qquad P_2^\pm = \prod_{m_{j}>0} \Eu(\pm\varphi_{j})^{\lceil m_{j}/2 \rceil}, \qquad P^\pm=P_1^\pm P_2^\pm \]

If it is not constant, the function $P^{\pm}$ vanishes along $D_{\al}$, except in the case $\tG=SL(2)$ when the weights of $\bfN$ lie in $(\Z+1/2)\alpha$; in that case, $P^{\pm}$ has non-zero but equal values on the locus $e^{\alpha/2}=-1$.

In order to compare with the BFN construction, we need to compute the BFN Coulomb branch in this case.  For $B=\Ga$, this result is well-known and proven in \cite[Lem. 6.9]{BFN}.  For $B=\Gm$, this result does not seem to appear in the literature, but the proof is a routine adaptation of the $B=\Ga$ case, so we include it here for completeness.

    If $G=PGL(2)$ and $g^+=g$ is a function on $\base{\tG}$ with $g^-=s\cdot g$, then we consider the function
    \begin{equation}\label{eq:PGL2-monopole}
    r_{\alpha^{\vee}/2}(g)=
    \begin{cases}
        \dfrac{P^+g^+e^{\alpha^{\vee}/2}-P^-g^-e^{-\alpha^{\vee}/2}}{\Eu(\alpha)} & \text{if } B=\Ga, \\
        \dfrac{P^+g^+e^{\alpha^{\vee}/2}-P^-g^-e^{-\alpha}e^{-\alpha^{\vee}/2}}{\Eu(\alpha)} & \text{if } B=\Gm.
    \end{cases}
    \end{equation}
        In all cases, we have $P^+=P^-$, $g^+=g^-$, and $e^{\alpha^{\vee}/2}=e^{-\alpha^{\vee}/2}$ at all points of $D_{\al,\pm 1}\cap D_{\alpha}$; in the case $B=\Gm$, we also have $e^{-\alpha}=1$ there.  Thus, the numerator vanishes along the center of the blowup, and so $r_{\alpha^{\vee}/2}$ is a regular function on $\nonabbl{\tG,\bfN}/W$.  We let \[r_{k\alpha^{\vee}/2}(g)=r_{\alpha^{\vee}/2}^{k-1}r_{\alpha^{\vee}/2}(g)\quad\text{for }k\in \Z_{>0}.\]

    If $G=SL(2)$ (which necessitates $\tG=SL(2)$ since we have assumed it also has rank 1), then we consider the function
    \begin{equation}\label{eq:SL2-monopole}
    r_{\alpha^{\vee}}(g)=
        \begin{cases}
            \dfrac{P^+g^+e^{\alpha^{\vee}}-P^+_1P^-_2g^++P^-_1P^+_2g^-+P^-g^-e^{-\alpha^{\vee}}}{\Eu(\alpha)\Eu(-\alpha)} &     \text{if } B=\Ga, \\
            \dfrac{P^+g^+e^{\alpha}e^{\alpha^{\vee}}-P^+_1P^-_2g^++P^-_1P^+_2g^-+P^-g^-e^{-\alpha}e^{-\alpha^{\vee}}}{\Eu(\alpha)\Eu(-\alpha)} &     \text{if } B=\Gm.
        \end{cases}
    \end{equation}
The numerator can be rewritten as $(P^+_1g^+e^{\alpha^{\vee}}-P^-_1g^-)(P^+_2-P^-_2e^{-\alpha})$ in the $B=\Ga$ case and similarly in the $B=\Gm$ case.  Both factors vanish at the point $e^{\alpha^{\vee}}=1$ above the origin of $\base{G}$, so the numerator lies in $J$, and thus $r_{\alpha^{\vee}}$ is a regular function on $\nonabbl{\tG,\bfN}/W$.  We let $r_{k\alpha^{\vee}}(g)=r_{\alpha^{\vee}}^{k-1}r_{\alpha^{\vee}}(g)$ for $k\in \Z_{> 0}$.
  
These formulas are easier to understand by considering the case $G=GL(2)$.  Every representation of $SL(2)$ can be extended to a representation of $GL(2)$, but this extension is not unique.  If an irrep is odd-dimensional, then we can use the unique extension where the center of $GL(2)$ acts trivially.  On the other hand, if an irrep is even-dimensional, then we can use the extension where the center acts with weight $-1$.  That is, the weight $\varphi_j=m_j\alpha/2$ becomes the $GL(2)$ weight $\tilde\varphi_j=(\lfloor m_j/2\rfloor,\lfloor -m_j/2\rfloor)$.

This gives a lift of our factorization from before:
\[Q_1^{+} = \prod_{m_j>0} \Eu(\tilde{\varphi}_j)^{\lfloor m_j/2 \rfloor}, \qquad Q_2^{+} = \prod_{m_j>0} \Eu(\tilde{\varphi}_j)^{\lceil m_j/2 \rceil}, \qquad Q_1^-=s\cdot Q_1^+, \qquad Q_2^-=s\cdot Q_2^+ \]
We can extend the formula \cref{eq:PGL2-monopole} to give functions 
\[r_{(1,0)}   =  \begin{cases}
        \dfrac{Q^+_1g^+e^{(1,0)}-Q^-_1g^-e^{(0,1)}}{\Eu(\alpha)} & \text{if } B=\Ga, \\
        \dfrac{Q^+_1g^+e^{(1,0)}-Q^-_1g^-e^{-\alpha}e^{(0,1)}}{\Eu(\alpha)} & \text{if } B=\Gm.
    \end{cases}
    \]\[
  r_{(0,-1)}   =  \begin{cases}
        \dfrac{Q^+_2g^+e^{(0,-1)}-Q^-_2g^-e^{(-1,0)}}{\Eu(\alpha)} & \text{if } B=\Ga, \\
        \dfrac{Q^+_2g^+e^{(0,-1)}-Q^-_2g^-e^{-\alpha}e^{(-1,0)}}{\Eu(\alpha)} & \text{if } B=\Gm.
    \end{cases}
    \]
We can then rewrite the formulas for $r_{\alpha^{\vee}}$ in the $SL(2)$ case as a product $r_{(1,0)}r_{(0,-1)}$, restricted to $\base{SL(2)}\subset \base{GL(2)}$.    

Let $\{g_b\}_{b\in B_+\cup B_-}$ be a basis of $\K[\base{\tT}]$ such that $g_b$ is $W$-invariant if $b\in B_+$ and $W$-anti-invariant if $b\in B_-$. 
\begin{Lemma}\label{lem:rank-1-BFN}
    The BFN Coulomb branch (rational or $K$-theoretic) for $G=SL(2)$ or $PGL(2)$ has a basis given by the functions $g$ for $g\in B_+$ and the operators $r_{\gamma}(g)$ for non-zero dominant coweights $\gamma$ of $G$ and $g\in B_+\cup B_-$.
\end{Lemma}
\begin{proof}
    Using the Atiyah--Bott localization formula to write this class as a sum of classes supported on the fixed points, we see that the image of $r_{\alpha^{\vee}/2}^{\BFN}(g)$ in $\K[\base{\tT}\times \Tvee]$ is given by the formula in \cref{eq:PGL2-monopole}; the denominator is the Euler class of the normal bundle to the fixed point, and the numerator is the contribution of the fiber over the fixed points.  This gives that the image is $\frac{P^+g^+e^{\alpha^{\vee}/2}}{\Eu(\alpha)}+\frac{P^-g^-e^{-\alpha^{\vee}/2}}{\Eu(-\alpha)}$, which simplifies to \cref{eq:PGL2-monopole}; the difference between the cases $B=\Ga$ and $B=\Gm$ is simply the ratio between $\Eu(\alpha)$ and $\Eu(-\alpha)$.  This confirms that the induced map $\K[\base{\tT}]\to \K[\tMB(\tG,G,\bfN)]$ is an isomorphism onto the corresponding piece of the associated graded from \cite[\S 6(i)]{BFN}.  The formula of \cite[Prop. 6.2]{BFN} shows that the elements $r_{k\alpha^{\vee}/2}(g)=r_{\alpha^{\vee}/2}^{k-1}r_{\alpha^{\vee}/2}(g)$ are dressed monopole operators up to lower-order terms in the $PGL(2)$ case.  The filtration of \cite[\S 6(i)]{BFN} shows that these form a basis.

    The same argument shows that $(r_{(1,0)}r_{(0,-1)})^k$ with dressing polynomials give a basis of the degree 0 part of  $\K[\tMB(GL(2),GL(2),\bfN)]$ under the grading induced by the quotient $GL(2)/SL(2)$.  Thus, writing $\tMB(SL(2), SL(2),\bfN)$ as the Hamiltonian reduction by the Langlands dual of this torus, as in \cite[Prop. 3.18]{BFN}, we see that the elements $r_{k\alpha^{\vee}}(g)$ give a basis of $\K[\tMB(SL(2), SL(2),\bfN)]$. 
\end{proof}

\begin{Lemma}\label{lem:rank-1-ours}
    If $G, \tG\in \{SL(2),PGL(2)\}$, then $\tM(\tG,G,\bfN)$ is isomorphic to the BFN Coulomb branch if $B=\Ga$ and the $K$-theoretic Coulomb branch if $B=\Gm$.
\end{Lemma}
Note that in this case, since all weights come in pairs $\pm \varphi_{ij}$, we have $D^+_{\ell_{\alpha}}=D^-_{\ell_{\alpha}}=D_{\ell_{\alpha}}/2$.
\begin{proof}
{\bf Comparison map:}    By definition $\tM(\tG,G,\bfN)$ is an affine variety, so we must compute the regular functions on $\nonabbl{\tG,\bfN}$.  Since this lemma only concerns the cases $B=\Ga$ and $B=\Gm$, the Euler class $\Eu(\alpha)$ is an honest function on $\base{\tG}$.  Thus, any regular function on $\nonabbl{\tG,\bfN}$ can be written as a quotient $f/\Eu(\alpha)^n$ for some $n\geq 0$ and some function $f$ on $\toricbl{\tT,\bfN}$ which lies in $J_{\alpha}^{\lceil n/\epsilon_{\al}\rceil}$.

By \cref{lem:rank-1-BFN}, we have a ring homomorphism
\[\K[\tMB(\tG,G,\bfN)]\to \K[\tM(\tG,G,\bfN)], \qquad r_{\gamma}^{\BFN}(g)\mapsto r_{\gamma}(g)\]
for every dominant coweight $\gamma$ of $G$.  

{\bf Surjectivity for $PGL(2)$:} It is clear that this map is injective, so we need only show that it is surjective. 
First consider $G=PGL(2)$, so $\epsilon_{\alpha}=1$.  In this case, 
any section $h\in \soa_{G,\bfN}(U)$ can be written as a Laurent polynomial in $e^{\alpha^{\vee}/2}$ with coefficients in $\mathcal{O}_{\base{\tT}}(U)$.  Consider the maximal $k$ such that $e^{k\alpha^{\vee}/2}$ appears with non-zero coefficient in this Laurent polynomial; by $W$-invariance, this is also the maximal $k$ such that $e^{-k\alpha^{\vee}/2}$ appears with non-zero coefficient.  As discussed above, we can write  $h=f/\Eu(\alpha)^n$ with $f\in J^n$.  We can assume that $f$ is not divisible by $\Eu(\alpha)$; thus, this is only possible if $f$ is divisible by $(e^{\alpha^{\vee}/2}-e^{-\alpha^{\vee}/2})^n$, so $k\geq n$.  

By multiplying numerator and denominator by the function $\Eu(\alpha)^{n-k}$, we can assume that $h=f_2/\Eu(\alpha)^k$.  In this case, we have $f_2\in \soa_{T,\bfN}(U)$,  which implies by \cref{lem:regularity-coefficients} that the coefficient of $e^{\pm k\alpha^{\vee}/2}$ in $f_2$ is divisible by $(P^+)^{k}$, or equivalently by $(P^-)^{k}$; these functions have the same divisor. 

This implies that we can choose $g\in \mathcal{O}_{\base{\tT}}(U)$ so that $h-r_{k\alpha^{\vee}/2}(g)$ has maximal power of $e^{\alpha^{\vee}/2}$ strictly less than $k$.  We can therefore apply induction to show that $h$ is in the image of the map from $\K[\tMB(\tG,G,\bfN)]$.

{\bf Surjectivity for $G=SL(2)$:} The argument for $G=SL(2)$ is effectively the same with some care about factors of 2:  we write $h=f/\Eu(\alpha)^{n}$ for $f\in J^{\lceil n/2\rceil}$ as a Laurent  polynomial in $e^{\alpha^{\vee}}$ with $k$ the maximal degree appearing. We find that if we have chosen $n$ minimal by cancelling out all possible factors of $\Eu(\alpha)$, the numerator must be divisible by $(e^{\alpha^{\vee}}-1)^{n}$. Thus $k\geq n/2$, and we can have $h=f_2/\Eu(\alpha)^{2k}$ with $f_2\in \soa_{T,\bfN}(U)$.  The same argument with leading terms shows we can choose $g\in \mathcal{O}_{\base{\tT}}(U)$ so that $h-r_{k\alpha^{\vee}}(g)$ has lower degree in $e^{\alpha^{\vee}}$.  This completes the proof of surjectivity, and thus the proof of the lemma.
\end{proof}

\begin{Theorem}\label{thm:main-comparison}
    For an arbitrary group $\tG$, the variety $\tM(\tG,G,\bfN)$ is isomorphic to the BFN Coulomb branch if $B=\Ga$ and the $K$-theoretic Coulomb branch if $B=\Gm$.
\end{Theorem}
\begin{proof}
    As usual, we can compare these using the birational map to $(\base{\tT}\times \Tvee)/W$.  
    By \cref{lem:bullet-affinization,lem:bullet-affinization-ours}, it suffices to show that this map induces an isomorphism for the pair $(\Gl,\Nl)$ for $\lambda$ subgeneric, since this will establish the isomorphism \[\tM^{\bullet}(\tG,G,\bfN)\cong \tMB^{\bullet}(\tG,G,\bfN).\]  
    \notation{\vtop{\hbox{$\tM^{\bullet}(\tG,G,\bfN)$,}\hbox{$\tMB^{\bullet}(\tG,G,\bfN)$}}}{Open loci over generic and subgeneric points in the Coulomb branch}

    As discussed in \cref{subsubsec:semi-simple-rank-1}, we can write $\tM(\tGl,\Gl,\Nl)$ as a finite quotient of $\tM(H'\times \tilde{C},H\times C',\Nl)$ using \cref{lem:cover-ours}, and similarly for the BFN Coulomb branch by \cref{lem:cover-BFN}. Since $C'$ is abelian, the factor $\tM(\tilde{C},C',0)$ is isomorphic to $\tMB(\tilde{C},C',0)$ by \cref{thm:abelian-case}.  Thus, it suffices to show that the map $\tM(H',H,\Nl)\to \tMB(H',H,\Nl)$ is an isomorphism.  This follows from \cref{lem:rank-1-ours}.
\end{proof}

\subsubsection{Comparison between rational, $K$-theoretic, and elliptic cases}

One advantage of our construction is that it allows us to compare the rational, $K$-theoretic, and elliptic Coulomb branches more directly. Since the only input is the base $\base{\tG}$ and the divisors $D_{\ell_i}$ and $D_{\alpha}$, we can directly compare the resulting spaces for different choices of $B$ by comparing these data.  

In particular, if we look \'etale locally near a point $b\in \base{\tG}$, then an \'etale neighborhood of $b$ in $\base{\tG}$ is isomorphic to an \'etale neighborhood of $0$ in $\LtT/W_b$, where $W_b$ is the stabilizer of a preimage of $b$ in $\base{\tT}$.  Furthermore, we can define reductive groups $\tG_b, G_b$ with maximal tori $\tT,T$ and Weyl group $W_b$, and a representation $\bfN_b$ of $\tG_b$ by taking the direct sum of the root subalgebras for those $\alpha$ with $b\in D_{\alpha}$ and the weight spaces for those $\varphi$ with $b\in D_{\varphi}$.
\notation{$\tG_b,\bfN_b$}{Local reductive group and representation governing the \'etale model near $b$}\begin{enumerate}
    \item If $B=\Gm$, then $b$ can be thought of as an element of the torus $\tT$, and $G_b=Z_{\tG}(b)^{\circ}$ and $\bfN_b$ is the fixed point set of this torus element.  In the trigonometric case, a vector of weight $n$ contributes to $\bfN_b$ if and only if $b$ has order dividing $n$.
    \item If $B=\elli$, then $b$ can be thought of as a semistable degree $0$ $G$-bundle on $E$, and $\bfN_b$ is the space of global sections of the associated bundle with fiber $\bfN$, and similarly for the Lie algebra $\LtG_b$.  Likewise, in the elliptic case, a vector of weight $n$ contributes to $\bfN_b$ if and only if $b$ is an $n$-torsion point.
\end{enumerate}

In the result below, we use the notation $\tM^{B}(\tG,G,\bfN)$ to indicate the dependence of $\tM(\tG,G,\bfN)$ on the choice of $B$.\notation{$\tM^{B}(\tG,G,\bfN)$}{Coulomb branch constructed from the chosen cohomology theory $B$}\begin{Proposition}\label{prop:local-comparison}
    For every point $b\in \base{\tG}$, there are \'etale neighborhoods $U$ of $b$ and $U_0$ of $0$ in $\LtT/W_b$, such that we have an isomorphism
     \[
    \tM^{B}(\tG,G,\bfN)\times_{\base{\tG}} U\cong \tM^{\Ga}(\tG_b,G_b,\bfN_b)\times_{\LtT/W_b} U_0.
    \]
\end{Proposition}
\begin{proof}
    This follows almost immediately from the definitions.  Given any finite number of smooth hypersurfaces through the point $b$, we can find an \'etale neighborhood $U$ of $b$ and an \'etale map $U\to \LtT_b/W_b$ that sends these hypersurfaces to linear hyperplanes.  This exactly sends the divisors used in the definition of $\tM^{B}(\tG,G,\bfN)$ to the divisors attached to $\tM^{\Ga}(\tG_b,G_b,\bfN_b)$.  Thus, the regularity conditions defining $\tM^{B}(\tG,G,\bfN)$ and $\tM^{\Ga}(\tG_b,G_b,\bfN_b)$ are the same after pullback to $U$ and $U_0$ respectively, so we obtain the desired isomorphism. 
\end{proof}
\subsubsection{Coulomb branches with symmetrizers}
Finally, we explain how the same perspective extends to the Coulomb branch with symmetrizers introduced by Nakajima and Weekes for non-symmetric Cartan data.  The underlying geometric picture is the same, but the pairing between weights and coweights must be rescaled to reflect the symmetrizers.

Nakajima and Weekes define these Coulomb branches by a construction very similar to the BFN construction, but with a different space of triples, depending on a product decomposition of the gauge group $G=\prod_{k\in \Z_{>0}} G_{k}$ and a direct sum decomposition of the representation $\bfN=\oplus_{k\in \Z_{>0}} \bfN_k$.  The factor $G_k$ acts non-trivially on $\bfN_j$ only if $j$ divides $k$.  As discussed in \cite[\S C(iii)]{nakajimaCoulombBranches2021}, this product decomposition encodes the length data in the quiver attached to a non-symmetric Cartan datum: $G_k$ is the product of the groups $GL_{v_i}$ corresponding to quiver nodes with $k=d_i$, and $\bfN_j$ is the sum of the spaces $\operatorname{Hom}(\C^{v_b},\C^{v_a})$ over arrows $b\to a$ with $\gcd(d_a,d_b)=j$.  

One subtlety discussed in \cite[\S C(ii)]{nakajimaCoulombBranches2021} is that the authors introduce two competing definitions of the Coulomb branch with symmetrizers, one of which is the normalization of the other.  We work with the normal variety defined there.

The resulting Coulomb branch is a variety over $\base{G}=\prod_{k\in \Z_{>0}} \base{G_k}$, and the construction of the space of triples is such that it admits a birational map to $(\base{\tT}\times \Tvee)/W$ as in the symmetric case.  The difference is subtle and is most conveniently expressed by modifying the definition of the pairing between weights and coweights.  Writing $\gamma=(\gamma_k)_k\in X_*(T)$ and $\varphi=(\varphi_k)_k\in X^*(T)$ according to the product decomposition, we set
\[
\langle \gamma, \varphi \rangle_{\mathbf{s}}=\sum_{k\in \Z_{>0}} k\,\langle \gamma_k,\varphi_k\rangle,
\]
where $\mathbf{s}=(k)_{k\in \Z_{>0}}$ records the scaling factors.\notation{$\langle \gamma, \varphi \rangle_{\mathbf{s}}$}{Rescaled pairing used for Coulomb branches with symmetrizers}

In this case, we have to keep track of the weight decomposition not just of $\bfN$, but of the summands $\bfN_n$.  Thus, as before, we choose primitive weights $\eta_i$ and integers $m_{ij}$ such that $\varphi_{ij}=m_{ij}\eta_i$ are the weights of $\bfN$ with multiplicity, but now we choose a weight basis for each $\bfN_n$ separately, so that we have a well-defined invariant $\wtlev_{ij}$ attached to each weight $\varphi_{ij}$, recording which summand the corresponding weight vector lies in.  Note that $\wtlev_{ij}$ divides every $s$ such that $\varphi_{ij}$ is non-trivial on $T_s$, but $\varphi_{ij}$ does {\bf not} necessarily factor through the quotient map to $T_s$ for any integer $s$.
\notation{$\wtlev_{ij}$}{Level for which the weight $\varphi_{ij}$ appears in $\bfN_{\wtlev_{ij}}$}We can then define the rational divisors $D_{\ell_{i}}$ almost as before, but now weighted by $\wtlev_{ij}$:
\[ D_{\ell_i}^{\pm}=\sum_{j:\,\pm m_{ij}>0} \frac{\pm m_{ij}}{\wtlev_{ij}}D_{\varphi_{ij}}.\]
As before, let $k_{i,D}^{\pm}$ be the coefficient of $D$ in $D_{\ell_i}^{\pm}$.  
\begin{Definition}\label{def:regularity-coefficients2}
    The sheaf $\soa_{T,\bfN}(U)$ is the span of the functions $fe^{\gamma}$ such that for every divisor $D$ in $U\subset \base{\tT}$, we have 
    \begin{align}\label{eq:valuation-formulas2}
        \nu_{D,+}(fe^{\gamma}) &= \nu_D(f)-\langle \gamma,\eta_i\rangle_{\mathbf{s}} k^+_{i,D} \geq 0,\\
        \nu_{D,-}(fe^{\gamma}) &= \nu_D(f)+\langle \gamma,\eta_i\rangle_{\mathbf{s}} k^-_{i,D} \geq 0.
    \end{align}
\end{Definition}
Note that while $k^{\pm}_{i,D}$ are rational numbers, the pairing $\langle \gamma,\eta_i\rangle_{\mathbf{s}}$ is an integer multiple of $\wtlev_{ij}$ for any weight $\varphi_{ij}$ that appears, so the right-hand sides of these formulas are integers.  We can also rewrite this in terms of a compactification like $\toricbl{\tT,\bfN}$ as before, but this introduces some annoying book-keeping, since one has to account for whether $\eta_i$ remains primitive after scaling its $\LT_k$ component by $k$.  The rest of the construction proceeds as before, and we obtain a variety $\tM(\tG,G,\bfN)$ with a birational map to $(\base{\tT}\times \Tvee)/W$.

First, we note that we have the analogue of \cite[Rmk. C.2]{nakajimaCoulombBranches2021} in this case:
\begin{Lemma}\label{lem:one-k}
    If $G=G_k$, then the variety $\tM(\tG,G,\bfN)$ defined as above is isomorphic to the usual Coulomb branch for the group $G$ and representation $\bigoplus_{q}\bfN_q^{\oplus k/q}$.  
\end{Lemma}
\begin{proof}
    This is a straightforward comparison of \cref{lem:regularity-coefficients} and \cref{def:regularity-coefficients2}.  In this case $\langle \gamma,\eta_i\rangle_{\mathbf{s}}=k\langle \gamma,\eta_i\rangle$, so indeed the contribution of weights in $\bfN_q$ is exactly $k/q$ times the contribution they would make with the usual Coulomb branch construction.
\end{proof}

\begin{Theorem}\label{thm:symmetrizer-comparison}
    When $B=\Ga$, the variety $\tM(\tG,G,\bfN)$ defined as above is isomorphic to the Coulomb branch with symmetrizers defined by Nakajima and Weekes in \cite[\S C(ii)]{nakajimaCoulombBranches2021}.
\end{Theorem}
\begin{proof}
    The proof is essentially the same as in the symmetric case.  The only difference is that we consider loop groups and representations where we vary the power of the parameter $t$ that appears:
    \[G_{\mathcal{O}} = \prod_{k\in \Z_{>0}} G_k\llbracket t^{k} \rrbracket, \qquad \bfN_{\mathcal{O}}=\bigoplus_{k\in \Z_{>0}} \bfN_k\llbracket t^{k} \rrbracket.\]
    In particular, when we write $t^{\gamma}$ for $\gamma\in X_*(T)$, we scale the cocharacters of $T_k$ so that this element lies in $G_{\mathcal{O}}$.
    For $G$ abelian, we can apply the same Euler class formula, giving the monopole operator defined topologically as 
    \[r^{\BFN}_{\gamma}= e^{\gamma} \cdot \Eu(\bfN_{\mathcal{O}}/ t^{\gamma}\bfN_{\mathcal{O}})=e^{\gamma} \cdot \prod_{i,j} \Eu(\varphi_{ij})^{\langle \gamma, \eta_i \rangle_{\mathbf{s}}/ \wtlev_{ij} }.\]
    This satisfies the conditions in \cref{def:regularity-coefficients2}, so we obtain a map from the Coulomb branch with symmetrizers defined by Nakajima and Weekes to $\tM(\tG,G,\bfN)$.  The same localization argument as in the symmetric case shows that this is an isomorphism after pullback to $\base{\tT}$, so it is an isomorphism.

    The extension to the non-abelian case proceeds exactly as before.  Since every root of $G$ must factor through the projection to $G_k$ for some $k$, for a generic cocharacter on which $\alpha$ vanishes, the action on $\Nl$ factors through $G_k$ for some $k$.  Thus by \cref{lem:one-k}, we can reduce to the usual Coulomb branch after passing to an open subset of $\base{\tG}$ where we remove $D_{\ell_i}$ for all other lines of weights and $D_{\alpha'}$ for all other roots.  
\end{proof}

For the rest of the paper, we will let $\tM$ denote this Coulomb branch with symmetrizers.  In most later applications, however, we return to the ordinary case $G=G_1$ and $\bfN=\bfN_1$, so the additional symmetrizers disappear.  Note also that \cref{lem:slice-ours} and \cref{lem:cover-ours} continue to hold in the setting with symmetrizers.

\subsection{Comparison with the universal centralizer}

One interesting result of this framework is that it clarifies the relationship between the $K$-theoretic Coulomb branch with trivial matter and the universal centralizer.  In \cite{bezrukavnikovEquivariantHomology2005}, Bezrukavnikov, Finkelberg, and Mirkovi\'c show that 
pure Coulomb branches are closely related to the universal centralizer of a simply-laced and simply-connected group $G_{\sc}$, but stating this result correctly requires some care about adjoint and simply-connected groups.  
Since we can consider the variety $\tM(\tG,G,\mathbf{0})$ for $\tG,G$ being $G_{\sc}$ or $G_{\ad}$ or any intermediate covers, we can state these results in a more convenient way.  Note that because $G$ is simply-laced, we have the coincidence that for any quotient $G$ of $G_{\sc}$, the Langlands dual $G^{\vee}$ is a quotient of $G_{\sc}$ as well, with the Cartan matrix inducing an isomorphism $(G_{\ad})^{\vee}\cong G_{\sc}$. 

Thus, we can consider the universal centralizer variety \[\mathfrak{Z}^{G}_{\tG} = \{ (g,\gamma) \in G \times \tG \mid \Ad_g(\gamma)=\gamma, \gamma \text{ regular} \}/\!\!/ G_{\ad}.\]
\begin{Lemma}\label{lem:universal-centralizer}
    If $G$ is simply-laced and simply-connected, then $\tM(\tG,G^{\vee},\mathbf{0})$ is isomorphic to the universal centralizer $\mathfrak{Z}^{G}_{\tG}$.
\end{Lemma}
\begin{proof}
    By \cite[Prop. 2.8]{bezrukavnikovEquivariantHomology2005}, the universal centralizer $\mathfrak{Z}^{G}_{\tG}$ is identified with the blowup $\mathfrak{B}^{G}_{\tG}$ of $\tT\times T$ along the subscheme defined by the union of the subschemes defined by $(\alpha(x),\alpha(x))=(1,1)$ in $\tT\times T$ for all roots $\alpha$, with the strict transform of the divisors $D_{\alpha}\times T$ removed.  Since $\mathfrak{Z}^{G}_{\tG}$ is smooth, it is normal, and its ring of functions is the ring of $W$-invariant rational functions on $\tT\times T$ which are regular away from $D_{\alpha}\times T$ and are regular on the exceptional divisor of the blowup lying over each component of $D_{\alpha}\times D_{\alpha}$.
    
    On the other hand, the space $\nonabbl{\tG,\mathbf{0},\alpha}$ is defined by the blowup at one of these subschemes $D_{\alpha}\times D_{\alpha}$ with the same strict transform removed.  Thus, by \cref{def:Coulomb-branch-nonabelian}, the function ring of  $\tM(\tG,G^{\vee},\mathbf{0})$ is defined by the same regularity conditions, and we obtain the desired isomorphism.
\end{proof}

Note that this allows us to see the difference between the universal centralizer for a non-simply-laced group $G$ and the $K$-theoretic Coulomb branch for $G^{\vee}$ more clearly.  These are defined by blowups of $T\times T$ and $T^{\vee}\times T$ respectively and there is no natural isomorphism between $T$ and $T^{\vee}$ that matches roots, even when we play the game above of passing to a cover or a quotient.  We can see this most clearly by considering the isogeny $T_{\sc} \to T^{\vee}_{\ad}$, where a root pulls back to a multiple of a root, but roots with different lengths pull back to different multiples, so the varieties do not match.   

\section{Functoriality}
\label{sec:functoriality}
\subsection{A preparatory lemma}

The geometric description of $\tM(\tG,G,\bfN)$ in \cref{sec:new-construction} allows us to prove the functoriality results of \cite{gannonFunctorialityCoulomb2025} in full generality, avoiding the technical gluability hypothesis on the homomorphism $H\to G$ that was required there.  The key point is that the Coulomb branch is now described by explicit regularity conditions along the weight and root divisors, together with the slice and cover results proved in \cref{subsubsec:slices-codim-1,subsubsec:finite-covers}.

Recall that $\tM(G)=\tM(\tG,G,\mathbf{0})$ is the Coulomb branch with trivial matter, and write $\soa_G=\soa_{G,\mathbf{0}}$ for its sheaf of algebras on $\base{\tG}$, as well as $\mathcal{O}=\mathcal{O}_{\base{\tG}}$ for the structure sheaf of $\base{\tG}$. 
\notation{$\tM(G),\soa_G$}{Coulomb branch with trivial matter and its sheaf of algebras}

The group scheme $\tM(G)$ acts on $\tM(\tG,G,\bfN)$, which is expressed algebraically by a comodule map of sheaves
\[\soa_{G,\bfN}\to \soa_{G,\bfN}\otimes_{\mathcal{O}}\soa_G.\]
We will need to consider group cohomology in the category of $\tM(G)$-modules, that is, the category of coherent sheaves $\mathcal{F}$ on $\base{\tG}$ equipped with a comodule structure $\mathcal{F}\to \mathcal{F}\otimes_{\mathcal{O}}\soa_G$.  We calculate this cohomology using the usual Cartan-Eilenberg complex
\[ \mathcal{F} \to \mathcal{F}\otimes_{\mathcal{O}}\soa_G\to \mathcal{F}\otimes_{\mathcal{O}}\soa_G\otimes_{\mathcal{O}}\soa_G\to \cdots
\]  
The first step is a vanishing result for this group cohomology, which is what will later show that the categorical quotient commutes with base change:
\begin{Lemma}\label{lem:vanishing-cohomology}
    For any $\tM(G)$-space $X$ which is affine as a $\base{\tG}$-scheme, the algebraic group cohomology sheaf on $\base{\tG}$ satisfies \[H^i_{\mathcal{O}}(\tM(G),\cO(\tM(\tG,G,\bfN)\times_{\base{\tG}} X))=0\text{ for all } i>0.\] 
\end{Lemma}
\begin{proof}
    Since the claim is local on $\base{\tG}$, it is enough to work on an affine open subset $U\subset \base{\tG}$.  We can then replace $\soa_G(U)$ and $\soa_{G,\bfN}(U)$ by their associated graded for the filtration induced by \cite[Prop. 6.1]{BFN}.  Either by comparison with the abelian case or directly from the geometric definition, one checks that the group structure on $\tM(G)$ and the comodule structure on $\soa_{G,\bfN}$ are compatible with this filtration, in the sense that the image of $\soa_{G,\bfN}(U)_{\leq \lambda}$ under comultiplication lies in $\soa_{G,\bfN}(U)_{\leq \lambda}\otimes \soa_G(U)_{\leq \lambda}$.  

    We claim that we have an isomorphism $\operatorname{gr}\soa_{G,\bfN}(U)\cong \operatorname{gr}\soa_G(U)$ as a $\operatorname{gr}(\soa_G(U))$-comodule.  This isomorphism is induced by sending dressed monopole operators for a coweight $\mu$ to the same monopole operators without matter.  This commutes with the comodule action since it is induced on elements of degree $\mu$ by multiplication by a rational section of the commutative group scheme $\tM(G)$, which depends on $\mu$:
    \[\epsilon_{\mathbf{N},\mu}=\prod_{\langle \mu, \varphi_i\rangle >0} \varphi_i^{\vee}(\Eu(\varphi_i))^{\langle \mu, \varphi_i\rangle }.\]

    Since the higher group cohomology of $\cO(\tM(G)_U\times_U X_U)$ vanishes by the usual homotopy on the bar complex, the same is true of $H^i(\operatorname{gr}\tM(G)_U,\operatorname{gr}\soa_{G,\bfN}(U)\otimes_{\cO(U)}\K[X_U])$.  By semicontinuity, this establishes the same vanishing before passing to the associated graded, as desired.
\end{proof}

With this vanishing in hand, we can see that the categorical quotient by $\tM(G)$ is compatible with base change:
\begin{Lemma}\label{lem:base-change}
    Let $X$ be a $\tM(G)$-space that is affine over $\base{\tG}$, and let $S$ be an arbitrary scheme with a morphism $\kappa\colon S \to \base{\tG}$.  Then we have an isomorphism 
    \[(\tM(\tG,G,\bfN)\times_{\base{\tG}}^{\tM(G)} X)_S\cong \tM(\tG,G,\bfN)_S\times_{S}^{\tM(G)_S} X_S.\]
\end{Lemma}
\begin{proof}
Consider the functor from $\tM(G)$-equivariant sheaves on $\tM(\tG,G,\bfN)\times_{\base{\tG}} X$ to $D^b\mathsf{Coh}(S)$ given by 
\[A(\mathcal{F})=\mathbb{L}\kappa^*(\mathbb{R}\Hom_{\tM(G)}(\cO_{\base{\tG}}, \pi_*\mathcal{F}))\]
Letting $\bar{\kappa}$ and $\bar\pi$ be the base changes of these maps, this is the same as \[A(\mathcal{F})=\mathbb{R}\Hom_{\tM(G)_S}(\cO_S, \bar{\pi}_*\mathbb{L}\bar{\kappa}^*(\mathcal{F})).\]
These two presentations yield two spectral sequences.  In the former case, we obtain the $E^2$ page
\[\mathbb{L}_p\kappa^*(H^q(\tM(G),\cO(\tM(\tG,G,\bfN)\times_{\base{\tG}} X)))=\begin{cases}
    \mathbb{L}_p\kappa^*(\cO(\tM(\tG,G,\bfN)\times^{\tM(G)}_{\base{\tG}} X))& q=0\\
    0 & q\neq 0 
\end{cases}\]
with these terms in homological degree $q-p$.
In the latter case, we obtain
\[(H^p(\tM(G)_S,\mathbb{L}_q\kappa^*\cO(\tM(\tG,G,\bfN)\times_{\base{\tG}} X)))=\begin{cases} H^p(\tM(G)_S,\cO(\tM(\tG,G,\bfN)_S\times_{S} X_S)) & q=0\\
0 & q\neq 0
\end{cases}\]
with these terms in homological degree $p-q$.
Comparing the terms in homological degree $0$ gives the desired isomorphism.  
\end{proof}

\subsection{General functoriality}
Let $\tH, H$ be another pair of reductive groups satisfying the conditions of \cref{subsec:setup-and-notation}. 
Let $\mapFromHToG: \tH \to \tG$ be a map of reductive groups inducing a map $H\to G$; we will occasionally refer to the induced map of pairs of algebraic groups as\notation{$\mapFromHToG$}{Map of pairs used in the functoriality comparison}\begin{equation}\label{eq:map-of-pairs}(\tH, H) \xrightarrow{\mapFromHToG} (\tG, G).\end{equation}
Our goal is to compare the Coulomb branch for $(\tH,H)$ with the base change of $\tM(\tG,G,\bfN)$ along $\overline{\mapFromHToG}$.  The slice and cover results of \cref{subsubsec:slices-codim-1,subsubsec:finite-covers} reduce this comparison to the same local models along a single weight or root divisor that controlled the construction of the Coulomb branch itself.  The proof has two steps: we first treat the massive case, where codimension-$1$ analysis is sufficient, and then deduce the general case using a deformation argument.  

We choose a maximal torus $\tT_{H}\subset \tH$ and let $T_H=H\cap \tT_{H}$.  To avoid confusion, we denote the tori in $\tG,G$ chosen in \cite[\S 1.1]{gannonFunctorialityCoulomb2025} by $\tT_G, T_G$.  Without loss of generality, we can assume that $\mapFromHToG(\tT_{H})\subset \tT_G$. We have an induced map of affine varieties \(\overline{\mapFromHToG}\colon \base{\tH}  \xrightarrow{} \base{\tG}\).  Following \cite{gannonFunctorialityCoulomb2025}, this also induces a map  
\begin{equation}\label{eq:definition-of-tildef}\tilde{f}_{\mapFromHToG, \mathbf{N}}: \tM(H) \times_{\base{\tG}}^{\tM(G)} \tM(\tG,G, \mathbf{N}) \cong \tM(H) \times_{\base{\tH}}^{\tM(G)_{\base{\tH}}} \tM(\tG,G, \mathbf{N})_{\base{\tH}} \to \tM(\tH,H, \mathbf{N}).\end{equation}  We claim that this map is an isomorphism.\notation{$\tilde{f}_{\mapFromHToG,\mathbf{N}}$}{Functoriality morphism comparing reduction and restriction of groups}

 \begin{Theorem} \label{thm:restriction-groups-hamiltonian-reduction}
For every map of pairs as in \eqref{eq:map-of-pairs}, the map $\tilde{f}_{\mapFromHToG, \mathbf{N}}$ is an isomorphism. 
\end{Theorem}
An important special case of this result is \cite[Prop. 3.18]{BFN}:

\begin{Lemma}\label{lem:flavor-case}
    \cref{thm:restriction-groups-hamiltonian-reduction} holds when $H\subset G$ is normal and $G/H$ is a torus.  
\end{Lemma}

We call the data $(\tG,G,\bfN)$ \emph{massive} if $\tG$ contains a central $\Gm$ which acts by a positive weight on all vectors in $\bfN$.  The easiest way to achieve this is to take $\tG=\Gmsc\times G$ where $\Gmsc$ acts on $\bfN$ by scaling.

\begin{Lemma}\label{lem:normality-massive}
    Assume that the triple $(\tH,H,\bfN)$ is massive.
    The varieties $\tM(\tG,G,\bfN)_{\base{\tH}}$, $\tM(\tG,G,\bfN)\times_{\base{\tG}}\tM(H)$   and $\tM(\tG,G,\bfN)\times_{\base{\tG}}^{\tM(G)}\tM(H)$ 
    are normal.  
\end{Lemma}
\begin{proof}
    First, note that $\tM(G,G,\bfN)$ is Cohen--Macaulay, since it has symplectic singularities by \cite{bellamyCoulombBranches2023}.  The same argument shows that $\tM(\tG,G,\bfN)$ is a flat family over the smooth variety $\base{\tG/G'}$ with Cohen--Macaulay fibers, and is therefore itself Cohen--Macaulay.  By \cite[Lem. 37.22.6]{stacks-project}, the map $\tM(\tG,G,\bfN)\to \base{\tG}$ is Cohen--Macaulay.  By flat base change, $\tM(\tG,G,\bfN)_{\base{\tH}}$ is also Cohen--Macaulay, and 
    $\tM(\tG,G,\bfN)\times_{\base{\tG}}\tM(H)$ is a complete intersection in this variety and so is itself Cohen--Macaulay.

    Since all these varieties are Cohen--Macaulay, they are normal if and only if their singular locus has codimension $\geq 2$.  This follows immediately from the facts:
    \begin{enumerate}
        \item  The smooth locus of $\tM(\tG,G,\bfN)$ has dense intersection with each fiber of the map to $\base{\tG}$.
        \item The generic fiber over $\base{\tH}$ is smooth: by the massive assumption, a generic cocharacter $\lambda$ of $\tH$ has trivial fixed space in $\bfN$, and so the fiber is isomorphic to the fiber over the origin in $\tM(G^{\lambda},G^{\lambda},\mathbf{0})$, which is smooth.
    \end{enumerate}

    Finally, the categorical quotient of an integral normal variety is always integral and normal, so $\tM(\tG,G,\bfN)\times_{\base{\tG}}^{\tM(G)}\tM(H)$ is normal as well.
\end{proof}

\begin{Lemma}\label{lem:massive-case}
    \cref{thm:restriction-groups-hamiltonian-reduction} holds if $\tH$ is massive.
\end{Lemma}
\begin{proof}
The map $\tilde{f}_{\mapFromHToG, \mathbf{N}}$ is an isomorphism over $\base{\tH}^{\circ}$, and thus is birational.  By \cref{lem:normality-massive}, both the source and target of this map are normal.  
To prove that this birational map is an isomorphism, it is therefore enough to control its behavior in codimension $1$: by Zariski's Main Theorem, the restriction of $\tilde{f}_{\mapFromHToG, \mathbf{N}}$ to its quasi-finite locus is the inclusion of an open subscheme, so it suffices to show that the complement of the quasi-finite locus has codimension $\geq 2$ in the source and that its image has codimension $\geq 2$ in the target.

Equivalently, we must show that every Weil divisor in $\tM(\tH,H,\bfN)$ is the image of a Weil divisor in $\tM(\tG,G,\bfN)\times_{\base{\tG}}^{\tM(G)}\tM(H)$, and that every Weil divisor in $\tM(\tG,G,\bfN)\times_{\base{\tG}}^{\tM(G)}\tM(H)$ is isolated in its fiber.  

 Accordingly, we have lines $\ell_i$ for $i=1,\dots,p$ in $X^*(H)$ spanned by the weights of $\bfN$ restricted to $H$, and lines $\ell_{\alpha}$ spanned by the roots of $H$.  We have corresponding divisors $D_{\ell_i}$ and $D_{\ell_{\alpha}}$ in $\base{\tH}$ which have no components in common since $\tH$ is massive.  

 Let $U_i=\base{\tH}\setminus D_{\ell_i}$ and $U_{\alpha}=\base{\tH}\setminus D_{\ell_{\alpha}}$.  Let \[V_i=(\bigcap_{j\neq i}U_j)\cap (\bigcap_{\al}U_{\alpha})\qquad\text{and }\qquad V_{\alpha}=(\bigcap_i U_i)\cap (\bigcap_{\beta\neq \alpha}U_{\beta}).\]  Write $\lambda_i$ for a cocharacter of $T$ lying on the hyperplane $\ell_i^{\perp}$ but on no other weight or root hyperplane, so that the fixed space $\bfN^{\lambda_i}$ is the sum of the weight spaces of $\bfN$ whose weights lie on the line $\ell_i$.

On the open sets $V_i$, we have \[\tM(\tG,G,\bfN)_{V_i}\cong \tM(\tT,T,\bfN^{\lambda_i})_{V_i},\qquad (\tM(\tG,G,\bfN)\times_{\base{\tG}}\tM(H))_{V_i}\cong \tM(\tT,T,\bfN^{\lambda_i})\times_{V_i}\tM(T_H).\]  
Passing to the corresponding categorical quotient, we obtain the local form of \cref{thm:restriction-groups-hamiltonian-reduction} over $V_i$.  Hence the theorem holds there by \cref{lem:flavor-case}.  In particular, any Weil divisor lying over $V_i$ in the source lies in the quasi-finite locus, and any Weil divisor lying over $V_i$ in the target is in the image.  

    On the other hand, we have
    \[\tM(\tG,G,\bfN)_{V_{\al}}\cong \tM(G)_{V_{\al}},\qquad (\tM(\tG,G,\bfN)\times_{\base{\tG}}\tM(H))_{V_{\al}}\cong \tM(G)\times_{V_{\al}}\tM(H).\]
    Thus, after base change to $V_{\al}$, \cref{thm:restriction-groups-hamiltonian-reduction} reduces to the pure case and therefore holds.  Passing to the corresponding categorical quotient, we conclude exactly as above that any Weil divisor lying over $V_{\al}$ in the source is in the quasi-finite locus, and any Weil divisor lying over $V_{\al}$ in the target is in the image.  

This completes the proof.
\end{proof}

 \begin{proof}[Proof of \cref{thm:restriction-groups-hamiltonian-reduction}]
     By \cref{lem:massive-case}, we can assume that the result is true for the obvious map of pairs $(\tH\times\Gmsc,H)\to(\tG\times\Gmsc,G)$, so we can consider the base change of this isomorphism.  
     This induces the desired isomorphism by \cref{lem:base-change}.
     \end{proof}

\section{Transverse Hilbert schemes}
\label{sec:transverse-hilbert-schemes}

With the construction and functoriality results in hand, we turn to a comparison with transverse Hilbert schemes.  The blowup construction of the Coulomb branch in \cref{def:Coulomb-branch-nonabelian} has a similar flavor to the transverse $W$-Hilbert scheme construction of Bielawski and Foscolo \cite{bielawskiHypertoricVarieties2023}.  In the case $G=PGL(2)$, the two constructions agree, but in general the transverse $W$-Hilbert scheme does not itself recover the Coulomb branch.  Even if we avoid gauge-group factors isomorphic to $Sp(2n)$, the correct statement in general requires replacing the abelian Coulomb branch $\tM(\tT,T,\bfN)$ with a proper open subset $\tM^{\bullet}(\tT,T,\bfN)$.

\subsection{Relation to nilHecke algebras: rank 1}

We begin with the rank-$1$ case, which already contains the key nilHecke calculation behind the comparison.

Although we found the work of Bielawski and Foscolo \cite{bielawskiHypertoricVarieties2023} quite inspiring, some of the global comparison statements in that paper need correction.  We therefore begin by isolating the transverse-Hilbert-scheme calculation itself in a form that can later be compared directly with the Coulomb branch.

Let $W=\{1,s\}$ act on a complex vector space $V$ with $s$ acting as a reflection; let $\alpha$ be a linear function satisfying $s\alpha=-\alpha$.   Consider an affine variety $X$ equipped with a $W$-action and a dominant equivariant map $\pi\colon X \to V$.  Let  $ \K[X]^{W}$ be the invariants of $\K[X]$ and $\K[X]^{W,-}=\{ g\in \K[X] \mid sg=-g\}$.\notation{$\K[X]^{W,-}$}{$W$-anti-invariant functions on $X$}

Following \cite{bielawskiHypertoricVarieties2023}, we consider several variations on the full $W$-Hilbert scheme $\WHilb{X}$ parametrizing $W$-invariant subschemes $Y$ of $X$ such that  $\mathcal{O}(Y)\cong \K W$ as a $W$-module.  Such a subscheme has a distinguished invariant vector $\bar{1}$, the image of the constant function, and a unique anti-invariant line.\notation{$\WHilb{X}$}{Full $W$-Hilbert scheme of $X$}
\begin{Definition}\hfill
    \begin{enumerate}
        \item The full $W$-Hilbert scheme $\WHilb{X}$ is the moduli space of $W$-invariant subschemes $Y\subset X$ such that $\mathcal{O}(Y)\cong \K W$ as a $W$-module.
        \item The closed subset $\HilbW{X}\subset \WHilb{X}$ is the closure of the locus of free orbits.
        \item The transverse $W$-Hilbert scheme $\tWHilb{\pi}{X}\subset \WHilb{X}$ is the open subset that parametrizes subschemes $Y$ such that $\pi(Y)$ is a subscheme of $V$, or equivalently that $\mathcal{O}(Y)$ is a cyclically generated module over $\K[V]$.
         \item The transverse $W$-Hilbert scheme $\tHilb{\pi}{X}\subset \HilbW{X}$ is the intersection \[\tHilb{\pi}{X}=\tWHilb{\pi}{X}\cap \HilbW{X}.\]
\notation{$\HilbW{X}$}{Closure of the free $W$-orbits inside $\WHilb{X}$}\notation{$\tWHilb{\pi}{X}$}{Transverse open subset of the full $W$-Hilbert scheme}\notation{$\tHilb{\pi}{X}$}{Transverse open subset of $\HilbW{X}$}\end{enumerate}
\end{Definition}
The two transverse loci $\tWHilb{\pi}{X}$ and $\tHilb{\pi}{X}$ are each the complement of the zeros of the section $\bar{1}\wedge \bar{\alpha}$ of the second wedge power of the corresponding tautological bundle $\mathcal{T}$.
In particular, if $g\in \K[X]$ is any function satisfying $sg=-g$, then \[\Psi(g)=\frac{\bar{1}\wedge \bar{g}}{\bar{1}\wedge \bar{\alpha}}\] is a well-defined function on $\tWHilb{\pi}{X}$, and hence on $\tHilb{\pi}{X}$.  Note that for $f\in \K[X]^W$, we have 
\[\Psi(f\alpha)=\frac{\bar{1}\wedge \bar{f\alpha}}{\bar{1}\wedge \bar{\alpha}}=f, \qquad \Psi(gf)=\frac{\bar{1}\wedge \bar{gf}}{\bar{1}\wedge \bar{\alpha}}=f\Psi(g).\]
\notation{$\Psi(g)$}{Tautological divided-difference function on the transverse $W$-Hilbert schemes}

\begin{Theorem}\label{thm:alpha-inverse}
The functions on $\tWHilb{\pi}{X}$ are generated by $\K[X]^W$ and symbols $\Psi(g)$ for $g\in \K[X]^{W,-}$ subject to the relations $\Psi(f\alpha)=f$ and $\Psi(gf)=f\Psi(g)$ for $f\in \K[X]^W$.

The functions on $\tHilb{\pi}{X}$ are the image of the homomorphism $\K[\tWHilb{\pi}{X}] \to \K(X)^W$ sending $\Psi(g)\mapsto g/\alpha$ for $g\in \K[X]^{W,-}$.   That is, $\K[\tHilb{\pi}{X}]$ is the subring of $\K(X)^W$ generated by $\K[X]^W$ and $\frac{1}{\alpha}\K[X]^{W,-}$.
\end{Theorem}
\begin{proof}
We have already verified that these relations hold, so it remains to show that the functions on $\tWHilb{\pi}{X}$ admit generators of the stated form, and that the displayed relations generate all others.

    Let $a_1,\dots, a_m, b_1,\dots, b_n$ be generators of $\K[X]$ over $\K$, with $a_i$ invariant and $b_j$ anti-invariant for all $i,j$.  A point of $\tWHilb{\pi}{X}$ corresponds to a subscheme of $X$ for which the images $\{\bar{1},\bar{\alpha}\}$ form a basis.  
    
The functions on $\tWHilb{\pi}{X}$ are generated by $A_i=\frac{\bar{a}_i\wedge \bar{\alpha}}{\bar{1}\wedge \bar{\alpha}}$ and $B_j=\frac{\bar{1}\wedge \bar{b}_j}{\bar{1}\wedge \bar{\alpha}}$;  we will give a more careful argument for this fact in more generality in \cref{prop:Hilb-gen}.  We can think of $\K[\tWHilb{\pi}{X}]$ as the coordinate ring of an $X/W$-variety by identifying $A_i$ and $a_i$;  the functions on $\tWHilb{\pi}{X}$ are thus generated by the functions $B_i$ over $\K[X]^W$.  Our claim is that the only relations these generators satisfy are: 
\begin{enumerate}
    \item if $c_1,\dots, c_n\in \K[X]^W$ satisfy $\sum_i c_i b_i=0$, then $\sum_i c_i B_i=0$. 
    \item for all $i,j$ we have $\alpha^2 B_i B_j=b_ib_j$.
\end{enumerate} 
Note that these relations can be deduced from those in the statement of the theorem, since 
\[\sum_i c_i B_i=\Psi(\sum_i c_i b_i), \qquad  \alpha^2 B_i B_j=\Psi(\alpha^2b_i)\Psi(b_j)=\alpha b_i \Psi(b_j)=\Psi(\alpha b_ib_j)=b_ib_j.\]

To prove the claim, consider a $\K$-algebra $Q$ and a $Q$-point $x$ of $X/W$.  Let $\beta_i\in Q$ be scalars satisfying relations (1) and (2): that is, we have $\sum_i c_i(x) \beta_i=0$ whenever $\sum_i c_i b_i=0$, and $\alpha^2(x) \beta_i \beta_j=b_i(x)b_j(x)$ for all $i,j$. 
We claim that $a_i\mapsto a_i(x)$ and $b_i\mapsto \beta_i\alpha$ define a homomorphism to $Q[\alpha]/(\alpha^2-\alpha^2(x))$, whose kernel determines a $Q$-point of $\tWHilb{\pi}{X}$ mapping to $x$.  It is enough to verify the ring-homomorphism property for products of functions that are isotypic for $W$. 
\begin{itemize}
    \item If $f,g\in \K[X]^W$, then $fg\mapsto f(x)g(x)$.
    \item If $f\in \K[X]^W$ and $g\in \K[X]^{W,-}$, then this follows from the relation (1).
    \item If $f,g\in \K[X]^{W,-}$, then this follows from the relation (2).
\end{itemize}

Since the quotient of the free locus is a dense open subset of both $\tHilb{\pi}{X}$ and $X/W$, there is an injective map $\K[\tHilb{\pi}{X}]\hookrightarrow \K(X)^W$ induced by restriction to the free orbits.  Under this homomorphism, the restrictions of $A_i$ and $B_j$ are sent to $a_i$ and $b_j/\alpha$, respectively.  Since $\K[X]^{W,-}=\sum \K[X]^{W}b_i$, this completes the proof. 
\end{proof}

It is also useful to describe this geometrically using blowups.  
Consider the ideal $\mathcal{I}\subset \K[X]$ generated by the anti-invariant functions $g\in \K[X]^{W,-}$. 
The blowup of $X$ at this ideal carries a natural $W$-action, where the action on sections of $\mathcal{O}(1)$ is twisted by the sign representation.  
The quotient of this blowup by $W$ has a natural map to $\HilbW{X}$, induced by mapping anti-invariant functions to the unique anti-invariant line in the tautological bundle.
 This induces an isomorphism to $\tWHilb{\pi}{X}$ from the open subset where we remove the principal transform of the divisor defined by $\alpha$.  This principal transform is the zero locus of $\alpha$, viewed as a section of $\mathcal{O}(1)$.

We can state this in a more suggestive way.  The variety $\tWHilb{\pi}{X}$ carries a tautological bundle $\mathcal{T}$, which is trivialized by the sections $1,\alpha$.
We can think of any function on $X$ as an endomorphism of this vector bundle, and we can define a unique endomorphism $\psi$ satisfying $\psi(\alpha)=2$ and $\psi(1)=0$.

\begin{Definition}
Let $\NH{W}{X}$ be the algebra generated by the symbol $\psi$ and the functions of $\K[X]$ subject to the relations:
\begin{equation}\label{eq:NH-rank-one}
 \psi^2=0\qquad (1-\alpha\psi) g=(s\cdot g) (1-\alpha\psi)    
\end{equation}
\[g\psi =\psi g \qquad \text{ if } sg=g. \]
\notation{$\NH{W}{X}$}{nilHecke-type algebra attached to the $W$-action on $X$}\end{Definition}

\begin{Theorem}\label{thm:NH-rank-one}
    $\End(\mathcal{T})\cong \NH{W}{X}$.
\end{Theorem}
\begin{proof}
    We have already defined the endomorphism $\psi$, and we can define a homomorphism $\NH{W}{X}\to \End(\mathcal{T})$ sending $g\in \K[X]$ to multiplication by this element, and $\psi$ to the endomorphism $\psi$.  Note that in $\End(\mathcal{T})$, we can rewrite \eqref{eq:NH-rank-one} as
    \[\psi (gh)=(s\cdot g) \psi(h)+\psi(g)h    \] using the usual Leibniz rule for divided difference operators.  

    In particular, if $sg=-g$, then the element $\mathsf{\Psi}(g)=\frac{1}{2}(\psi g +g\psi)$ maps to $\Psi(g)$.  Note that the invariants $\K[X]^s$ and the elements $\mathsf{\Psi}(g)$ generate a commutative subalgebra.  Moreover, $\mathsf{\Psi}(f\alpha)=f$ and $\mathsf{\Psi}(gf)=f\mathsf{\Psi}(g)$ for $f\in \K[X]^s$, so the presentation of \cref{thm:alpha-inverse} gives a homomorphism $\K[\tWHilb{\pi}{X}]$ to this subalgebra sending the generator $\Psi(g)$ to $\mathsf{\Psi}(g)$.  Composing with the homomorphism to $\End(\mathcal{T})$ shows that this map is an isomorphism onto this commutative subalgebra.

    Consider the structure of $\End(\mathcal{T})$ as a left module over $\K[\tWHilb{\pi}{X}]$. Any element of $\NH{W}{X}$ can be written as a monomial $\psi g_1 \psi g_2 \cdots \psi g_m$ with $g_i\in \K[X]$.  We can write any $g=g^++g^-\alpha$ where $g_+=\frac{1}{2}(g+sg)$ and $g_-=\frac{1}{2}\psi(g)\in \K[\tWHilb{\pi}{X}]$.  Applying this to each $g_i$, we can reduce to the case where $g_i\in \{1,\alpha\}$ for $i>1$.  Since $\psi ^2=0,\psi\alpha\psi=\psi $  and $1=\frac{1}{2}(\psi\alpha+\alpha\psi)$, we can assume that there is exactly one appearance of $\psi$.  This shows that the elements $\frac{1}{2}\alpha \psi,\frac{1}{2}\psi, \frac{1}{2}\alpha\psi\alpha, \frac{1}{2}\psi \alpha$ generate $\NH{W}{X}$ as a $\K[\tWHilb{\pi}{X}]$ module.  

    In fact, these elements define a surjective map from $M_{2\times 2}(\K[\tWHilb{\pi}{X}])\to \NH{W}{X}$:
    \[\begin{bmatrix}
        g_1 & g_2\\
        g_3 & g_4\\
    \end{bmatrix}\mapsto \frac{1}{2}g_1\alpha \psi+ \frac{1}{2}g_2\alpha\psi\alpha+\frac{1}{2}g_3\psi+ \frac{1}{2}g_4\psi \alpha\]
    The kernel of this map must be generated by an ideal in $\K[\tWHilb{\pi}{X}]$.  This ideal must be trivial since the induced map obtained by composing with the map to $\End(\mathcal{T})$ is an isomorphism.  Thus, we have the desired isomorphism.
\end{proof}
Let $e=\frac{1}{2}\psi \alpha$; in the action of $\NH{W}{X}$ on $\mathcal{T}$, this acts by the projection $f\mapsto \frac{1}{2}(f+sf)$.
\begin{Corollary}
    $\K[\tWHilb{\pi}{X}]\cong e \NH{W}{X} e$.
\end{Corollary}

We can recover $\K[\tHilb{\pi}{X}]$ instead of $\K[\tWHilb{\pi}{X}]$ by considering the image of $\NH{W}{X}$ in the twisted character ring $W\ltimes \K(X)$, where the element $\psi$ is sent to the operator $g\mapsto \frac{g-sg}{\alpha}$.

\subsection{Relation to nilHecke algebras: general case}
Now let $W$ be a finite reflection group acting faithfully on the finite-dimensional complex vector space $V$. Let $\simples$ be the set of reflections in $W$, and fix $\alpha_s$ a linear function such that $s\alpha_s=-\alpha_s$.  Let $\Delta=\prod \alpha_s$ be the unique lowest-degree anti-invariant.\notation{$\Delta$}{Lowest-degree anti-invariant for the reflection representation of $W$} Let $R=\K[V]=\Sym V^*$.

\begin{Definition}
For any dominant affine $V$-scheme $X$ with a compatible $W$-action, we let $\NH{W}{X}$ be the algebra generated by a copy of  $\K[X]$ and the operators $\psi_{s}$ for each simple reflection $s$, subject to the relations that
\begin{enumerate}
    \item $\psi_{s}$ satisfies the rank-$1$ relations \cref{eq:NH-rank-one} with respect to the reflection $s$ and the function $\alpha_s$.
    \item $\psi_{s}$ commutes with functions invariant under $s$.
    \item The operators $\psi_{s}$ satisfy the braid relations of $W$.  That is, if $s_1,\dots, s_m$ and $t_1,\dots, t_m$ are two sequences of simple reflections such that $s_1\cdots s_m=t_1\cdots t_m$ has length $m$ in $W$, then $\psi_{s_1}\cdots \psi_{s_m}=\psi_{t_1}\cdots \psi_{t_m}$.
\end{enumerate} 
\end{Definition}
The case of $X=V$ is particularly nice:
\begin{Proposition}\hfill
\begin{enumerate}
    \item $\NH{W}{V}=\End_{R^W}(R)$ is a matrix ring of rank $\# W$ over $R^W$.
    \item The ring extension $R^W\subset R$ is Frobenius with Frobenius trace \[\Psi_{w_0}(f)=\frac{1}{\# W}\cdot\frac{\sum (-1)^{\ell(w)}wf}{\Delta}.\]
\end{enumerate}
\end{Proposition}
\notation{$\Psi_{w_0}$}{Frobenius trace for the free extension $R^W\subset R$}That is, there are bases $\{f_1,\dots, f_{d}\}$ and $\{f^1, \dots, f^d\}$ of $R$ as a free $R^W$ module such that 
$g=\sum_{i=1}^df^i\Psi_{w_0}(f_ig)$ for all $g\in R$. Furthermore, we can assume that the $\K$-span of $\{f_1,\dots, f_{d}\}$ and $\{f^1, \dots, f^d\}$ are $W$-invariant.  For simplicity, we assume that $f_1=f^d=1, f_d=f^1=\Delta$, so all other $f_i,f^i$ have trivial projection to the invariant and anti-invariant subspaces.  
By this freeness, the images of the vectors $f_i$ give a basis of the functions on any subscheme whose functions are a rank 1 $\K W$-module.  This corresponds to a maximal ideal of $R^W$, and one of the things we can do with this Frobenius structure is describe the action of multiplication by $f_i$ on this subscheme as a function of this maximal ideal: it is given by the $R^W$-valued $d\times d$ matrix $(F_i)_{j,k}=\Psi_{w_0}(f_if_jf^k)$.  

This gives us a natural approach to describing the transverse Hilbert scheme.  Given such a subscheme, every element of $\K[X]$ is congruent to a unique linear combination of the $f_i$, so we can think of the coefficients in this linear combination as functions on 
$\tWHilb{\pi}{X}$.  Given $h\in \K[X]$, let $H_i$ be the unique function on $\tWHilb{\pi}{X}$ given by $h\equiv\sum H_if_i$.  
\notation{$H_i,H_i^{(j)}$}{Coefficient functions expressing generators of $\K[X]$ in the Frobenius basis on the transverse Hilbert scheme}

To see this function more explicitly, the wedge $\upsilon$ of the vectors $f_{i}$ (in some order) gives a section of $\bigwedge^{\operatorname{top}}\mathcal{T}$ which is non-vanishing on $\tWHilb{\pi}{X}$.  We can rewrite the function $H_i$ as the ratio of the wedge with $f_i$ replaced by $h$ with the section $\upsilon$.  This means that the action of $h$ on the subscheme defining a point of $\tWHilb{\pi}{X}$ is given by the matrix $\sum_iH_iF_i$.

Observe that if $h$ is invariant, then $H_1$ is the only one of these functions that is non-zero, and it can be identified with $h$ as a function on $X/W$, pulled back via the Chow map.  Similarly, if $h$ is anti-invariant, then only $H_d$ is non-zero; we have an equality of invariant functions $H_d\Delta=h$, so on the open set of free orbits, we can identify $H_d$ with the pullback of $h/\Delta$.  We let $\Psi(h)=H_d$ for an arbitrary function $h\in \K[X]$; this is a well-defined function on $\tWHilb{\pi}{X}$, and it agrees with the previous definition of $\Psi(h)$ in the rank 1 case.  Equivalently, $\Psi_{w_0}(h)=\Psi(h)$.  

We will show below that these functions $H_i$ generate the functions on $\tWHilb{\pi}{X}$ modulo ``the obvious relations.''

Let us be more precise.  To make it clear that our variety is of finite type, we choose a finite generating set $h^{(1)},\dots, h^{(p)}$ of $\K[X]$ as an $R$-algebra.  For simplicity, we will assume that the span of  $\{h^{(1)},\dots, h^{(p)}\}$ is $W$-invariant.

There are two natural sets of relations that the functions $H^{(j)}_i$ must satisfy:
\begin{enumerate}
	\item The map $h^{(j)}\mapsto \sum_i H_i ^{(j)} F_i$ must be $W$-invariant.  Since $w\cdot h^{(j)}$ is a linear combination of other $h^{(k)}$, and the same is true of $F_i$, each $w$ gives linear relations on $H_i^{(j)}$ for all $i$ and $j$. 
    \item If a polynomial $q(h^{(1)},\cdots, h^{(p)})=0 $ gives a relation on the $h^{(*)}$, then the same is true of its image.  That is:
	\[q\Big(\sum_i H_i ^{(1)} F_i, \dots , \sum_i H_i ^{(p)} F_i\Big)=0\]
	The left-hand-side of this equation is a linear combination of $F_i$'s with coefficients that are polynomials in $H_i^{(j)}$, and so this equation gives relations between them. 
\end{enumerate}

\begin{Remark}\label{rmk:isotypic}
We can simplify the relations (1) by more carefully choosing our basis to be compatible with the $W$-representation structure. 	

Enumerate the simple characters $\{\chi_1,\dots, \chi_t\}$ of $W$.  
We can change to a basis $\{f_{i,r,s}\}$ of the span of the $f_i$'s for $i=1,\dots, t$ and $r,s=1,\dots, \chi_i(1)$, so that $V_{i,s}=\operatorname{span}(f_{i,r,s})_{r=1,\dots, \chi_i(1)}$ for $i,s$ fixed is a copy of $V_i$, an irrep affording character $\chi_i$, for some fixed basis $v_{i,s}$.  

We can similarly organize our generators of $\K[X]$ to be $h^{(i,r,k)}$ so that $\operatorname{span}(h^{(i,r,k)})_{r=1,\dots, \chi_i(1)}$ for $i,k$ fixed is a copy of $V_i$.  

Once we have chosen these bases, the relations (1) say that $H^{(j,\ell,k)}_{i,r,s}$ is zero unless $i=j$ and $\ell=r$, and that in the non-zero case, it depends only on $i,k,s$, not on $r$. 
\end{Remark}

\begin{Proposition}\label{prop:Hilb-gen}
    The functions on $\tWHilb{\pi}{X}$ are generated over $R^W$ by the functions $H_i^{(j)}$ subject to the relations (1-2).  That is, $\tWHilb{\pi}{X} $ is the subscheme of $\mathbb{A}^{dp}\times X/W$ with these defining equations. 
\end{Proposition}
This proposition makes more precise the discussion in \cite[Rmk. 3.6 \& 3.9]{bielawskiHypertoricVarieties2023}.
\begin{proof}
	Let $Z$ be the subscheme defined by these equations, and for a commutative $\K$-algebra $Q$, consider the set $Z(Q)$ of $Q$-points of $Z$.  Such a $Q$-point $x$ gives scalars $H_i^{(j)} (x)\in Q$ which satisfy the relations (1)-(2).  This gives a $Q$-module homomorphism $Q\otimes_{\K}\K[X]\to Q^d$.  The relations (2) show that the kernel is an ideal, and the relations (1) show that this ideal is $W$-invariant.  Thus, the quotient by this ideal defines a subscheme of $X_Q$ that is flat over $Q$, whose coordinate ring is a free module of rank 1 over $QW$, and whose pushforward to $V_Q$ is a subscheme.  
	
Conversely, the coordinate ring of such a subscheme is a free $Q$-module with basis given by the images of the $f_i$ for $i=1,\dots, d$.  For each generator $h^{(j)}$, the corresponding coefficients $H_i^{(j)}$ are therefore well-defined elements of $Q$, and these determine a point of $Z(Q)$.  Thus, there is a natural bijection between the $Q$-points of $Z$ and $\tWHilb{\pi}{X}$, which proves the desired isomorphism.
\end{proof}

We can deduce from this:
\begin{Theorem}\label{th:divide-delta}
    The functions on $\tWHilb{\pi}{X}$ are generated by $\K[X]^W$ and the functions $\Psi(g)$ for $g\in \K[X]$ subject to the relations 
    \[\Psi(gh)=\sum_i\Psi(f^ig)\Psi(f_ih) \qquad\Psi(h)=\frac{1}{\# W}\sum_{w\in W} (-1)^{\ell(w)} \Psi(wh) \qquad \Psi(h\Delta)=\frac{1}{\# W}\sum_{w\in W} wh\] for all $g,h\in \K[X]$.  The functions on $\tHilb{\pi}{X}$ are the image of the induced homomorphism $\K[\tWHilb{\pi}{X}] \to \K(X)^W$ sending $\Psi(g)\mapsto g/\Delta$ for $g\in \K[X]^{W,-}$; equivalently, this means that $\K[\tHilb{\pi}{X}]$ is the $W$-invariants of the {\bf Demazure envelope} of $\K[X]$ in the sense of \cite{gannonCoordinateRing2026}.
\end{Theorem}
It will be useful to note that we can apply the relation above twice to obtain:
\begin{equation}\label{eq:triple-expansion}
\begin{split}
\Psi(g_1g_2g_3)&=\sum_i\Psi(f^ig_1)\Psi(f_ig_2g_3)\\
&=\sum_{i,j}\Psi(f^ig_1)\Psi(f^jg_2)\Psi(f_if_jg_3).
\end{split}
\end{equation}
\begin{proof}
    This is effectively the same as the proof of \cref{thm:alpha-inverse}, but with more complicated notation.  We can choose generators $h^{(j)}$ of $\K[X]$ as above, and consider the functions $H_i^{(j)}$.  To begin with, we write each $H_i^{(j)}$ as a scalar multiple of $\Psi(g)$.  
    
    Consider the function $h^{(j)} f^i\in \K[X]$.  Let $g$ be the projection of this function to the anti-invariants.  The function $h^{(j)} f^i\in \K[X]$ acts on the tautological subscheme by the sum $\sum_k H_k^{(j)}F_kF^i$.  Since $f_kf^i$ has trivial projection to the anti-invariants if $k\neq i$, and projection $\Delta=f_d$ if $k=i$, we find that $g$ acts on this subscheme by multiplication by $H^{(j)}_iF_d$.  That is, $\Psi(g)=H^{(j)}_i$.  This shows the desired functions generate.

    The relations are also straightforward to verify, so the non-trivial part of the proof is to show that they are sufficient to deduce all other relations.  Equivalently, we need to show that if we assign values $\Psi(g)(x)\in Q$ for a $\K$-algebra $Q$ to the generators $\Psi(g)$, and these values satisfy the relations in the statement of the theorem, then they define a $Q$-point of $\tWHilb{\pi}{X}$.  We define a homomorphism $\eta$ by setting $\eta(h)= \sum_i \Psi(f^ih)(x)F_i$.  Using a basis compatible with irreps as in \cref{rmk:isotypic}, we can verify that this homomorphism is $W$-invariant, since the tensor $\sum_i f^i\otimes F_i$ is anti-invariant, and $\Psi$ intertwines the sign and trivial $W$-representations.  Moreover, the image contains the operators coming from $R$, so $\eta$ is surjective.

    Thus, we only need to verify that this homomorphism is well-defined, that is, that $\eta(gh)=\eta(g)\eta(h)$ for all $g,h\in \K[X]$.  We can write $g=\sum_i \Psi(f^ig)(x)f_i$ and $h=\sum_i \Psi(f^ih)(x)f_i$.  Thus, we have
    \[\eta(gh)=\sum_k \Psi(ghf^k )(x)F_k=\sum_{i,j,k} \Psi(gf^i)(x)\Psi(hf^j)(x)\Psi(f_if_jf^k)(x)F_k\] by applying \cref{eq:triple-expansion} with $g_1=g$, $g_2=h$, and $g_3=f^k$.
    On the other hand, we have  
    \[\eta(g)\eta(h)=\sum_{i,j} \Psi(gf^i)(x)\Psi(hf^j)(x)F_iF_j=\sum_{i,j,k} \Psi(gf^i)(x)\Psi(hf^j)(x)\Psi(f_if_jf^k)(x)F_k,\]
    since $F_iF_j=\sum_k \Psi(f_if_jf^k)F_k$ by the definition of the matrices $F_i$.
    Comparing these expressions shows that we have a homomorphism as required, so we have the desired isomorphism.
\end{proof}
As in the rank-$1$ case, we can consider the element $e= \frac{1}{\# W}\psi_{w_0}\Delta$ which acts by projection to the trivial isotypic component, and we have:
\begin{Lemma}
$\K[\tWHilb{\pi}{X}]\cong e\NH{W}{X}e.$
\end{Lemma}
\begin{proof}
    Let $\psi_s$ be the endomorphism of the tautological bundle $\mathcal{T}$ that sends $g$ to
    \[\psi_{s}(g)=\sum_j \frac{f^j-sf^j}{\alpha_s}\Psi(f_jg).\]
    In the case of $X=V$, this is the usual divided difference operator, so calculations in that case verify that the relations satisfied by $\psi_s$ are exactly those in the definition of $\NH{W}{X}$.  Thus, we have a homomorphism $\NH{W}{X}\to \End(\mathcal{T})$ sending $g\in \K[X]$ to multiplication by this element, and $\psi_s$ to the endomorphism $\psi_s$.  Furthermore, this case shows that $\frac{1}{\# W}\psi_{w_0}$ gives the operation $g \mapsto \Psi(g)$ defined above.
    
The bundle $\mathcal{T}$ is trivialized by the sections $f_i$, and for each $i$, the element $f^i\psi_{w_0}f_i$ acts by projection to the line spanned by $f_i$.  Thus, we have \[e\End(\mathcal{T})e\cong \K[\tWHilb{\pi}{X}].\]  The homomorphism $\NH{W}{X}\to \End(\mathcal{T})$ induces a homomorphism \[\gamma\colon e\NH{W}{X}e\to \K[\tWHilb{\pi}{X}],\] which we wish to show is an isomorphism.  

We will prove this by showing that we have an inverse map \[\beta \colon \K[\tWHilb{\pi}{X}]\to e\NH{W}{X}e\] sending $\beta(h)=eh$ for $h\in \K[X]^W$ and  $\beta(\Psi(g))= \frac{1}{\# W}\psi_{w_0}ge$.  To see that this is well-defined, we need to verify that the relations in \cref{th:divide-delta} are satisfied.  The first relation follows from \[\beta(\Psi(gh))=\frac{1}{\# W}\psi_{w_0}ghe=\frac{1}{(\# W)^2}\sum_i (\psi_{w_0}gf^ie)(\psi_{w_0}hf_ie)= \sum_i\beta(\Psi(gf^i))\beta(\Psi(hf_i)).\]
The second relation follows from the fact that $s\psi_{w_0}=-\psi_{w_0}$ for all reflections $s$.  The third relation follows from the fact that $\psi_{w_0}\Delta h e=\sum_{w\in W} whe$.  

Finally, we need to verify that $\beta$ and $\gamma$ are inverses.  The composition $\gamma\beta$ is clearly the identity on $\K[X]^W$.  On the other hand, \[ \gamma\beta(\Psi(g))=\gamma\Big(\frac{1}{\# W}\psi_{w_0}ge\Big)=\Psi(g)\] since we have already verified that $\frac{1}{\# W}\psi_{w_0}$ maps to the operation $\Psi$.  This shows that $\gamma\beta$ is the identity.

On the other hand, using the Frobenius formula $\psi_s=\sum_j \frac{f^j-sf^j}{\alpha_s}\, \psi_{w_0} f_j$, any element of $e\NH{W}{X}e$ can be written as a product of $\frac{1}{\# W}\psi_{w_0} g_1 e, \frac{1}{\# W}\psi_{w_0} g_2e,\dots, \frac{1}{\# W}\psi_{w_0} g_me$ for some $g_i\in \K[X]$.  Applying $\gamma$ to this element gives $\Psi(g_1) \Psi(g_2)\cdots \Psi(g_m)$, and applying $\beta$ gives the same element back since \[\Big(\frac{1}{\# W}\psi_{w_0} g_1 e\Big)\Big(\frac{1}{\# W}\psi_{w_0} g_2e\Big)\cdots \Big(\frac{1}{\# W}\psi_{w_0} g_me\Big)=\beta(\Psi(g_1))\beta(\Psi(g_2))\cdots \beta(\Psi(g_m)).\]  This shows that $\beta\gamma$ is the identity, so we have the desired isomorphism.
\end{proof}
\subsubsection{Comparison with Coulomb branches}
For the rest of this comparison, we return to the rational Coulomb branch with $B=\Ga$.  Let $X=\tM(\tT,T,\bfN)$ and $V=\base{\tT}=\LtT$, with structure map $\pi\colon X\to V$ and the natural action of the Weyl group $W$ on both varieties.  
\begin{Lemma}
 There is a natural inclusion of algebras $\K[\tHilb{\pi}{X}] \subset \K[\tM(\tG,G,\bfN)]$.  
\end{Lemma}
\begin{proof}
The relevant portion of the proof of \cite[\S 7.1]{bielawskiHypertoricVarieties2023} gives this inclusion, though the claim that it is a surjection is not in general true.  
 
 On the other hand, we can prove this directly using the description of functions on the Coulomb branch from \cite[\S 3]{WebSD}.  There $\K[X]$ is denoted $\EuScript{A}_{\operatorname{ab}}$ and the algebra $\NH{W}{X}$ naturally maps to a subalgebra of the larger algebra $\EuScript{A}$, the Iwahori Coulomb branch algebra \cite[Def. 3.2]{WebSD}.  Applying the idempotent $e$ on both sides gives  a map
    \[\K[\tWHilb{\pi}{X}]=e\NH{W}{X}e\to \K[\tM(\tG,G,\bfN)]=e\EuScript{A}e.\]
    Since the Coulomb branch $\K[\tM(\tG,G,\bfN)]$ is birational to $X/W$, the image of this map is exactly $\K[\tHilb{\pi}{X}]$, so we have the desired inclusion.
\end{proof}
By the Morita equivalence \cite[Th. 3.3]{WebSD}, this inclusion is an isomorphism if and only if $\NH{W}{X}=\EuScript{A}$.
There is one important case where this equality does hold:

\begin{Proposition}\label{prop:rank-1}
    If $G$ is a reductive group of semi-simple rank $1$ and $G\ncong SL(2)\times Z(G)^{\circ}$, then $\K[\tHilb{\pi}{X}] \cong \K[\tM(\tG,G,\bfN)]$.
\end{Proposition}  
The argument in the proof of \cite[Th. 6.13]{bielawskiHypertoricVarieties2023} establishes this, but we reproduce it here for clarity.
\begin{proof}
    Note that the Weyl group of $G$ is the order $2$ group $\{1,s\}$.  Since $G\ncong SL(2)\times Z(G)^{\circ}$, the cocharacter lattice of $G$ is generated by cocharacters satisfying $\alpha(\nu)=\pm 1, 0$.  The explicit rank-$1$ calculations in \cref{subsubsec:BFN-comparison} show that the non-abelian coordinate ring is generated by the corresponding dressed monopole operators, which we give formulas for in \cref{eq:PGL2-monopole}.  These generators are either identical with the abelian monopole operators (if $\alpha(\nu)=0$) or obtained from an abelian monopole operator by a divided-difference operator, given by the Atiyah--Bott formula for $\mathbb{P}^1$.  In either case, they lie in $\NH{W}{X}$, so the inclusion is surjective as required.  
\end{proof}

\subsection{A counterexample}
Outside semi-simple rank~$1$, the global comparison can fail.\footnote{This issue is acknowledged by the authors in an arXiv note attached to \cite{bielawskiHypertoricVarieties2023}.}  We give two counterexamples, pointed out to us by Ki Fung Chan: the first shows that the transverse Hilbert scheme can fail to be flat over the base, and the second, for $SL(3)$, shows that the Coulomb branch is not isomorphic to the transverse Hilbert scheme in this case.  In particular, the second example shows that \cite[Th. 7.1]{bielawskiHypertoricVarieties2023} is false in general, even when we avoid $Sp(2n)$ factors.  We will give a corrected version of this statement in \cref{prop:BF}.

Consider the representation of $G=GL(1)\times GL(2)$ on $\C^2\oplus \Hom(\C^2,\C^1)$, where $\C^1$ and $\C^2$ are the defining representations of the two factors.  The abelian Coulomb branch is isomorphic to 
\[X\cong \Spec \C[x_{ij}]_{i,j=1,\dots, 4}/(x_{11}+x_{22}+x_{33}+x_{44},x_{ij}x_{k\ell}-x_{i\ell}x_{kj})\]
the space of $4\times 4$ nilpotent matrices of rank $\leq 1$.  This isomorphism sends
\begin{equation*}
    (a,b,c)\in \LtT^* \mapsto (a-b)x_{11}+(a-c)x_{22} +b x_{33} +c x_{44}
\end{equation*}
\begin{align*}
    r_{1,0,0}&\mapsto x_{12} & r_{-1,0,0} & \mapsto -x_{21} &
    r_{-1,-1,0} &\mapsto -x_{23} & r_{1,1,0}& \mapsto x_{32} \\
    r_{1,1,1}& \mapsto x_{34} & r_{-1,-1,-1}& \mapsto -x_{43} &
    r_{0,-1,0}& \mapsto -x_{13} & r_{0,1,0}& \mapsto -x_{31}\\
    r_{0,0,1}&\mapsto x_{24} & r_{0,0,-1} & \mapsto x_{42} &
    r_{1,0,1} &\mapsto x_{14} & r_{-1,0,-1} & \mapsto -x_{41} 
\end{align*}
The action of $s\in W$ sends $(a,b,c)\to (a,c,b)$ and $r_{a,b,c}\to r_{a,c,b}$.  In particular, in the degree 2 part of $\K[X]$, the span of $(a,0,0)\in \LtT^*$ and the monopole operators $ r_{\pm 1,0,0}$ and $r_{\pm(1,1,1)}$ give a 5-dimensional space on which $s$ acts trivially, and the complementary 10-dimensional space is a free module of rank $5$ over $\C W$, and thus has a $-1$ eigenspace of multiplicity 5.  Thus, $\WHilb{X}$ contains a subvariety isomorphic to $\mathbb{P}^4$, given by the ideals containing all functions of degree $>2$, all $s$-fixed elements of degree $2$, and a codimension 1 subspace in the $-1$-eigenspace of $s$ (the choice of this subspace gives a point of $\mathbb{P}^4$).  This $\mathbb{P}^4$ lies in the fiber over the origin in $\C^3/W$, showing that this fiber has dimension $\geq 4$, while the general fiber has dimension 3.  This shows that:
\begin{Proposition}
    The map $\WHilb{X}\to \C^3/W$ is not flat in this case.
\end{Proposition}
This failure is not isolated.  In particular, \cite[Th. 7.1]{bielawskiHypertoricVarieties2023} does not hold in many cases beyond those involving an $Sp(2n)$ factor, especially when the center has several components.  The group $SL(3)$ already illustrates the issue.  For example, take the representation $\bfN=(\C^3)^{\oplus 4}$.  
The degree of an abelian monopole operator for $(a_1,a_2,a_3)$ with $a_1+a_2+a_3=0$ is $4(|a_1|+|a_2|+|a_3|)$.  In particular, the minimal such degree is 8, since at least two of the $a_i$ are non-zero. 
By \cref{th:divide-delta}, this means that the span of all non-trivial dressed monopole operators in the nonabelian Coulomb branch is concentrated in degree $\geq 2$, and so the degree-$0$ part of $\K[\HilbW{X}]$ consists only of scalars, whereas the monopole operator $r_{(1,0,-1)}$ has degree $0$, and thus is not in $\K[\HilbW{X}]$.  Thus, we have:
\begin{Proposition}\label{prop:SL3-counterexample}
    When $X=\M(T,\bfN)$, where $T\subset SL(3)$ is the maximal torus, the natural map does not induce an isomorphism between $\HilbW{X}$ and $\M(SL(3),\bfN)$.
\end{Proposition}
This shows that the subring $\K[\M(SL(3),\bfN)]$ is not integral over $\K[\HilbW{X}]$. The variety $\M(SL(3),\bfN)$ is normal and affine and the map $\M(SL(3),\bfN)\to \HilbW{X}$ is birational, so Zariski's Main Theorem implies that $\M(SL(3),\bfN)$ is a proper open subset of the normalization of $\HilbW{X}$.  Note that this shows that this normalization is not flat over $\LT/S_3$, since a divisor in the complement of $\M(SL(3),\bfN)$ must have image of codimension $\geq 2$ in $\LT/S_3$.
The same argument should apply to many other $SL(3)$ Coulomb branches, though without enough matter to give the Coulomb branch a conical grading, the dimension count becomes less transparent.  Interestingly, \cite[Cor. 1.2.2]{gannonCoordinateRing2026} implies that if $\bfN=0$, then the natural map is an isomorphism as long as $G$ has no $Sp(2n)$ factors, in particular, for $G=SL_n$ for $n>2$.

The origin of the issue in \cref{prop:SL3-counterexample} is easily explained in terms of the presentation of $\K[\tM(\tG,G,\bfN)]$ given in \cite[\S 3]{WebSD}:  elements of $\K[\tM(\tG,G,\bfN)]$ can be described using paths in the real Cartan that join points in the coweight lattice, with an optional application of a Demazure operator, that is, the affine divided-difference operator $\psi_s$, at each point where a path crosses the vanishing locus of an affine root.  Here the affine hyperplane $y_1-y_3=1$ corresponds to the affine root hyperplane $\alpha_{13}-\delta$.  The path from $(0,0,0)$ to $(1,0,-1)$ crosses this hyperplane, but $\NH{W}{X}$ has no way to add the corresponding Demazure operator.  This illustrates the utility of the extended-category approach to the Coulomb branch, which corrects the Hilbert-scheme construction.  

On the other hand, there are many cases where $\HilbW{\tM(\tT,T,\bfN)}\cong \tM(\tG,G,\bfN)$.  By \cite[Cor. 2.11]{weekesGeneratorsCoulomb2019}, this holds if $\tM(\tG,G,\bfN)$ is generated by minuscule monopole operators, that is, monopole operators attached to minuscule coweights.  In particular, this holds for quiver gauge theories by \cite[Prop. 3.1]{weekesGeneratorsCoulomb2019}.

\begin{Remark}
    One possible way to repair this failure is to enlarge the group, as in \cite{bielawskiHypertoricVarieties2023}, where $Sp(2n)$ is embedded in the conformal symplectic group.  One can then recover the Coulomb branch for the smaller group as a symplectic quotient.  This is natural to expect because enlarging the gauge group enables more factorization of abelian monopole operators
    and thus the construction of lower degree elements of $\K[\HilbW{X}]$.  For example, if we consider gauge group $GL(3)$ instead of $SL(3)$, then $\bfN=(\C^3)^{\oplus 4}$ will give a quiver gauge theory and so $\M(GL(3),\bfN)$ will be generated by minuscule monopole operators and is isomorphic to the transverse Hilbert scheme.
\end{Remark}

\subsection{Corrected comparison}

The counterexamples above show that $\tM(\tG,G,\bfN)$ is not, in general, itself a transverse Hilbert scheme.  The correct replacement is to restrict to the open locus controlled by the generic and subgeneric slices of \cref{subsubsec:slices-codim-1}.  Recall that by a generalized root we mean either a root of $G$ or a weight of $\bfN$.  Consider the subset $\base{\tG}^{\bullet}\subset \base{\tG}$ given by the set of points that lie on zero or one hyperplane defined by the vanishing of a generalized root.  The papers \cite{BFN, bielawskiHypertoricVarieties2023} both use the fact that the preimage $\tM^{\bullet}(\tG,G,\bfN)$ of this set in $\tM(\tG,G,\bfN)$ controls its geometry; we let $\tM^{\bullet}(\tT,T,\bfN)$ be the preimage of this set in $\tM(\tT,T,\bfN)$.  As discussed in \cref{lem:bullet-affinization}, the variety $\tM(\tT,T,\bfN)$ is the affinization of this open subset over $\base{\tG}$.

While $\tM(\tG,G,\bfN)$ is not a transverse Hilbert scheme, the open subvariety $\tM^{\bullet}(\tG,G,\bfN)$ is.
\begin{Proposition}\label{prop:BF}
    If $G$ has no factors isomorphic to $Sp(2n)$, then we have an isomorphism of $\base{\tG}$-schemes $\tM^{\bullet}(\tG,G,\bfN)\cong \HilbW{\tM^{\bullet}(\tT,T,\bfN)}$.
\end{Proposition}

\begin{proof}
Both $\HilbW{\tM^{\bullet}(\tT,T,\bfN)}$ and $\tM^{\bullet}(\tG,G,\bfN)$ have flat maps to $\base{\tG}^{\bullet}=\base{\tT}^{\bullet}/W$.  It therefore suffices to produce a $W$-equivariant isomorphism after base change to $\base{\tT}^{\bullet}$.  Over the generic locus $\base{\tG}^{\circ}$, both varieties are isomorphic to $\base{\tG}^{\circ}\times T^{\vee}$, so this gives a birational isomorphism that we wish to extend across the codimension-$1$ loci. 

It is enough to check this after base change to $\base{\tGl,\lambda}$ for subgeneric $\lambda$, since these give an \'etale cover of $\base{\tG}^{\bullet}$.  

 By \cref{lem:slice-ours}, after base change we obtain $\tM^{\bullet}(\tGl,\Gl,\Nl)_{\base{\tGl,\lambda}}$, where $\tGl$ has semi-simple rank $1$ or $0$.  Moreover, $\tGl$ is not isomorphic to $SL(2)\times Z(\tGl)^{\circ}$, since $G$ has no factors isomorphic to $Sp(2n)$.

 On the other hand, by \cite[Prop. 2.18]{bielawskiHypertoricVarieties2023}, the base change of $\HilbW{\tM^{\bullet}(\tT,T,\bfN)}$ is
 $\operatorname{Hilb}^{W_{\lambda}}(\tM^{\bullet}(\tT,T,\bfN)_{\base{\tT,\lambda}})$, where $W_{\lambda}$, the stabilizer of $\lambda$ in $W$, is either trivial or of order $2$, depending on whether $\Gl$ is abelian or not.
 If $\tGl=\tT$ is abelian, then the isomorphism is tautological.  If $\tGl$ is non-abelian, then we are reduced to the case of a group of semi-simple rank $1$, which is covered by \cref{prop:rank-1}.  
\end{proof}

\section{Slices and leaves}
\label{sec:slices-and-leaves}

In this section, we specialize to the case $\tG=G$ and $B=\Ga$.  Furthermore, in some subsections, we will only work in the symmetrizer-free setting.  We will make some remarks indicating how these results might change in the symmetrizer setting.  In particular, $\M(G,\bfN):=\tM(G,G,\bfN)$ is a symplectic singularity by \cite{bellamyCoulombBranches2023}, so it has finitely many symplectic leaves.  As before, we write $\ints\colon \M(G,\bfN)\to \LT/W$ for the natural integrable system and write $R=\Sym(\LT^*)$.  The variants $B=\Gm$ and $B=\elli$ are obtained from the same codimension-$1$ local models \'etale locally on $\base{G}$, so the leaf classification below can easily be modified to cover all three cases.

The local models of \cref{subsubsec:slices-codim-1,subsubsec:finite-covers} reduce the description of the leaves of $\M(G,\bfN)$ to the smaller Coulomb branches $\M(\Gl^1,\Nl)$.  Our goal is to make this reduction explicit.  We begin with two preparatory steps: first we explain how passing from a possibly unfaithful action to an almost faithful one affects leaves, and then we record the monopole calculation that forces the image of a leaf to be a flat.

\subsection{Preparatory results}

\subsubsection{Unfaithful representations}
\label{sec:unfaithful}

In this subsection, assume that we are symmetrizer-free.
Most often, authors consider Coulomb branches attached to faithful representations, but useful information can also be gained from unfaithful ones.  Let $G^0$ be the identity component of the kernel of the action of $G$ on $\bfN$, and let $G^1=G/G^0$.  The remaining kernel $K$ of the induced action of $G^1$ on $\bfN$ is therefore finite and central, so this action is almost faithful.  By \cref{lem:cover-ours}, $\M(G^1,\bfN)$ is the quotient of $\M(G^1/K,\bfN)$ by the action of $K^{\vee}$, which is not necessarily free and thus can affect the set of leaves.  
\begin{Example}
    If $G=\Gm$ and $\bfN$ is 1-dimensional with weight $n$, then $K$ is the group of $n$-th roots of unity.  The Coulomb branch $\M(G/K,\bfN)$ is $\mathbb{A}^2$, since $G/K$ acts with weight 1.  The action of $K^{\vee}$ is the usual symplectic action of the $n$-th roots of unity on $\mathbb{A}^2$, so the quotient $\M(G,\bfN)$ is the $A_{n-1}$ surface singularity.  In particular, $\M(G/K,\bfN)$ has one leaf, while $\M(G,\bfN)$ has two.
\end{Example}

Since $G$ is reductive, the extension $0\to G^0 \to G \to G^1\to 0$ splits after quotienting by a finite central subgroup: there is a central finite group $Z\subset G^0$ and a homomorphism $G\to \bar{G}^0=G^0/Z$ whose restriction to $G^0$ is the obvious quotient map.

Let $\bar{G}=G^1\times \bar{G}^0\cong G/Z$.  
\notation{$\bar{G}^0,\bar{G}$}{Central quotient of $G^0$ and the resulting quotient $\bar{G}=G^1\times \bar{G}^0$}\begin{Lemma}\label{lem:unfaithful}
    The Coulomb branch $\M(\bar{G},\bfN)\cong \M(G^1,\bfN)\times \M(\bar{G}^0,0)$ carries a free action of the Pontryagin dual $Z^{\vee}$ induced by actions on the two factors, and $\M(G,\bfN)$ is the quotient by this action.  
\end{Lemma}
\begin{proof}
The product decomposition arises from \cite[\S 3(vii)(a)]{BFN} and the action of $Z^{\vee}$ from \cite[\S 3(vii)(c)]{BFN}.  Since the map $\pi_1(\bar{G})\to Z$ is induced by maps on the two factors, the same is true for the induced action on $\M(\bar{G},\bfN)$.
 The action on $\M(\bar{G}^0,0)$, which is the universal centralizer in the Langlands dual $(\bar{G}^0)^{\vee}$, is free, so the action on $\M(\bar{G},\bfN)$ is also free.
\end{proof}

\begin{Remark}
    This result is necessarily more complicated in the symmetrizer setting, since the kernel of the representation is not necessarily compatible with the product decomposition of the group.  It seems quite likely that this is not a serious bar to extending the results of this section to the symmetrizer setting, but it seems that a more subtle analysis of the residual theory giving the slice is needed.  
\end{Remark}

Since $\M(\bar{G}^0,0)$ is smooth, every leaf of $\M(\bar{G}, \bfN)$ is of the form $L \times \M(\bar{G}^0,0)$ for some leaf $L$ of $\M(G^1,\bfN)$.  The freeness of the action of $Z^{\vee}$ on $\M(\bar{G}^0,0)$ implies that the image of such a leaf in $\M(G,\bfN)$ is again a leaf, rather than introducing new singularities along the quotient.  However, multiple leaves of $\M(G^1,\bfN)$ might map to the same leaf of $\M(G,\bfN)$.

This shows that:
\begin{Corollary}\label{cor:unfaithful}
    The leaves of $\M(G,\bfN)$ are in bijection with $Z^{\vee}$-orbits on the leaves of $\M(G^1,\bfN)$.
\end{Corollary}

\subsubsection{A monopole calculation}

In order to prove our desired leaf classification, we will need to do a small calculation of Poisson brackets with monopole operators, in the associated graded for the filtration by dominant coweights  given in \cite[\S 6(i)]{BFN}.  Note that in this subsection, the results hold in the symmetrizer setting as well.

In this filtration, the $\mu$ graded component $\soa_\mu$ consists of exactly the monopole operators $r_{\mu}(g)$ with dressings $g\in R^{W_\mu}$ for $W_{\mu}$ the stabilizer of $\mu$ in $W$.  The Poisson bracket with $\SymtW$ preserves this filtration, and thus induces a Poisson bracket $\SymtW\otimes \soa_\mu\to \soa_\mu$.  Note that the directional derivative $\partial_\mu f$ in the $\mu$ direction, considered as a vector field on $\LT$, preserves the subalgebra $R^{W_\mu}$.

\begin{Lemma}\label{lem:monopole-bracket}
    For $f,g\in R^{W_\mu}$, we have 
    \begin{equation}\label{eq:monopole-bracket}
        \{r_{m\mu}(f),r_{m'\mu}(g)\}=r_{(m+m')\mu}((\partial_{m'\mu} f)\cdot g-(\partial_{m\mu} g)\cdot f).
    \end{equation}
\end{Lemma}

\begin{proof}
    This follows from the Atiyah--Bott localization formula for the monopole operator $r_{m\mu}(f)$, modulo lower order terms; see \cite[(6.3)]{BFN}. The monopole operator, in its action on the GKLO representation, is given by the symmetrization of a rational function times an $\hbar$-scaled difference operator for $m\mu$ (corresponding to the $T$-fixed point $t^{m\mu}$ under Atiyah--Bott), plus a sum of translations which are lower in Bruhat order.  Thus, we have 
    \[[r_{m\mu}(f), r_{m'\mu}(g)]=r_{(m+m')\mu}((\delta_{\hbar\mu}f-f) g-(\delta_{\hbar\mu}g-g)\cdot f)\] plus lower order translations. 
    
    Taking semi-classical limit turns the difference operator into a directional derivative in the $m\mu$ direction, yielding \cref{eq:monopole-bracket}.
\end{proof}

\subsubsection{Monopole operators and ideals}

In this subsection, the results hold in the symmetrizer setting as well.
Let $I\subset \K[\M(G,\bfN)]$ be an ideal and $A=\K[\M(G,\bfN)]/I$ be the corresponding quotient.  The filtration by dominant coweights on $\K[\M(G,\bfN)]$ induces a filtration on $A$, with associated graded $\gr A$.  For each dominant coweight $\mu$, taking the dressed monopole operators gives a map $\mathsf{r}_{\mu}\colon R^{W_\mu}\to \gr A$.

Let $J_\mu\subset R^{W_\mu}$ be the kernel of this map. This is not obviously an ideal, but it is a $\SymtW$-submodule and if $n>n'$, then $ R^{W_\mu}J_{n'\mu}\subset J_{n\mu}$.  Since $R^{W_\mu}$ is Noetherian, both as a ring and as a module over $\SymtW$, this implies that: 

\begin{Lemma}\label{lem:monopole-ideal}
For $n\gg 0$, the subspace $J_{n\mu}$ stabilizes to an ideal $J^{\infty}_\mu$ of $R^{W_\mu}$.  If $I$ is a Poisson ideal, then $V(J^{\infty}_\mu)$ is invariant under translation by $\mu$ in $\LT/W_{\mu}$.  
\end{Lemma}
\begin{proof}
    The first claim follows from Noetherianity as above: for $n\gg 0$, we have $J_{n\mu}=J_{(n+1)\mu}$, which implies that $R^{W_\mu}J_{n\mu}= J_{n\mu}$.

    For the second claim, note that if $I$ is a Poisson ideal, then $\{I,r_\mu\}\subset I$.  By \cref{eq:monopole-bracket}, we have that $\{r_{n\mu}(f), r_\mu\}\subset r_{n\mu}(\partial_{\mu} f)$ for all $f\in R^{W_\mu}$ and $n\gg 0$.  Thus, we find that if $f\in J^{\infty}_\mu$, then  $\partial_{\mu} f\in J^{\infty}_\mu$ as well. This shows the claim about translation invariance of $V(J^{\infty}_\mu)$.
\end{proof}

\subsection{Leaves and flats}
\label{sec:leaves-and-flats}

Given a leaf $L$ of $\M(G,\bfN)$, the image $\ints(\overline{L})$ is a closed subvariety of $\LT/W$, which we denote $U_L$.  We will analyze leaves by considering the question of which subsets of $\LT/W$ can occur as $U_L$ for various leaves $L$.  The analysis below shows that these images are controlled by the hyperplanes defined by the weights of $\bfN$.  

Algebraically, leaves are expressed as the zero sets of prime Poisson ideals.  Thus, given a leaf $L$, let $I$ be the ideal of functions vanishing on the closure $\overline{L}$ of $L$,  
and let $A=\K[\M(G,\bfN)]/I$ be the corresponding quotient.  

The map $\K[\LT/W]\to A$ has a kernel $I_0\subset \SymtW$;  the zero set $V(I_0)$ is the image of the closed Poisson subscheme $V(I)=\Spec A\subset \M(G,\bfN)$, and in particular, is integral.  Note that since the fibers of the map $V(I)\to V(I_0)$ are coisotropic, we have $k=\dim V(I_0)\leq \frac{1}{2}\dim V(I)$.  

As in \cref{subsec:setup-and-notation}, we enumerate the lines $\ell_i$ spanned by the weights of $\bfN$.  Let $\ell_i^{\perp}$ be the corresponding hyperplanes in $\base{T}=\LT$.
\begin{Definition}\label{def:flat}
    A {\bf flat} of the weight hyperplane arrangement in $\LT$ is a subspace defined as the intersection of a subset of the hyperplanes $\ell_i^{\perp}$, including the empty intersection $\LT$.
\end{Definition}
Note that we have {\it not} included the root hyperplanes in this definition.  While the fiber of $\ints$ can change over these hyperplanes, we will see below that this never happens in a way that leads to a new leaf.  \cref{cor:unfaithful} establishes the key observation:  if the action of $G$ on $\bfN$ is not almost faithful, then there are no zero-dimensional leaves.  
The next lemma shows that only these flats can occur as images of leaves.
\begin{Lemma}\label{lem:leaf-gives-flat}
    Given a Poisson prime ideal $I$, the subvariety $V(I_0)$ is the image in $\LT/W$ of a $k$-dimensional flat of $\LT$ and $\dim V(I)=2k$.
\end{Lemma}

\begin{proof}
    Given a generic point $\lambda$ in $V(I_0)$, by \cref{lem:slice-ours}, we may replace $G$ with $\Gl$ and $\bfN$ with $\Nl$, and thereby assume that the corresponding cocharacter $\lambda$ is central and acts trivially on $\bfN$.  As before, let $\Gl^0$ be the identity component of the kernel of the action of $\Gl$ on $\Nl$, and let $k'$ be the rank of $\Gl^0$.  Equivalently, $k'$ is the dimension of the smallest flat $U$ containing $\lambda$.  Note that this shows that $V(I_0)\subset U$, since otherwise, a generic point would not lie in $U$.  Thus, $k=\dim V(I_0)\leq k'$.  Thus, the same is true of all $V(J_{\mu}^{\infty})$ for $\mu$ a dominant coweight of $\Gl$.  By \cref{lem:monopole-ideal}, $V(J_{\mu}^{\infty})$ is invariant under translation by $\mu$, so if $\mu\notin U$, this is only possible if $V(J_{\mu}^{\infty})$ is empty, and so all dressed monopole operators $r_{\mu}(g)$ map to zero in $A$.  By Noether normalization, this shows that the fiber over a point in $V(I_0)$ must have dimension $\leq k'$.  

    On the other hand, choosing a splitting $\LT\to \LT^0_\lambda$ of the inclusion $\LT^0_\lambda\subset \LT$ such that simple roots pull back to simple roots, we can define a map $\K[\M(\Gl^0,0)]\to A$ by sending $r_{\mu}(g)\mapsto r_{\mu}(g)$ for $\mu\in ^0_\lambda$, with the map $\K[\LT^0_\lambda]\to \K[\LT]$ induced by the splitting chosen above. This induces a Poisson map $\M(\Gl,\Nl)\to \M(\Gl^0,0)$.  Since the target is symplectic of dimension $2k'$, any symplectic leaf in the source must have dimension $\geq 2k'$.  This implies that the fiber of $V(I)$ over $\lambda$ has dimension $\geq k'$, since it is coisotropic.  This is only possible if $k=k'$, and thus $V(I_0)=U$ (since both are irreducible).  This completes the proof that $V(I_0)$ is the image of a flat.  
\end{proof}
Thus, the classification problem now splits into a combinatorial step and a geometric step.  First we enumerate the images of flats of the weight hyperplane arrangement in $\LT$; these are in bijection with the $W$-orbits of flats in $\LT$.  Then, for each such flat, we determine the set of leaves $L$ such that $U_L$ is the image of that flat.

\subsection{Analysis of generic fibers}

Let $U\subset \LT/W$ be the image of a flat $U_{\lambda}$ in $\LT$ with generic point $\lambda$.  We wish to understand the leaves $L$ such that $U_L=U$.  

The idea is that once the flat is fixed, the remaining ambiguity is entirely in the residual Coulomb branch attached to the fixed subrepresentation $\Nl$.  The next two lemmas isolate the two finite-group quotients that appear in passing from that residual Coulomb branch back to $\M(G,\bfN)$.

Consider the centralizer $\Gl=C_G(\lambda)$, the subrepresentation $\Nl$ of $\Gl$, and the stabilizer $\Gamma_{\lambda}$.  Note that $\Gamma_{\lambda}$ contains the Weyl group $W_{\lambda}$ of $\Gl$, but could potentially be larger.  The subgroup $W_{\lambda}\subset \Gamma_{\lambda}$ is normal.  
We know that the base change of $\M(G,\bfN)$ to $\base{G,\lambda}$ is isomorphic to the base change of $\M(\Gl,\Nl)$ to the same set.  
Thus, a leaf $L$ such that $U_L=U$ must be the image of a leaf of $\M(\Gl,\Nl)$ whose image is $U_{\lambda}/W_{\lambda}$.  
This set of leaves in $\M(\Gl,\Nl)$ carries an action of $\bar{\Gamma}_{\lambda}:=\Gamma_{\lambda}/W_{\lambda}$ induced by the action of $W$ on $\M(T,\bfN)$, and two such leaves have the same image in $\M(G,\bfN)$ if and only if they lie in the same $\bar{\Gamma}_{\lambda}$-orbit.  
The group $\bar{\Gamma}_{\lambda}$ is also the Galois group of the generic cover $U_{\lambda}/W_{\lambda} \to U$.  Thus, we have:
\notation{$\Gamma_{\lambda},W_{\lambda},\bar{\Gamma}_{\lambda}$}{Stabilizer of $\LT_\lambda\subset \LT$ in $W$, the Weyl group of $\Gl$ and the quotient $\bar{\Gamma}_{\lambda}=\Gamma_{\lambda}/W_{\lambda}$}\begin{Lemma}\label{lem:leaf-from-Levi}
    The set of leaves of $\M(G,\bfN)$ such that $U_L=U$ is the quotient of the set of leaves of $\M(\Gl,\Nl)$ such that $U_{L'}=U_{\lambda}/W_{\lambda}$ by the action of $\bar{\Gamma}_{\lambda}$.
\end{Lemma} 

\begin{Lemma}\label{lem:closure-relation}
    A leaf $L$ lies in the closure of another leaf $L'$ if and only if the corresponding closure relation holds with at least one choice of preimages in $\M(\Gl,\Nl)$.  In particular, if $L\subset \overline{L'}$, then we have $U_L\subset U_{L'}$ for the corresponding flats.
\end{Lemma}
\begin{proof}
    Since the leaf $L$ has non-trivial intersection with the open set $V$ obtained by removing all $D_{\ell_i}$ for lines of weights that do not vanish on $U$, the closure $\overline{L}$ also has non-trivial intersection with this open set, so we may check closure relations after intersecting with $V$.  
\end{proof}

Now, we assume that we are symmetrizer-free.  In order to understand the leaves of $\M(\Gl,\Nl)$, we can apply \cref{cor:unfaithful} to the action of $\Gl$ on $\Nl$. Just as in \cref{sec:unfaithful}, we can choose a central subgroup $Z\subset \Gl$ such that the quotient $\bGl=\Gl/Z$ is a product $\Gl^1\times \bGl^0$ and find that $\M(\bGl,\Nl)\cong \M(\Gl^1,\Nl)\times \M(\bGl^0,0)$ carries a free action of $Z^{\vee}$ such that $\M(\Gl,\Nl)$ is the quotient by this action.  This behavior is particularly simple when $(\Gl^1,\Nl)$ is good in the usual sense of the literature:
\notation{$\bGl$}{Quotient $\bGl= \Gl^1\times \bGl^0$ of $\Gl$ splitting off the almost-faithful factor}\begin{Definition} Recall that $\M(G,\bfN)$ is called {\bf good} if the space of functions of degree $\leq 1$ on $\M(G,\bfN)$ is precisely the scalars.  We call $U$ {\bf good} if the Coulomb branch $\M(\Gl^1,\Nl)$ is good in the usual sense.  In particular, this means $\M(\Gl^1,\Nl)$ has a unique zero-dimensional leaf.
\end{Definition}

Pulling this all together, we can combine these two quotient operations using the fact that $\bar{\Gamma}_{\lambda}$ naturally acts on $Z$ and thus $Z^{\vee}$, and so we can form the semidirect product $Z^{\vee}\rtimes \bar{\Gamma}_{\lambda}$.  This is a finite group that acts on $\M(\bGl,\Nl)$ preserving the product decomposition and thus also on  $\M(\Gl^1,\Nl)$.
\begin{Lemma}\label{leaf-bijection} Assume that we are symmetrizer-free.
    The leaves $L$ such that $U_L = U$ are in canonical bijection with $Z^{\vee}\rtimes \bar{\Gamma}_{\lambda}$-orbits on the zero-dimensional leaves of $\M(\Gl^1,\Nl)$.  The closure of another leaf $L'$ contains $L$ if and only if the same holds for at least one choice of preimages in $\M(\Gl^1,\Nl)$.

    In particular, if $U$ is good, then there is a unique leaf $L$ with $U_L = U$.  
    If $U$ and $U'$ are both good, then the converse to \cref{lem:closure-relation} also holds, so if $U\subset U'$, then $L\subset \overline{L'}$, where $L$ and $L'$ are the unique leaves with images $U_L=U$ and $U_{L'}=U'$.
\end{Lemma}
\begin{proof}
    The first claim follows from the discussion above, together with the fact that the action of $Z^{\vee}$ on $\M(\bGl^0,0)$ is free, so the zero-dimensional leaves of $\M(\Gl,\Nl)$ are in bijection with $Z^{\vee}$-orbits on the zero-dimensional leaves of $\M(\Gl^1,\Nl)$.  The second claim follows from \cref{lem:closure-relation} and the fact that if $U$ is good, then there is a unique zero-dimensional leaf in $\M(\Gl^1,\Nl)$, so there is no ambiguity in choosing preimages.
\end{proof}

\begin{Corollary}\label{cor:slice-to-leaf}
    Assume that we are symmetrizer-free.
    The slice to any leaf in $\M(G,\bfN)$ is always isomorphic to the slice to a zero-dimensional leaf in $\M(\Gl^1,\Nl)$.  In particular, if $U$ is good, then this slice is isomorphic to $\M(\Gl^1,\Nl)$.
\end{Corollary}

This gives a partial positive answer to \cite[Question 2.1(2)]{nakajimaQuestionsProvisional2015} concerning transverse slices.

Since the statements above are much simpler when $U$ is good, the next natural question is when $\M(G,\bfN)$ has zero-dimensional leaves, and in particular whether such leaves can exist when $\M(G,\bfN)$ is not good.  Following the standard terminology in the physics literature \cite{gaiottoSDualityBoundary2009}, we refine the notion of goodness to a trichotomy of theories into good, ugly, and bad:
\begin{enumerate}
    \item We call $(G,\bfN)$ ugly if the natural $\Gm$-action on $\M(G,\bfN)$ contracts to a fixed point, but this point is not a leaf;  an ugly theory must have a unique lowest-dimensional leaf, which is isomorphic to a symplectic vector space (in fact, the dual of the space of degree 1 functions on $\M(G,\bfN)$).
    \item We call $(G,\bfN)$ bad if the natural $\Gm$-action on $\M(G,\bfN)$ does not contract to a fixed point.
\end{enumerate}
Other related notions are:  we call $(G,\bfN)$ or the corresponding Coulomb branch {\bf conical} if it is good or ugly and {\bf weakly conical} if every point has a limit under the contracting $\Gm$-action, but this fixed point is not necessarily unique.  Recall that:
\begin{enumerate}
    \item The theory $(G,\bfN)$ is good if and only if the Hilbert series for the $\Gm$-action on the Coulomb branch is well-defined and of the form $1+O(t^2)$, which is equivalent to the condition that the monopole operator of each non-trivial coweight has degree at least 2.
    \item The theory $(G,\bfN)$ is conical if and only if the Hilbert series for the $\Gm$-action on the Coulomb branch is well-defined and of the form $1+O(t)$, which is equivalent to the condition that the monopole operator of each non-trivial coweight has degree at least 1.
    \item The theory $(G,\bfN)$ is weakly conical if and only if every monopole operator has degree at least 0.  It does not necessarily have a well-defined Hilbert series, since there could be infinitely many degree 0 monopole operators.
\end{enumerate}

The two classifications above are not independent: since good and ugly theories are both conical by definition, and conical theories are automatically weakly conical, the only freedom is in how a bad theory can fail to be conical.  This nesting is summarized in the following diagram.
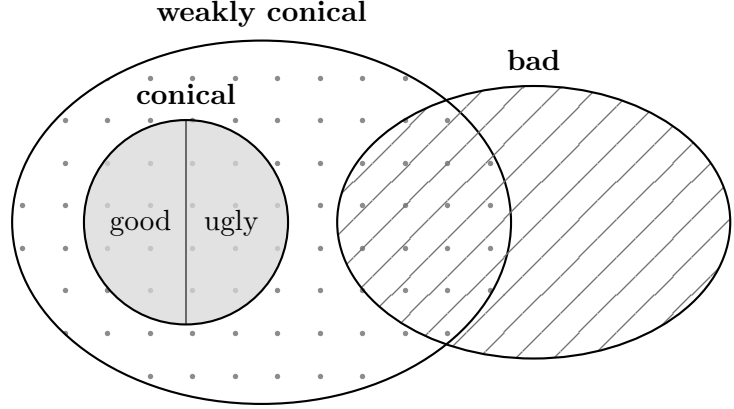
\begin{center}
\begin{tikzpicture}
\draw[thick, pattern={Dots[radius=1pt,distance=16pt,yshift=4pt]}, pattern color=black!45] (-1.3,0) ellipse (3.3 and 2.4);
\draw[thick, pattern={Lines[angle=45,distance=10pt,line width=0.5pt]}, pattern color=black!55] (2.3,0) ellipse (2.6 and 1.8);
\draw[thick, fill=black!15, fill opacity=0.75] (-2.3,0) circle (1.35);
\draw (-2.3,-1.35) -- (-2.3,1.35);
\node at (-1.3,2.75) {\textbf{weakly conical}};
\node at (2.3,2.15) {\textbf{bad}};
\node at (-2.3,1.7) {\textbf{conical}};
\node at (-2.9,0) {good};
\node at (-1.7,0) {ugly};
\end{tikzpicture}
\end{center}

An ugly theory manifestly has no zero-dimensional leaves.  For bad theories, however, both behaviors occur: some have zero-dimensional leaves, and some do not.
\begin{Example}\label{ex:bad}
Assume that $\bfN=\mathfrak{g}$.  In this case, the classical description of the Coulomb branch is correct and we have an isomorphism $\M(G,\fg)\cong T^*\check{T}/W$.  This is a bad but weakly conical theory, since the monopole operator of each non-trivial coweight has degree $0$.

In particular, if $G$ is simple, each fixed point of $W$ in $\check{T}$ is a zero-dimensional leaf. Thus, by taking $G=PGL_n$, we can arrange to have any number of zero-dimensional leaves in a bad Coulomb branch.
\end{Example}
More complicated examples of bad theories with multiple zero-dimensional leaves are given in \cite{bourgetTaleCones2023}.

\subsection{Comparison with the Higgs branch}
In this subsection, we assume that we are symmetrizer-free.
As mentioned in the introduction, it is expected that the symplectic leaves of $\M(G,\bfN)$ will correspond to the symplectic leaves of $\M_{\Higgs}(G,\bfN)$ with closure order reversed.  It is also well-known that this correspondence is not perfect, and examples such as affine type A quiver varieties show that the general result must be something other than a bijection between leaves.  

Our main result \cref{leaf-bijection} suggests a first approximation to such a correspondence: it associates to each leaf of $\M(G,\bfN)$ a flat, together with the extra data coming from the action of $Z^{\vee}\rtimes \bar{\Gamma}_{\lambda}$ on the zero-dimensional leaves of $\M(\Gl^1,\Nl)$.  On the other hand, the leaves of $\M_{\Higgs}(G,\bfN)$ are classified by conjugacy classes of stabilizer subgroups.    

Given a point $[x]\in \M_{\Higgs}(G,\bfN)$, choose a representative $x\in T^*\bfN$ with closed $G$-orbit, and let $G_x$ be its stabilizer.  We can assume that we have chosen a maximal torus $T$ of $G$ such that $G_x\cap T=T_x$ is a maximal torus of $G_x$.  In this case, $T_x$ is the common kernel of the weights on which $x$ has non-trivial component, so its Lie algebra is a flat.  For example, when $G=T$ is a torus, this recovers the usual toric-arrangement picture: $\LT_x$ is cut out by the weights that appear in $x$.
\notation{$G_x$}{Stabilizer in $G$ of a point $x\in T^*\bfN$}

Let $\Leaf{\M_{\Coulomb}}$ and $\Leaf{\M_{\Higgs}}$ be the sets of leaves of $\M(G,\bfN)$ and $\M_{\Higgs}(G,\bfN)$, respectively.
\notation{\vtop{\hbox{$\Leaf{\M_{\Coulomb}}$,}\hbox{$\Leaf{\M_{\Higgs}}$}}}{Sets of symplectic leaves on the Coulomb and Higgs branches}

\begin{Definition}
    Define $\CH\subset \Leaf{\M_{\Coulomb}}\times \Leaf{\M_{\Higgs}}$ by declaring that $(L,L')\in \CH$ if, for a generic point $[x]\in L'$, the set $U_L$ is the image of $\LT_x$ in $\LT/W$.  This is independent of the choice of point, since the stabilizer of a generic point is constant along a leaf up to conjugacy.
\end{Definition}
\notation{$\CH$}{Relation between Coulomb and Higgs leaves defined by matching flat images}This correspondence need not be the graph of a bijection.  There are several issues that can arise:
\begin{enumerate}
    \item There may be distinct leaves $L,L'$ of $\M(G,\bfN)$ with $U_L=U_{L'}$, so all of these leaves will be related to the same leaves of $\M_{\Higgs}(G,\bfN)$.
    \item A given flat may fail to occur as $\LT_x$ for any leaf of $\M_{\Higgs}(G,\bfN)$, so a leaf of $\M(G,\bfN)$ may not be related to any leaf of $\M_{\Higgs}(G,\bfN)$.
    \item There may be multiple leaves of $\M_{\Higgs}(G,\bfN)$ whose corresponding stabilizer has the same maximal torus, so all of these leaves will be related to the same leaves of $\M(G,\bfN)$.
\end{enumerate}
Issue (1) is already visible in the case discussed in \cref{ex:bad}.  Issues (2) and (3) occur for $G=GL_n,\bfN=\C^n\oplus \mathfrak{gl}_n$, discussed in more detail in \cref{ex:Hilbert,ex:Hilbert-ellgt1}. 
These examples suggest that the correct comparison should single out a distinguished subset of the leaves, called ``special,'' rather than all leaves at once.  A definition of special leaves is given in \cite[\S 6.3]{BLPWgco}, but this discussion suggests a possible alternate definition.
\begin{Conjecture}\label{conj:special}
    The special leaves of $\M(G,\bfN)$ are precisely those whose image under $\ints$ is a good flat.  The relation $\CH$ gives a bijection between the special leaves of $\M(G,\bfN)$ and the special leaves of $\M_{\Higgs}(G,\bfN)$.
\end{Conjecture}
In many examples, this relation gives a bijection between the leaves of $\M(G,\bfN)$ and the leaves of $\M_{\Higgs}(G,\bfN)$, but in others it does not.  

\begin{Example}\label{ex:abelian}
    If $G$ is a torus with unimodular action, then the theory $(\Gl^1,\Nl)$ is good if and only if the corresponding flat is coloop-free and ugly otherwise, so there is a unique leaf of $\M(G,\bfN)$ corresponding to each coloop-free flat.  The leaves of $\M_{\Higgs}(G,\bfN)$ are also classified by these flats by \cite[\S 2]{proudfootIntersectionCohomology2007}.  Thus, in this case, $\CH$ gives the well-known bijection between the leaves of $\M(G,\bfN)$ and $\M_{\Higgs}(G,\bfN)$.  In particular, every leaf is special, so \cref{conj:special} holds by \cite[Th. 10.8]{BLPWgco}.
\end{Example}
\begin{Example}\label{ex:adjoint}
    If $G=GL_n$ and $\bfN=\mathfrak{gl}_n$, then $\M(G,\bfN)$ is the $n$th symmetric power of $\C\times \C^*$, with the integrable system given by the projection to $\C$.  The flats are given by loci in $\C^n/S_n$ obtained by choosing a set partition corresponding to a partition $\lambda$ of $n$ and requiring that the coordinates in each part be equal.  In this case, the Levi subgroup $\Gl$ corresponds to the partition $\lambda$, and $\Nl$ is the sum of the adjoint representation of each factor.  In particular, the identity component $\Gl^0$ of the kernel of the action of $\Gl$ on $\Nl$ is the center of this Levi, which is a torus of rank equal to the number of parts of $\lambda$, and $\Gl^1=\prod PGL_{n_i}$, where $n_i$ are the sizes of the parts of $\lambda$.  
    
    Thus, the Coulomb branch $\M(\Gl^1,\Nl)$ is a product \[\prod \M(PGL_{n_i},\mathfrak{pgl}_{n_i})=\prod T^*(\C^{\times})^{n_i}_0/S_{n_i}\] where $(\C^{\times})^{n_i}_0$ denotes the determinant-1 diagonal $n_i\times n_i$ matrices.  The zero-dimensional leaves of this Coulomb branch are the $(\prod_i S_{n_i})$-fixed points, of which there are $\prod n_i$, obtained by choosing an order-$n_i$ scalar matrix for each factor.  The group $Z^{\vee}$ acts simply transitively on these zero-dimensional leaves, namely on the same set of block scalar matrices, so there is a unique leaf of $\M(G,\bfN)$ corresponding to the flat associated to $\lambda$.
\end{Example}
\begin{Example}\label{ex:Hilbert}
    If $G=GL_n$ and $\bfN=\C^n\oplus \mathfrak{gl}_n$, then the Coulomb branch $\M(G,\bfN)$ and the Higgs branch $\M_{\Higgs}(G,\bfN)$ are the $n$th symmetric power of $\mathbb{A}^2$, with the integrable system given by taking the product of the two coordinates in each factor of $\mathbb{A}^2$.

    The flats include all the ones discussed in \cref{ex:adjoint}, but now there are also weight hyperplanes corresponding to one coordinate being zero.  Equivalently, to specify a flat one chooses an integer $n_0\leq n$, recording how many coordinates are forced to be zero, together with a partition $\lambda$ of $n-n_0$, recording how the remaining non-zero coordinates are grouped into equal blocks.  We then have $\Gl=GL_{n_0}\times GL_{n_1}\times \cdots \times GL_{n_k}$ where $\lambda=(n_1, \dots, n_k)$, and $\Gl^1=GL_{n_0}\times PGL_{n_1}\times \cdots \times PGL_{n_k}$, acting on $\C^{n_0}$ plus the adjoint representations of $\mathfrak{gl}_{n_i}$.

    If $n_0>0$, then $(GL_{n_0}, \C^{n_0}\oplus \mathfrak{gl}_{n_0})$ is ugly; indeed, $\Sym^{n_0}\mathbb{A}^2$ has a positive-dimensional minimal stratum isomorphic to $\mathbb{A}^2$.  Hence, this theory has no zero-dimensional leaves, so we only get Coulomb leaves in the case $n_0=0$, where we recover exactly the same description of the leaves as in \cref{ex:adjoint}.

    On the Higgs side, the stabilizers of points are given by Levi subgroups corresponding to set partitions of $n$, so the relevant Higgs leaves are also indexed by partitions of $n$.  The groups $G_x$ are thus all Levi subgroups, and each contains a maximal torus, so they all correspond to $\LT$ as a flat.  Thus, every Higgs leaf is related to the open leaf in $\M(G,\bfN)$, and no other leaves of $\M(G,\bfN)$ are related to any leaves of $\M_{\Higgs}(G,\bfN)$.

    This meshes with the well-known fact that the closure order on the leaves of the symmetric power of $\mathbb{A}^2$ is given by the refinement order on partitions, so there is no order-reversing bijection between the leaves of $\M(G,\bfN)$ and $\M_{\Higgs}(G,\bfN)$.
\end{Example}

\begin{Example}\label{ex:Hilbert-ellgt1}
    We can generalize further by considering $G=GL_n$ and $\bfN=(\C^n)^{\oplus \ell} \oplus \mathfrak{gl}_n$ for $\ell>1$.  In this case: 
    \begin{itemize}
        \item The Coulomb branch $\M(G,\bfN)$ is the $n$th symmetric power of $\mathbb{A}^2/\Z_{\ell}$, with the integrable system given by taking the product of the two coordinates in $\mathbb{A}^2$.
        \item The Higgs branch $\M_{\Higgs}(G,\bfN)$ is the Gieseker quiver variety, related to the moduli of rank $\ell$ torsion-free sheaves on $\mathbb{P}^2$ framed at $\infty$ with second Chern class $n$. 
    \end{itemize}
    
    Since $\ell>1$, we have that $(GL_{n_0}, (\C^{n_0})^{\oplus \ell}\oplus \mathfrak{gl}_{n_0})$ is good for all $n_0$.  So, compared with \cref{ex:Hilbert}, the combinatorics of the flats is the same, but the residual theory at the origin is now good rather than ugly.  Combined with the structure of the $PGL_{n_i}$ factors, this shows that there is a unique leaf of $\M(G,\bfN)$ corresponding to each flat, that is, to each choice of $n_0$ and $\lambda$.  In terms of the symmetric power of $\mathbb{A}^2/\Z_{\ell}$, the flat corresponding to $n_0$ and $\lambda$ is the image of the locus where $n_0$ of the points are at the origin of $\mathbb{A}^2/\Z_{\ell}$, and the remaining points are grouped according to $\lambda$.

    On the Higgs side, the stabilizers of a generic point of a leaf of $\M_{\Higgs}(G,\bfN)$ are of the form $\C^*\times GL_{n_1}\times \cdots \times GL_{n_k}$, where the $\C^*$ factor corresponds to the scalar matrices in $GL_{n_0}$ and the point in the Higgs branch defines a simple non-trivial representation of the Gieseker variety of rank $n_0$ and second Chern class $\ell$.  So, the only Coulomb branch leaves that can be related to Higgs leaves are those corresponding to some $n_0$ and $\lambda=(1,\dots,1)$, that is, the ones where the non-zero coordinates are all distinct.  As in previous examples, many Coulomb leaves have no Higgs partner under $\CH$.
\end{Example}

We can refine this correspondence by asking how the geometry of slices and leaves compares on the two sides.  Fix a flat $U$ with generic point $\lambda$, and suppose that $U$ corresponds to some leaves in $\M_{\Higgs}(G,\bfN)$.  Choose a leaf $L$ maximal among these, and let $x$ be a generic point of a component of $\mu^{-1}(0)\cap T^*\Nl$ representing $L$.  Then the stabilizer $G_x$ has maximal torus $T_x$ with Lie algebra $U$.   

The space $T^*\Nl$ carries a natural action of $N_G(T_x)$, but this action is not faithful, since $T_x$ acts trivially.  Let $K$ be its kernel, so $K$ is a normal subgroup of $N_G(T_x)$ containing $T_x$.  Furthermore, the Lie algebra of $N_G(T_x)$ is precisely $\LG^{T_x}$, and the restriction of the moment map for the action of $G$ on $T^*\bfN$ to $T^*\Nl$ naturally lands in $(\LG^*)^{T_x}$.  This gives the moment map for the action of $N_G(T_x)$ on $T^*\Nl$.  

We can thus consider the Higgs branch $\M_{\Higgs}(N_G(T_x)/K,\Nl)$, which is naturally equipped with a Poisson map $\M_{\Higgs}(N_G(T_x)/K,\Nl)\to \M_{\Higgs}(G,\bfN)$. 
\begin{Lemma}
        The leaf $L$ in $\M_{\Higgs}(G,\bfN)$ is isomorphic to an open leaf in the Higgs branch $\M_{\Higgs}(N_G(T_x)/K,\Nl)$.
\end{Lemma}
\begin{proof}
 We can also write $\Nl=\bfN^{T_x}$.  Since $T_x\subset G_x$, every point of $L$ has a representative in $T^*\Nl$; indeed, $T^*\Nl\supset T^*\bfN^{G_x}$.  Let $V$ be the preimage of $L$ in $T^*\Nl$.      

If two points $x,x'\in V$ are conjugate by an element $g\in G$, then $gT_xg^{-1}$ is a maximal torus of $G_{x'}$, so there is an element $h\in G_{x'}$ such that $hgT_xg^{-1}h^{-1}=T_{x'}$.  Thus $hg\in N_G(T_x)$ and $hg\cdot x=x'$.  Conversely, if $x$ and $x'$ lie in the same orbit of $N_G(T_x)$, then they certainly lie in the same orbit of $G$.
Since $K$ acts trivially on $T^*\Nl$, the orbits of $N_G(T_x)$ and $N_G(T_x)/K$ on $T^*\Nl$ are the same.  Hence $L$ embeds in $\M_{\Higgs}(N_G(T_x)/K,\Nl)$, and by the maximality of $L$, its image is an open leaf there.  
\end{proof}

Note that we have a natural map $\Gl\to N_G(T_x)$, and that the preimage of $K$ under this map is the kernel of the action of $\Gl$ on $\Nl$.  Its identity component is therefore $\Gl^0$, so we obtain an induced map $\Gl^1\to N_G(T_x)/K$.  We have already established that this map is an isomorphism on Lie algebras, so it induces an isomorphism on identity components.  In particular, we have a map $\M_{\Higgs}(\Gl^1,\Nl)\to \M_{\Higgs}(N_G(T_x)/K,\Nl)$, which is a finite cover on each open leaf of $\M_{\Higgs}(N_G(T_x)/K,\Nl)$.  Thus:
\begin{Corollary}
    Each leaf $L$ which is maximal among those corresponding to a flat $U$ has a finite cover by an open leaf in $\M_{\Higgs}(\Gl^1,\Nl)$.
\end{Corollary}

This is an approximate answer to the slice portion of \cite[Question 2.6]{nakajimaQuestionsProvisional2015}.  Up to taking connected components and dealing with bad cases, the slice to a leaf $L$ in $\M(G,\bfN)$ is the Coulomb branch of a theory whose Higgs branch is approximately the corresponding dual leaf.   

We also expect a dual version of this statement, describing a leaf of the Coulomb branch in terms of the Coulomb branch of a modified theory, and similarly for Higgs branches.  On the Higgs side, this process is well-understood: acting on $x$ gives a map of quotient adjoint representations $\mathsf{a}\colon\LG/\LG_x\to T^*\bfN$, and the slice to the leaf at $x$ is isomorphic to the reduction by $G_x$ of the symplectic orthogonal complement of the image of $\mathsf{a}$.
This suggests there is an analogous description of the leaf in the Coulomb branch in terms of the same theory, replacing $G_x$ by the connected kernel $\Gl^0$.

 \subsection{Quiver varieties}

This calculation is particularly clean for quiver gauge theories.  Let $I$ be the set of vertices of the quiver, and let $i\to j$ range over its arrows.  For convenience, let \[\GL_{\Bv}=\prod_{i\in I}GL_{v_i}\qquad \PGL_{\Bv}=\GL_{\Bv}/\C^{\times}\qquad \bfN_{\Bv}=\bigoplus_{i\to j}\Hom(\C^{v_i},\C^{v_j})\qquad V=\bigoplus_{i\in I}\C^{v_i}.\]
\notation{$\GL_{\Bv},\PGL_{\Bv},\bfN_{\Bv}$}{Quiver gauge group, its adjoint quotient, and the arrow representation for dimension vector $\Bv$}

 In this case, a cocharacter of $\PGL_{\Bv}$ is given by a grading on the representation $V$, up to an overall grading shift coming from the quotient by $\C^{\times}$ in the definition of $\PGL_{\Bv}$.  This gives a decomposition of the dimension vector as a sum $\Bv=\Bv^{(1)}+\cdots +\Bv^{(m)}$.  The residual theory corresponding to this cocharacter can be written as
 \[\Nl=\bigoplus_{i}\bfN_{\Bv^{(i)}} \qquad \qquad\Gl^1=\prod \PGL_{\Bv^{(i)}}\]
so the smaller Coulomb branch that describes the slice is given by \[\M(\Gl^1,\Nl)\cong \prod \M(\PGL_{\Bv^{(i)}},\bfN_{\Bv^{(i)}}). \]

Thus, to understand the leaves of Coulomb branches of quiver gauge theories, it suffices to understand zero-dimensional leaves in the residual theories.  

Recall that our quiver determines a symmetric generalized Cartan matrix $C$.
We can define a bilinear form on the space of dimension vectors $\Z^{I}$ by $\langle \Bv,\Bv'\rangle=\Bv^T C \Bv'$.  This form controls the goodness and conicality of the residual theories as follows:
\begin{Lemma}[\mbox{\cite[Lem. A.5]{muthiahFundamentalMonopole2024}}]
    The degree of the monopole operator for the fundamental weight is  $\deg \omega_{\mathbf{m}}=-\langle \Bm,\Bv-\Bm\rangle$.  Thus, the theory $(\PGL_{\Bv},\bfN_{\Bv})$ is good if and only if $\langle \Bm,\Bv-\Bm\rangle\leq -2$ for all $\Bm$ strictly between $0$ and $\Bv$.
\end{Lemma}
Recall that if $C$ is symmetric, so that the quiver carries no additional symmetrizer data, there is a set of dimension vectors $\Sigma_0$, defined in \cite{crawley-boeveyGeometryMoment2001} (see also \cite{bellamyCoulombBranches2023}); these are precisely the dimension vectors of the simple representations of the preprojective algebra by \cite[Th. 1.2]{crawley-boeveyGeometryMoment2001}.
\notation{$\Sigma_0$}{Set of dimension vectors giving simple preprojective representations}\begin{Lemma}
    If $C$ is symmetric, then $(\PGL_{\Bv},\bfN_{\Bv})$ is good if and only if $\Bv \in \Sigma_0$, and $(\PGL_{\Bv},\bfN_{\Bv})$ is conical if and only if the moment map $T^*\bfN_{\Bv}\to \mathfrak{pgl}_{\Bv}$ is flat.
\end{Lemma}
Nakajima observes this relationship between goodness, conicality, and flatness of the moment map in \cite[\S 2(iv)]{NaCoulomb}.  We do not know whether it continues to hold more generally.
\begin{proof}
    By \cite[Th. 5.6]{crawley-boeveyGeometryMoment2001}, we have that $\Bv \in \Sigma_0$ if and only if $\deg \omega_{\mathbf{m}}=-\langle \Bm,\Bv-\Bm\rangle\geq 2$.  This proves that $\Bv \in \Sigma_0$ if and only if the theory is good as discussed in \cite[Cor. A.6]{muthiahFundamentalMonopole2024}.

    Similarly, this shows that the Coulomb branch is conical if and only if $\langle \Bm,\Bv-\Bm\rangle \leq -1$ for all subvectors $0< \Bm< \Bv$.  Rewriting this using the polarized form $p(\Bv)=1-\frac{1}{2}\langle \Bv,\Bv\rangle$, we find that
     \[\deg \omega_{\mathbf{m}}=p(\Bv)-p(\Bv-\Bm)-p(\Bm)+1. \]

    Thus, conicity holds if and only if $p(\Bv)\geq p(\Bv-\Bm)+p(\Bm)$, and we have seen above that goodness holds if and only if this inequality is always strict.  If this inequality is violated, the moment map is not flat by \cite[Th. 1.1]{crawley-boeveyGeometryMoment2001}.  

     On the other hand, if the moment map is not flat, then consider a decomposition $\Bv=\Bv^{(1)}+\cdots +\Bv^{(r)}$ that maximizes $p(\Bv^{(1)})+\cdots + p(\Bv^{(r)})$, with minimal $r$.  Note that $r>1$, by the failure of flatness. If $\langle \Bv^{(i)},\Bv^{(j)}\rangle \leq -1$, then we can change this decomposition by replacing $\Bv^{(i)}$ and $\Bv^{(j)}$ by their sum.  This weakly increases $p(\Bv^{(1)})+\cdots + p(\Bv^{(r)})$ and strictly decreases $r$.  Thus, we must have $\langle \Bv^{(i)},\Bv^{(j)}\rangle \geq 0.$ 

     This shows that 
     $\deg \omega_{\mathbf{v}^{(1)}}=-\langle \Bv^{(1)},\Bv^{(2)}+\cdots +\Bv^{(r)}\rangle \leq 0$, so conicity fails.
\end{proof}

\begin{Remark}\label{rem:ugly-no-affine}
To illustrate this criterion, we can also use it to identify ugly theories, where the lowest monopole degree is exactly $1$ and the Coulomb branch has a positive-dimensional minimal stratum.  Consider the four-node chain in which the two end nodes each carry an edge-loop and the number of edges connecting successive pairs is $2,1,2$, so that the Cartan matrix and dimension vector are
\[
C=\begin{bmatrix} 0 & -2 & 0 & 0\\ -2 & 2 & -1 & 0\\ 0 & -1 & 2 & -2\\ 0 & 0 & -2 & 0\end{bmatrix},\qquad \Bv=(2,1,1,2).
\]
In this case, the cocharacter $\Bm=(2,1,0,0)$ with $\Bv-\Bm=(0,0,1,2)$, satisfies
\[
\deg\omega_{\Bm}=-\langle \Bm,\Bv-\Bm\rangle = -\Bm^{\top}C(\Bv-\Bm)=1,
\]
so $\Bv\notin\Sigma_0$ and the theory is ugly.
\end{Remark}

\begin{Corollary}
    For each unordered tuple $\Bv^{(*)}=(\Bv^{(1)},\ldots,\Bv^{(d)})$ such that $\Bv=\Bv^{(1)}+\cdots +\Bv^{(d)}$ with $\Bv^{(i)}\in \Sigma_0$, we have a unique $d$-dimensional leaf $L_{\Bv^{(*)}}$ of $\M(\PGL_{\Bv},\bfN_{\Bv})$ such that $U_{L_{\Bv^{(*)}}}=\LT_{\Bv^{(*)}}$ is the flat corresponding to the decomposition and whose transverse slice is $\prod_i \M(\PGL_{\Bv^{(i)}},\bfN_{\Bv^{(i)}})$.

By \cref{lem:closure-relation}, the inclusions between these leaves are induced by the corresponding inclusions of flats, which in turn come from refinement of the decomposition above.  That is, $L_{\Bv^{(*)}}\subset \overline{L_{\Bw^{(*)}}}$ if and only if $\Bw^{(*)}$ can be obtained by refining $\Bv^{(*)}$.  
\end{Corollary}

More generally, every leaf $L$ gives such a decomposition induced by a generic point of $U_L$, with the number of summands matching the dimension of $L$; outside of certain special cases, we cannot rule out the possibility that some of the vectors appearing are not in $\Sigma_0$ but still have zero-dimensional leaves, perhaps more than one.

\subsubsection{Quivers with symmetrizers}

For a symmetrizable generalized Cartan matrix, we have a corresponding Coulomb branch with symmetrizers.  We cannot apply the full classification from \cref{sec:leaves-and-flats} directly, especially the reduction to zero-dimensional leaves.  However, the flatness constraint and Levi-localization argument still apply, and this already constrains the possibilities considerably.  Let $L$ be a $2d$-dimensional leaf in the Coulomb branch of a quiver with symmetrizers.  Choosing a lift to $\LT$ of a generic point of $U_L$ determines a decomposition of $\Bv$ as a sum of dimension vectors.  Moreover, the same argument as in \cref{lem:leaf-from-Levi} shows that the leaf $L$ must arise from a $d$-dimensional leaf $L'$ of the Coulomb branch of $(\Gl,\Nl)$, where $\Gl=\prod \GL_{\Bv^{(i)}}$ and $\Nl=\oplus \bfN_{\Bv^{(i)}}$.  This Coulomb branch factors as a product over the summands:
\[\M(\Gl, \Nl) =\prod_i \M(\GL_{\Bv^{(i)}}, \bfN_{\Bv^{(i)}} ).\]
Thus, $L'$ must be a product of leaves in these smaller Coulomb branches whose dimensions add up to $d$.  Note that we can assume that the support of each $\Bv^{(i)}$ is connected, since otherwise we can further decompose it as a sum of dimension vectors with disjoint support, and $\lambda$ was not generic in its flat.

\begin{Lemma}\label{lem:no-zero-dim-leaves}
    The variety $\M(\GL_{\Bv^{(i)}}, \bfN_{\Bv^{(i)}} )$ has no zero-dimensional leaves.  If the support of $\Bv^{(i)}$ is connected, then it has a unique one-dimensional flat.
\end{Lemma}
\begin{proof}
    The Coulomb branch $\M(G,\bfN)$ always carries an action of $G_{\ab}^{\vee}$, where $G_{\ab}=G/[G,G]$ is the abelianization of $G$.  
    In the case of $\GL_{\Bv^{(i)}}$, we have $G_{\ab}=(\C^{\times})^{I_i}$, where $I_i$ is the support of $\Bv^{(i)}$; the weight of a monopole operator is exactly the element of $\pi_1(G_{\ab})$ given by the cocharacter defining that operator.  Consider a primitive cocharacter $\nu$ such that $\nu$ is a power of the diagonal scalar cocharacter.  More explicitly, the component of $\nu$ in $\LG_k$ is the scalar matrix $(k'/k)I$, where $k'$ is the greatest common divisor of the symmetrizers $d_j$ for nodes in $I_i$.  The corresponding monopole operator has non-zero weight under the $G_{\ab}^{\vee}$-action, namely $k'v^{(i)}_j/d_j$ for $j\in I_i$, so it transforms non-trivially under a one-dimensional subtorus $F\subset G_{\ab}^{\vee}$.  

    By \cref{def:regularity-coefficients2}, the monopole operator corresponding to $\nu$ is just $e^{\nu}$, so $r^{\nu}r^{-\nu}=1$.  In particular, this operator is nowhere vanishing, and therefore every $F$-orbit has dimension at least $1$.  Since the action of $F$ is Hamiltonian, symplectic leaves are unions of $F$-orbits, so there are no zero-dimensional leaves.

    The second sentence is just the fact that if $\Bv^{(i)}$ has connected support, then the intersection of all weight hyperplanes is a 1-dimensional space, so this is the unique one-dimensional flat.  
\end{proof}
Since there is only one way for $d$ strictly positive integers to add up to $d$, we have that $L'$ must be a product of $d$ two-dimensional leaves.
That is:
\begin{Corollary}\label{cor:quiver-symmetrizer-leaves}
    There is a leaf in $\M(\GL_{\Bv}, \bfN_{\Bv} )$ corresponding to the decomposition $\Bv=\Bv^{(1)}+\cdots +\Bv^{(d)}$ if and only if $\M(\GL_{\Bv^{(i)}}, \bfN_{\Bv^{(i)}} )$ has a two-dimensional leaf for each $i$.  This leaf must lie over the flat given by the product of the unique one-dimensional flats in each factor.
\end{Corollary}
As in the case without symmetrizers, we can check that $\M(\GL_{\Bv^{(i)}}, \bfN_{\Bv^{(i)}} )$ has a unique two-dimensional leaf by showing that all monopole operators except those arising from the center in the proof of \cref{lem:no-zero-dim-leaves} have degree at least $2$.  In the symmetrizer-free case, this is equivalent to $\Bv^{(i)}\in \Sigma_0$.  

In the case with symmetrizers, there does not seem to be a similarly clean combinatorial criterion for the good condition, let alone the existence of a two-dimensional leaf.  For completeness, we include a discussion of (weak) conicity.

The degree function for monopole operators in a quiver theory with symmetrizers is a piecewise-linear function \cite[(2.3)]{nakajimaCoulombBranches2021}
\[\wp(\lambda) = -2\sum_{i \in I} \sum_{a\neq b} |\lambda_i^a-\lambda_i^b|  + \sum_{\{i,j\} : c_{ij} < 0} \sum_{a=1}^{v_i} \sum_{b=1}^{v_j} \bigl|c_{ij}\,\lambda^b_j - c_{ji}\,\lambda^a_i\bigr|\]
on the space of dominant cocharacters $\lambda = (\lambda^1_i \geq \cdots \geq \lambda^{v_i}_i)_{i \in I}$.

When trying to determine whether a quiver theory with symmetrizers is (weakly) conical, the most natural monopole operators to consider are those that lie on rays that are faces of the chambers where the degree function is linear;  if any monopole operator has negative degree, then $\wp$ must be negative on some such ray.  
Each such ray corresponds to a decomposition $\Bv=\Bm+\Bn$, where $m_i$ coordinates of the cocharacter are positive and $n_i$ coordinates are zero.  In the symmetric case, we would take all non-zero coordinates to be 1, but in general, we consider the cocharacter $\pi_{\Bm}$ whose components at vertex $i$ are $\pi/d_i$ with multiplicity $m_i$ and $0$ with multiplicity $n_i$, for $\pi\in\mathbb{Q}_{>0}$.  

We can write any dominant integral cocharacter $\lambda$ as a sum of these:  let $\pi$ be the maximum value of $d_i\lambda^a_i$ over all $i,a$, and let $\pi'$ be the second largest value of $d_i\lambda^a_i$; then $\lambda-(\pi-\pi')_{\Bm}$ is still dominant (note that we have not shown it is integral).  Repeating this process, we can write $\lambda$ as a sum of cocharacters of the form $\pi_{\Bm}$, and these will all lie on faces of the same chamber where $\wp$ is linear, so we can compute $\wp(\lambda)$ as the sum of the contributions from these cocharacters.

We can then compute the degree of the corresponding monopole operator as follows:
\begin{Lemma}\label{lem:face-cochar-degree}
Let $\Bv=\Bm+\Bn$ with $\Bm,\Bn\in\mathbb{N}^I$.  Then
\[\wp(\pi_{\Bm}) = -\pi\,\Bm^\top C D^{-1}\,\Bn.\]
\end{Lemma}
\begin{proof}
We substitute $\lambda$ into the formula for $\wp$.  For an edge $\{i,j\}$ with $c_{ij}<0$, the contribution from a pair $(a,b)$ is $\bigl|c_{ij}d_j^{-1}\pi - c_{ji}d_i^{-1}\pi\bigr|=0$ when both coordinates are positive (since $|c_{ij}|d_j^{-1}=|c_{ji}|d_i^{-1}$ by symmetrizability), and $-c_{ij}d_j^{-1}\pi= -c_{ji}d_i^{-1}\pi$ when exactly one is positive.  There are $m_i n_j+m_j n_i$ such mixed pairs, and the total contribution from the edge is $ -c_{ij}d_j^{-1}\pi(m_i n_j+m_j n_i).$ 
Summing over all edges and both orderings of $(i,j)$ gives
\[
\wp(\lambda) = \pi\sum_{i\neq j:\,c_{ij}<0}(-c_{ij})d_j^{-1}\,m_i n_j = -\pi\sum_{i,j}c_{ij}d_j^{-1}\,m_i n_j = -\pi\,\Bm^\top C D^{-1}\,\Bn.\qedhere
\]
\end{proof}

\begin{Proposition}\label{prop:nonneg-monopole-degree}
The Coulomb branch $\M(\GL_{\Bv},\bfN_{\Bv})$ is weakly conical if and only if $C^{\top}\Bv\leq 0$.  In this case, the set of fixed points under the attracting action has dimension $\geq 1$ if and only if $C^{\top}\Bv=0$, in which case the support of $\Bv$ is an affine Dynkin diagram and $\Bv$ is a positive multiple of the primitive null coroot $\delta^{\vee}$.
\end{Proposition}
\begin{proof}
Assume for simplicity that the support $\supp(\Bv)$ is equal to the quiver. Let $B=CD^{-1}$, so $B$ is a symmetric matrix with $B_{ij}\leq 0$ for $i\neq j$.   

  If we write $\alpha_i=m_i/v_i$ and $\beta_i=n_i/v_i=1-\alpha_i$ for $i$ in the support of $\Bv$, then the algebraic identity $\alpha_i\beta_j+\alpha_j\beta_i = \alpha_i\beta_i+\alpha_j\beta_j+(\alpha_i-\alpha_j)^2$,
  immediate from $\beta_k=1-\alpha_k$, gives
  \[
  m_i n_j + m_j n_i = v_iv_j\bigl(\alpha_i\beta_i+\alpha_j\beta_j+(\alpha_i-\alpha_j)^2\bigr).
  \]
  Substituting and collecting terms, we find that
  \begin{align*}
        -\Bm^{\top}B\Bn
    &= \sum_{i<j}(-B_{ij})v_iv_j(\alpha_i-\alpha_j)^2
      +\sum_i \frac{m_i n_i}{v_i}\!\Bigl(\sum_{j\neq i}(-B_{ij})v_j - B_{ii} v_i\Bigr)\\
     &=
    \sum_{i<j}(-B_{ij})\,v_i v_j\!\left(\frac{m_i}{v_i}-\frac{m_j}{v_j}\right)^{\!2}
    \;+\;
    \sum_{i}\frac{m_i n_i}{v_i}\,(-(B\Bv)_i)
    \;\geq\; 0.
  \end{align*}
    This shows that all monopole operators have non-negative degree if $C^{\top}\Bv\leq 0$, so we have weak conicity.

  The fixed point set is the spectrum of the degree 0 elements.  We always have a copy of $\C[x^{\pm 1}]$ generated by $\pi_{\Bv}$.  Assume now that some other monopole operator has degree 0.  By convexity, there must be such an operator of the form $\pi_{\Bm}$ for some $\Bm$.  Since both terms in the formula above are non-negative, they must both vanish.  The first term vanishes only if $m_i/v_i=m_j/v_j$ for all $i,j$ in the support of $\Bv$, so $\Bm=\alpha \Bv,\Bn= (1-\alpha)\Bv $ for some $\alpha\in (0,1)$.  Since all entries of $\Bm$ and $\Bn$ are positive, the second term can only vanish if $C^{\top}\Bv=0$. 
  
   This implies that the support of $\Bv$ is an affine Dynkin diagram, by the Trichotomy Theorem for generalized Cartan matrices \cite[Th. 4.3]{kacInfinitedimensionalLie1990}, and $\Bv$ is a positive multiple of the primitive null coroot $\delta^{\vee}$.  In this case, the monopole operator corresponding to $\delta^{\vee}$ has degree zero, so the Coulomb branch is not conical.
\end{proof}

This matches the ``fission'' and ``decay'' description of leaves of quiver gauge theories discussed by Bourget--Sperling--Zhong in the section on the algorithm in \cite{bourgetDecayFission2024}; ``decay'' refers to the case where some of the vectors $\Bv^{(i)}$ are unit vectors, so the corresponding factor of the Coulomb branch is just a point, while ``fission'' refers to a decomposition into multiple (an)isotropic vectors, so the Coulomb branch is written as a non-trivial product.  However, in that paper, some assumptions are made about exactly when a theory of the form $\M(\GL_{\Bv}, \bfN_{\Bv} )$ will have a two-dimensional leaf.
A preliminary version of these conditions was given in \cite{bourgetDecayFission2024}, but it was not adequate to determine the correct list of theories with two-dimensional leaves;  
 \cref{rem:ugly-no-affine} gives an example satisfying all these conditions but for which the corresponding theory $(\PGL_{\Bv}, \bfN_{\Bv})$ is ugly.
 In \cite[Def. 2]{bourgetClassifyingIsolated2025}, Bourget--Lamouret--Soys\"uren--Sperling give an updated set of candidates and use the term ``good'' in a different sense.  
 \begin{enumerate}[label=(\roman*)]
    \item 
 The first group of these candidates (Def. 2(i) in their terminology) are precisely the theories that follow the usual definition of goodness, with a small modification to account for symmetrizers.  If we have no symmetrizers, then this is just the condition that $(\PGL_{\Bv}, \bfN_{\Bv})$ is good, but in the symmetrizer case, there is no natural way to remove the kernel of $\GL_{\Bv}\to \PGL_{\Bv}$.  Instead, our replacement for goodness is the condition (*) that any monopole operator not of the form $\pi_{\Bv}$ has degree at least $2$.

When this is satisfied, there is a 2-dimensional leaf whose image is the Lie algebra of this kernel and on which all monopole operators of degree $>0$ vanish.
    \item The second group of candidates (Def. 2(ii) in their terminology) consists of the cases where the support of $\Bv$ is affine and $\Bv$ is a positive integer multiple of the primitive imaginary coroot, or where the support of $\Bv$ is a Jordan quiver and $v>1$. 
 \end{enumerate}
One can easily see that if $\M(\GL_{\Bv},\bfN_{\Bv})$ is weakly conical and has attracting set of dimension $1$ (in the no symmetrizer case, this is equivalent to $(\PGL_{\Bv}, \bfN_{\Bv})$ being conical), then it has a two-dimensional leaf if and only if there are no monopole operators of degree $1$, so condition (i) above is satisfied.  Thus \cref{prop:nonneg-monopole-degree} shows that:
\begin{Corollary}\label{cor:weakly-conical-two-dim-leaf}
    If $\M(\GL_{\Bv},\bfN_{\Bv})$ has a two-dimensional leaf and it is weakly conical, then $(\GL_{\Bv},\bfN_{\Bv})$ is of the form (i) or (ii) above.
\end{Corollary}

Thus, we can deduce \cite[Conj. 1]{bourgetClassifyingIsolated2025} from \cref{cor:quiver-symmetrizer-leaves} if the following conjecture holds.
\begin{Conjecture}\label{conj:two-dim-leaf}
    The Coulomb branch $\M(\GL_{\Bv},\bfN_{\Bv})$ has a $2$-dimensional leaf (and no zero-dimensional leaves) if and only if one of the conditions (i) or (ii) above is satisfied, and in each case, there is a unique such leaf.  In particular, if $\M(\GL_{\Bv},\bfN_{\Bv})$ is not weakly conical, then it has no two-dimensional leaves.
\end{Conjecture}
Existence and uniqueness are clear in case (i).  For case (ii), this is confirmed in the Jordan quiver case by \cref{ex:adjoint} and, as discussed in \cref{ex:affine-A} below, for affine type A in \cite[Th. 7.26]{nakajimaCherkisBow2017}.  In general, $\M(\GL_{\Bv},\bfN_{\Bv})$ should be a partial compactification of $\operatorname{Bun}^G(\mathbb{A}^2)$ for $G$ the corresponding simple group, with a unique two-dimensional leaf given by the small diagonal in the symmetric power of $\mathbb{A}^2$; this is confirmed in the ADE case, though not discussed in much detail, in \cite[Th. 3.22]{BFNplus}, building on earlier work in \cite{bravermanUhlenbeckSpaces2007,bravermanPursuingDouble2010}.

Thus, it seems to be mostly a matter of carefully tracing through the literature to confirm the converse of \cref{cor:weakly-conical-two-dim-leaf}.  However, outside the weakly conical case, it is hard to know whether to expect there will be other cases where there are two-dimensional leaves.

\subsubsection{Relation to Higgs branches}

For each such decomposition $\Bv^{(*)}$, there is a unique maximal leaf $L^{\Higgs}_{\Bv^{(*)}}\subset \M_{\Higgs}(G,\bfN)$ where the stabilizer of a generic point is the product $\prod \GL_{\Bv^{(i)}}$.  Identifying points of the variety $\M_{\Higgs}(G,\bfN)$ with the set of semi-simple representations of the preprojective algebra, the leaf $L^{\Higgs}_{\Bv^{(*)}}$ is open in the set of representations which decompose into simple summands with dimension vector $\Bv^{(*)}$. On this subset of leaves, we thus have the expected order-reversing bijection between the leaves of $\M(G,\bfN)$ and $\M_{\Higgs}(G,\bfN)$, sending $L_{\Bv^{(*)}}$ to $L^{\Higgs}_{\Bv^{(*)}}$.

By \cite[Th. 1.9]{bellamySymplecticResolutions2021}, the leaves of $\M_{\Higgs}(G,\bfN)$ are determined by representation-type data.  Concretely, we must fix the dimension vector of each simple summand that occurs, together with its multiplicity, that is, we record which pairs of summands with the same dimension vector are actually isomorphic.  We can coarsen this description by fixing only the dimension vectors of the simple summands that occur, without regard for whether pairs are isomorphic.  

Collecting together all appearances of the same dimension vector, we have $\Bv=\sum m_i \Bv^{[i]}$ with all $\Bv^{[i]}$ distinct.  The corresponding pieces are precisely the products of symmetric powers $\prod_i S^{m_i} \M_{\Higgs}(\PGL_{\Bv^{[i]}},\bfN_{\Bv^{[i]}})$.  The leaf $L^{\Higgs}_{\Bv^{(*)}}$ is the unique open leaf in this product of symmetric powers.  This product has lower-dimensional leaves if any of the $m_i$ are greater than $1$ for a dimension vector $\Bv^{[i]}$ that is isotropic or anisotropic, or equivalently is not a unit vector.  This is precisely when the leaves $L_{\Bv^{(*)}}$ and $L^{\Higgs}_{\Bv^{(*)}}$ do not give a full list of the leaves of $\M(G,\bfN)$ or $\M_{\Higgs}(G,\bfN)$.  

We can recover the case of framed (i.e. Nakajima) quiver varieties by
the ``Crawley-Boevey trick'', starting with a quiver $I_0$ and adding one additional node $\infty$ to the quiver with $w_i$ arrows from the vertex $i\in I$ to $\infty$, and taking $v_{\infty}=1$.  In this case, there can only be one summand of the decomposition with support on $\infty$, so non-trivial symmetric powers of the space associated to these representations never occur.  

In the case where $I_0$ is a simply-laced Dynkin diagram, all isotropic or anisotropic elements of $\Sigma_0$ for $I$ have non-trivial support on the additional node $\infty$, since there are none supported on $I_0$.  On the other hand, a vector $\Bv'$ with $v'_{\infty}=1$ lies in $\Sigma_0$ if and only if \[\langle\alpha_i^{\vee}, \sum_{j\in I_0} v_j'\alpha_j \rangle\geq w_i.\]  We can describe this condition simply in terms of a dominant weight $\lambda$ satisfying $\langle\lambda, \alpha_i^{\vee}\rangle=w_i$; it simply says that $\mu'=\lambda-\sum_{i\in I_0} v'_i\alpha_i$ is a dominant weight.

Thus if we let $\mu=\lambda-\sum_{i\in I_0} v_i\alpha_i$, then the elements of $\Sigma_0$ with support on $\infty$ that appear in the decomposition are in bijection with the dominant weights $\mu'$ satisfying $\lambda \geq \mu'\geq \mu$, and any such $\mu'$ can be completed to a decomposition of $\Bv^{(*)}$ by adding unit vectors in a unique way;  since the residual theories are always good in these cases, we have a unique leaf of $\M(G,\bfN)$ corresponding to each such $\mu'$.  
\begin{Corollary}
    For a framed quiver gauge theory of type ADE, the leaves of $\M(G,\bfN)$ are classified by dominant weights $\mu'$ satisfying $\lambda \geq \mu'\geq \mu$.
\end{Corollary}
If $\mu$ is dominant, then the same data classifies the leaves of $\M_{\Higgs}(G,\bfN)$, and the relation $\CH$ induces a bijection between the leaves of $\M(G,\bfN)$ and $\M_{\Higgs}(G,\bfN)$.  In particular, \cref{conj:special} holds in this case.

\begin{Example}\label{ex:affine-A}
    In the case where $I_0$ is an $n$-cycle, the picture is more complicated, but the answer must match \cite[Th. 7.26]{nakajimaCherkisBow2017}.  As in finite ADE type, we obtain elements of $\Sigma_0$ from dominant weights $\mu'$ satisfying $\lambda \geq \mu'\geq \mu$, but the dimension vector $\delta$, which is constant with value $1$ on $I_0$, also lies in $\Sigma_0$.  

    First, we note that if $\Bv=k\delta$, then \cite[Th. 7.26]{nakajimaCherkisBow2017} in the case $\underline{k}=\emptyset$ and $\lambda=\kappa$ establishes that $\M(\PGL_{k\delta},\bfN_{k\delta})$ is isomorphic to the centered Uhlenbeck space $\mathcal{U}^k_{n}$.   

    Thus, we can reconstruct the leaves of $\M(G,\bfN)$ by considering the decompositions $\Bv=\Bv^{(1)}+\cdots +\Bv^{(m)}$ where $\Bv^{(1)}$ corresponds to a dominant $\mu'$, and all other vectors appearing are unit vectors or multiples of $\delta$.  This gives the same set of strata as \cite[Th. 7.26]{nakajimaCherkisBow2017}, with the partition $\underline{k}$ given by the multiples of $\delta$ that appear (and $\mu'$ being the weight $\kappa$ given there).  \cref{cor:slice-to-leaf} agrees with the description of the slices in \cite[Th. 7.26]{nakajimaCherkisBow2017}, since each multiple of $\delta$ gives a factor of $\mathcal{U}^k_{n}$.

   As discussed below \cref{conj:two-dim-leaf}, we expect that more general affine Dynkin diagrams will have similar features, but we believe that this should be confirmed in more detail before claiming it is established.  
\end{Example}

\ifanindex
\IndexOfNotation
\fi

{\renewcommand{\markboth}[2]{}\printbibliography}
\end{document}